\documentclass{amsart}

\usepackage{graphicx} % Required for inserting images
\usepackage{amsmath,amsfonts,euscript,amscd,amsthm,amssymb,upref,graphics,color,verbatim, float, placeins,mathrsfs, mathtools,comment,tikz, multirow, makecell, tikz-cd}

\newcommand{\Ascr}{\mathscr{A}}

\newcommand{\Dscr}{\mathscr{D}}

\newcommand{\Lscr}{\mathscr{L}}
\newcommand{\Mscr}{\mathscr{M}}
\newcommand{\Nscr}{\mathscr{N}}

\newcommand{\Pscr}{\mathscr{P}}

\newcommand{\Tscr}{\mathscr{T}}

\newcommand{\Xscr}{\mathscr{X}}

\newcommand{\bbG}{\mathbb{G}}
\newcommand{\bbL}{\mathbb{L}}
\newcommand{\bbT}{\mathbb{T}}

\usepackage{ulem}  %  (for \xout)

\usepackage[margin=3cm]{geometry}
\usepackage[all]{xy}
\usepackage{hyperref}
\usepackage[shortlabels]{enumitem}

\newcommand{\red}{\color{red}}

\newcommand{\lnd}{\operatorname{{\rm LND}}}

\newcommand{\Der}{\operatorname{{\rm Der}}}
\newcommand{\hlnd}{\operatorname{{\rm HLND}}}

\newcommand{\klnd}{\operatorname{{\rm KLND}}}
\newcommand{\Integ}{\ensuremath{\mathbb{Z}}}
\newcommand{\Gscr}{\mathscr{G}}
\newcommand{\Nat}{\ensuremath{\mathbb{N}}}
\newcommand{\Rat}{\ensuremath{\mathbb{Q}}}
\newcommand{\Comp}{\ensuremath{\mathbb{C}}}

\newcommand{\aff}{\ensuremath{\mathbb{A}}}
\newcommand{\bk}{{\ensuremath{\rm \bf k}}}

\newcommand{\dom}{		\operatorname{\rm dom}}

\newcommand{\G}{\ensuremath{\mathbb{G}}}

\newcommand{\trdeg}{	\operatorname{{\rm trdeg}}}
\newcommand{\Frac}{		\operatorname{{\rm Frac}}}

\newcommand{\Lie}{		\operatorname{{\rm Lie}}}

\newcommand{\setspec}[2]{\big\{\,#1\, \mid \,#2\, \big\}}
\renewcommand{\emptyset}{\varnothing}
\newcommand{\isom}{\cong}

\newcommand{\Aut}{		\operatorname{{\rm Aut}}}
\newcommand{\rank}{		\operatorname{\rm rank}}

\newcommand{\Spec}{		\operatorname{{\rm Spec}}}

\newcommand{\Supp}{		\operatorname{{\rm Supp}}}

\newcommand{\LieLN}{\operatorname{{\rm LieLN}}}
\newcommand{\DJL}{\operatorname{{\rm DJL}}}
\newcommand{\KS}{\operatorname{{\rm KS}}}
\newcommand{\kerseq}{\operatorname{{\rm kerseq}}}

\swapnumbers
\newtheorem*{theoremA}{Theorem A}
\newtheorem*{theoremB}{Theorem B}
\newtheorem*{corollaryC}{Corollary C}
\newtheorem*{theoremD}{Theorem D}
\newtheorem{theorem}[subsection]{Theorem}

\newtheorem*{theorem*}{Theorem}
\newtheorem{proposition}[subsection]{Proposition}
\newtheorem*{proposition*}{Proposition}
\newtheorem{lemma}[subsection]{Lemma}
\newtheorem{lemmadef}[subsection]{Lemma and definition}
\newtheorem*{lemma*}{Lemma}
\newtheorem{corollary}[subsection]{Corollary}
\newtheorem*{corollary*}{Corollary}

\theoremstyle{definition}

\newtheorem{remark}[subsection]{Remark}
\newtheorem{definition}[subsection]{Definition}

\newtheorem{noname}[subsection]{}
\newtheorem*{noname*}{}

\newtheorem*{nothing*}{}

\newtheorem{example}[subsection]{Example}

\newtheorem{notation}[subsection]{Notation}

\newcommand{\Id}{		\operatorname{{\rm Id}}}
\newcommand{\Grad}{\rm{Grad}}

\newcommand{\Agoth}{\mathfrak{A}}

\newcommand{\Lgoth}{\mathfrak{L}}

\newcommand{\ggoth}{\mathfrak{g}}
\newcommand{\hgoth}{\mathfrak{h}}

\newcommand{\Aeul}{\EuScript{A}}
\newcommand{\Beul}{\EuScript{B}}

\newcommand{\Deul}{\EuScript{D}}

\newcommand{\Leul}{\EuScript{L}}

\newcommand{\dl}{\operatorname{{\rm dl}}}

\newcommand{\bt}{\mathbf{t}}
\newcommand{\bi}{\mathbf{i}}
\newcommand{\NU}{\operatorname{{\rm NU}}}
\newcommand{\NC}{\operatorname{{\rm NC}}}
\newcommand{\EXP}{\text{\sc exp}}
\newcommand{\LOG}{\text{\sc log}}
\newcommand{\Borel}{\operatorname{{\rm Borel}}}
\newcommand{\nes}[1]{{#1}^{(\text{\rm nes)}}}
\newcommand{\clo}[1]{{#1}^{(\text{\rm cl)}}}
\newcommand{\urad}[1]{{#1}_{\text{\rm u}}}
\newcommand{\bbM}{\mathbb{M}}

\newcommand{\Aff}{\mathbb{A}}

\title[Borel subgroups of the automorphism group of an affine variety]{Borel subgroups of the automorphism group\\ of an affine variety}
\author{Michael Chitayat, Daniel Daigle and Andriy Regeta}

\address{Department of Mathematics and Statistics, University of Ottawa, Ottawa, ON, Canada, K1N 6N5.}
\email{ddaigle@uottawa.ca}

\address{Dipartimento di Matematica ``Tullio Levi-Civita'', 
Universit\`a di Padova, Via Trieste 63, I-35121 Padova}
\email{michael.chitayat@unipd.it, mikechitayat@gmail.com; andriyregeta@gmail.com}



\date{\today}

\begin{document}

\begin{abstract}
    We study Borel subgroups of the automorphism group of an affine variety $X$ and their associated Lie algebras. We prove a bijective correspondence between the set of Borel subgroups of $\Aut(X)$ and the set of maximal nested unipotent subgroups of $\Aut(X)$. 
    We show that if $X$ is factorial of dimension at most 3, then the Lie algebra of every maximal nested unipotent subgroup has the special property of being what we call \textit{integral}; we also show that this fails for $X = \Aff^n$ when $n \geq 4$.
    As an application of our results, we show that Nagata's automorphism is contained in a unique Borel subgroup. 
\end{abstract}

\maketitle

\newcommand{\I}{\text{\rm I}}
\newcommand{\II}{\text{\rm II}}
\newcommand{\plinth}{\operatorname{{\rm pl}}}

\newcommand{\bbD}{\mathbb{D}}
\newcommand{\dgoth}{\mathfrak{d}}
\newcommand{\Hom}{\operatorname{{\rm Hom}}}
\newcommand{\Span}{\operatorname{{\rm Span}}}
\newcommand{\kk}[1]{\bk^{[{#1}]}}
\newcommand{\Norm}{\operatorname{{N}}}
\newcommand{\ad}{	\operatorname{\text{\rm ad}}}

\newcommand{\UnitalAlg}[1]{\langle #1 \rangle_{\text{\rm u-alg}}}
\newcommand{\AssAlg}[1]{\langle #1 \rangle_{\text{\rm alg}}}
\newcommand{\LieAlg}[1]{\langle #1 \rangle_{\text{\rm Lie}}}

\newcommand{\lfd}{\operatorname{{\rm LFD}}}
\newcommand{\LieWLF}
{\operatorname{{\rm LieWLF}}}
\newcommand{\PGL}{\operatorname{{\rm PGL}}}
\newcommand{\SolLieWLF}{\operatorname{{\rm SolLieWLF}}}

\newcommand{\rtrdeg}{	\operatorname{{\rm rtrdeg}}}
\newcommand{\grank}{	\operatorname{\rm grank}}

\newcommand{\rien}[1]{}

\newcommand{\ba}{{\bf a}}
\newcommand{\bb}{{\bf b}}
\newcommand{\bx}{{\bf x}}

\setlength{\marginparwidth}{2.5cm}  % or smaller/larger as needed
\newcommand{\andriy}[1]{{\color{blue}*}\marginpar{\tiny \color{blue} AR: #1}}
\newcommand{\daniel}[1]{{\color{cyan}*}\marginpar{\tiny \color{cyan} DD: #1}}
\newcommand{\michael}[1]{{\color{red}*}\marginpar{\tiny \color{red} MC: #1}}

\section{Introduction}

Let $X$ be an irreducible affine algebraic variety over an algebraically closed field $\bk$ of characteristic zero.
This article studies the automorphism group $\Aut(X)$ of $X$, with particular attention to two types of subgroups:
Borel subgroups and maximal nested unipotent subgroups.

Recall that $\Aut(X)$ has the structure of an affine ind-group \cite[Theorem 5.1.1]{FurterKraft}.
A {\it Borel subgroup\/} of $\Aut(X)$ is a maximal element
of the set of connected solvable subgroups of $\Aut(X)$, where connectedness refers to the topology of $\Aut(X)$ as an ind-group.
A subgroup of $\Aut(X)$ is {\it nested\/} if it is the union of an increasing sequence
$G_0 \subseteq G_1 \subseteq G_2 \subseteq \cdots$ of algebraic subgroups of $\Aut(X)$.
An element $g$ of $\Aut(X)$ is {\it unipotent\/} if
$g = \Id_X$ or $g$ belongs to some algebraic subgroup of $\Aut(X)$ isomorphic to $\G_a$.
A subgroup of $\Aut(X)$ is {\it unipotent\/} if all its elements are unipotent.
(Refer to Section \ref{sectionUniqueBorel} for details).

Consider the important special case where $X = \Aff^n$.
It is well known that $\Aut(\Aff^1) \isom \G_m \ltimes \G_a$ is a solvable linear algebraic group.
Although $\Aut(\Aff^2)$ is infinite-dimensional and hence not an algebraic group,
it is still relatively well understood, thanks to \cite{Jung} and \cite{vdKulk}.
By contrast, $\Aut(\Aff^n)$ remains poorly understood for $n \ge 3$.
It is well known that all Borel subgroups of a linear algebraic group are conjugate.
The same holds for $\Aut(\Aff^1)$ and $\Aut(\Aff^2)$: the group $\Aut(\Aff^1)$ is its own (unique) Borel subgroup,
and all Borel subgroups of $\Aut(\Aff^2)$ are conjugate by \cite{BEE16} (case $\bk = \Comp$)
and \cite[Proposition 6.24]{solvableAutGroups} (general case).
On the other hand,  $\Aut(\Aff_\Comp^n)$ contains non-conjugate Borel subgroups for every $n \ge 3$ \cite[Corollary 1.5]{regeta2025group}. Let us now describe our results.

\bigskip

\noindent{\bf\boldmath Unipotent subgroups of $\Aut(X)$.}
In \cite{regeta2025maximal}, the authors describe maximal commutative unipotent subgroups of $\Aut(X)$
as a first step towards a systematic study of maximal unipotent and maximal solvable subgroups of $\Aut(X)$.
Subsequently, Perepechko showed in \cite{perepechko2025}
that maximal nested unipotent subgroups of $\Aut(X)$ correspond bijectively to the maximal elements of a certain class of Lie algebras of locally nilpotent derivations,
which he called {\it de Jonqui\`eres-like algebras.}\footnote{The term ``de Jonqui\`eres-like algebra" was used in an earlier version of \cite{perepechko2025} but was
not retained in the current version.}
This motivates the study of these Lie algebras, whose definition we now recall.

Let $B = \bk[X]$ be the coordinate algebra of $X$ and let $\Der_\bk(B)$ denote the Lie algebra of all $\bk$-derivations of $B$.
By an {\it admissible sequence}, we mean a finite sequence $(D_1,\dots, D_m)$ of locally nilpotent derivations of $B$ such that
$D_1, \dots,D_m$ pairwise commute and are  linearly independent over $B$.
Given an admissible sequence $(D_1,\dots, D_m)$, we define
\begin{itemize}

\item $A_i = \cap_{j = i}^m \ker D_j$ for each $i \in \{1,\dots,m\}$,
and we say that $(A_1, \dots, A_m)$ is the \textit{kernel sequence\/} of $(D_1, \dots, D_m)$;

\item $L(D_1, \dots, D_m) = \sum_{i = 1}^m A_i D_i$\,;

\item $\Lscr(D_1, \dots, D_m) = \setspec{\delta \in \Der_\bk(B)}{s \delta \in L(D_1, \dots, D_m) \text{ for some $s \in A_1 \setminus \{0\}$}}$.

\end{itemize}
One can see that $L(D_1, \dots, D_m) \subseteq \Lscr(D_1, \dots, D_m)$ are Lie-subalgebras of $\Der_\bk(B)$
and that all elements of $\Lscr(D_1, \dots, D_m)$ are locally nilpotent.
A Lie subalgebra $\Lscr$ of $\Der_\bk(B)$ is called \textit{de Jonqui\`eres-like} if $\Lscr = \Lscr(D_1, \dots, D_m)$ for some admissible sequence $(D_1, \dots, D_m)$.
We set: 
\begin{align*}
    \DJL(B) &= \text{ the set of de Jonqui\`eres-like Lie subalgebras of $\Der_\bk(B)$} \\
&= \setspec{ \Lscr(D_1, \dots, D_m) }{ \text{$(D_1, \dots, D_m)$ is an admissible sequence} } .
\end{align*}  

Implicitly contained in \cite{perepechko2025} is the fact that
\begin{equation} \label {7631DD5C4D884DC7} 
\xymatrix@1@C=16mm{
\fbox{\footnotesize maximal nested unipotent subgroups of $\Aut(X)$} \, 
\, \ar @<-.4ex>[r]_-{ \Lie } &  \ar @<-.4ex>[l]_-{ \EXP }
\, \fbox{\footnotesize maximal elements of $\DJL(B)$} }
\end{equation}
are mutually inverse bijections,
where $\EXP$ is essentially the exponentiation of locally nilpotent derivations and $\Lie$ is the operation
of forming the Lie algebra of a nested unipotent subgroup of $\Aut(X)$.
In view of these bijections, it is natural to ask whether one can describe the structure of the poset $( \DJL(B),\subseteq )$.
Section \ref{sectionDJL} is devoted to that question.
Corollary \ref{ujd3876g20oigjh2vy68} states that,
given admissible sequences $(D_1,\dots,D_m)$ and $(E_1,\dots,E_n)$,
$$
\Lscr(D_1, \dots,D_m) = \Lscr(E_1, \dots, E_n) \ \iff \ \begin{minipage}[t]{.5\textwidth}
$(D_1, \dots, D_m)$ and $(E_1, \dots, E_n)$ have the same\\ kernel sequence.
\end{minipage}
$$
This allows us to replace the study of the algebras $\Lscr \in \DJL(B)$ by that of their kernel sequences,
which turns out to be a fruitful strategy: several of the significant results of this paper are proved by applying that philosophy.
In particular, we obtain a practical criterion for deciding whether a given element $\Lscr \in \DJL(B)$ is maximal,
and a method for finding all immediate successors of $\Lscr$ if it is not maximal (Proposition \ref{E024C289-4F11-48E4-831C-C94E59C34ACE}).

Section \ref{SecintegralBorels} continues the study of $\DJL(B)$, but from a different angle. 
An element $\Lscr$ of $\DJL(B)$ is said to be {\it integral\/} if there exists an admissible sequence $(D_1,\dots, D_m)$ such that $\Lscr = L(D_1, \dots, D_m)$.
To see the group-theoretic significance of this notion, consider a maximal nested unipotent subgroup $U$ of $\Aut(X)$.
By \eqref{7631DD5C4D884DC7}, $\Lie(U)$ is a maximal element of $\DJL(B)$;
if $\Lie(U)$ is integral then $U$ admits a decomposition\footnote{Equality \eqref{0F7089CC7FF14DF9} follows from the fact that 
if $L = L_1 \ltimes L_2$ are Lie algebras consisting of locally nilpotent derivations then $\EXP(L) = \EXP(L_1) \ltimes \EXP(L_2)$.
The proof uses the Baker-Campbell-Hausdorff Theorem and related results.}
\begin{equation} \label {0F7089CC7FF14DF9} 
U = U_1 \ltimes ( U_2 \ltimes ( \cdots \ltimes (U_{m-1} \ltimes U_m) \cdots )) , \quad \text{with $U_i = \EXP(A_i D_i)$ for each $i=1,\dots,m$.} 
\end{equation}
Here,
$(D_1,\dots,D_m)$ is any admissible sequence satisfying $\Lie(U) = L(D_1,\dots,D_m)$, and $(A_1,\dots,A_m)$ is its kernel sequence.
Also, one can see that $m$ is the derived length of the solvable group $U$.
In particular, we note:
\begin{quote}
{\it If all maximal elements of $\DJL(B)$ are integral, then all maximal nested unipotent subgroups of $\Aut(X)$ can
be explicitly described as in \eqref{0F7089CC7FF14DF9}.}
\end{quote}
It is therefore natural to ask which affine $\bk$-domains $B$ have the property that all maximal elements of $\DJL(B)$ are integral.
Combining Propositions \ref{dim3Straight} and \ref{6BC46AAB-2493-4034-AE60-7FCC9264ECB8} with Example \ref{nonIntegralExample} gives:
\begin{theoremA}
\mbox{\ }
\begin{enumerate}[\rm(a)]

\item If $B$ is a factorial affine $\bk$-domain with $\dim B \le 3$ then all elements of $\DJL(B)$ are integral.

\item For each $n\ge4$, there exist non-integral maximal elements of $\DJL( \bk^{[n]})$.

\item There exist $3$-dimensional smooth affine varieties $X$ such that  $\DJL( \bk[X] )$ contains non-integral maximal elements.

\end{enumerate}
\end{theoremA}

\bigskip

\noindent{\bf\boldmath Borel subgroups of $\Aut(X)$.}
Our results on Borel subgroups rely on the following special case of \cite[Corollary B and Theorem D]{solvableAutGroups},
where $\Aut_{\text{\rm u}}(X)$ denotes the set of unipotent automorphisms of $X$:

\begin{theorem*}
Let $\Beul$ be a Borel subgroup of $\Aut(X)$ and let $\urad{\Beul} = \Beul \cap \Aut_{\text{\rm u}}(X)$.
Then $\urad{\Beul}$ is a nested unipotent subgroup of $\Aut(X)$ and there exists an algebraic torus $T$ in $\Aut(X)$ such that $\Beul = T \ltimes \urad{\Beul}$. 
\end{theorem*}

Given a Borel subgroup $\Beul$ of $\Aut(X)$, it is not a priori clear whether $\urad{\Beul}$ is maximal among nested unipotent subgroups.
Conversely, while every maximal nested unipotent subgroup of $\Aut(X)$ is connected and solvable,
and hence is contained in a Borel subgroup \cite[Proposition 6.24]{solvableAutGroups},
it is not clear whether such a Borel subgroup is unique.
The following simplified version of Theorem \ref{0CA62AB7-E680-4831-972F-E730B069503F} addresses these two issues explicitly.

\begin{theoremB}\label{uniquenessBorel}
The map $\Beul \mapsto \urad{\Beul}$ defines a bijection between the Borel subgroups of $\Aut(X)$ and the maximal nested unipotent subgroups of $\Aut(X)$.
\end{theoremB}

As a charming application of Theorem B and of results from Section \ref{sectionDJL} on the structure of \mbox{$(\DJL(B),\subseteq)$},
we mention the following fact, obtained in paragraph \ref{6A24CF4A9C6743B6}:

\begin{corollaryC}\label{IntroNagata}
    There exists a unique Borel subgroup (resp.\ a unique maximal nested unipotent subgroup) of $\Aut(\Aff^3)$ that contains Nagata's automorphism. 
\end{corollaryC}

The significance of Theorem B is better understood when we view it as a corollary of the following result,
which is essentially Theorem \ref{0CA62AB7-E680-4831-972F-E730B069503F}.

\begin{theoremD}
Let $B = \bk[X]$ denote the coordinate algebra of $X$ and $\lnd(B)$ the set of locally nilpotent derivations of $B$.
Consider the following diagram of sets and maps:
\begin{equation} \label {234219016D354992} 
\hspace{-8mm}\raisebox{5mm}{\scalebox{.95}{\xymatrix@R=7mm@C=24mm{
{\fbox{\begin{minipage}{25mm} \centering\footnotesize Borel subgroups\\ of $\Aut(X)$ \end{minipage}} }
\ar[d]_-{\Lie} \ar[r]^-{ \Beul \mapsto \urad{\Beul} := \Beul \cap \Aut_{\text{\rm u}}(X) }  \ar[r]_-{ \rho_X } 
&  {\fbox{\begin{minipage}{36mm} \centering\footnotesize maximal nested unipotent\\ subgroups of $\Aut(X)$ \end{minipage}} }  \ar @<-.4ex> [d]_-{\Lie}   \\
{\fbox{\begin{minipage}{43mm} \centering\footnotesize maximal solvable Lie algebras\\ of locally finite derivations of $B$ \end{minipage}} }
\ar[r]^-{\Lscr \mapsto \Lscr \cap \lnd(B)} \ar[r]_-{\rho_B}
 &  {\fbox{\begin{minipage}{47mm} \centering\footnotesize maximal elements of $\DJL(B)$ \end{minipage}} }
\ar @<-.4ex>[u]_-{ \EXP }
}}}
\end{equation}
\begin{enumerate}[\rm(a)]

\item If $f$ denotes any one of the five arrows in diagram \eqref{234219016D354992} then
\begin{enumerate}[\rm(i)]
\item $f$ is a well-defined bijection;
\item two elements of $\dom(f)$ are conjugate if and only if their images by $f$ are conjugate.
\end{enumerate}

\item Diagram \eqref{234219016D354992} is commutative.

\end{enumerate}
\end{theoremD}

Let us briefly indicate how $\rho_X$ is shown to be bijective in Theorem D.
By \eqref{7631DD5C4D884DC7}, we already know that the two vertical arrows on the right side of the diagram are mutually inverse bijections.
By Lemma \ref{21AFB0A0-0700-4E7F-8AE2-D3557F9656CC} and paragraph \ref{4B7B6F7E-43D3-4719-91CA-E6301252556A},
the set in the lower-left corner of the diagram equals the set of ``locally integrable Borel subalgebras of $\Lie(\Aut X)$'' in the terminology of \cite{Arz_Zaid_2026};
so Theorem 3.4 of that paper implies that the vertical arrow ``$\Lie$'' on the left side of the diagram is well-defined and bijective.
If we now take for granted the (non-trivial) facts that $\rho_X$ and $\rho_B$ are well-defined and that the diagram commutes,
we see that {\it $\rho_X$ is bijective if and only if $\rho_B$ is bijective.}
The main part of the proof then consists of showing that $\rho_B$ is bijective.
This bijectivity is asserted in Proposition \ref{CD178309-4EC6-4386-9C1E-34F625BEBC29},
whose proof relies on results of Sections \ref{sectionLocallyNilpotentandLocallyFinite} and \ref{Liealggradsec}.
We note that the results of \cite{solvableAutGroups} are used twice in the proof of Theorem D: they are used in the proof of \cite[Theorem 3.4]{Arz_Zaid_2026},
and the above-mentioned special case of \cite[Corollary B and Theorem D]{solvableAutGroups} is needed to verify the hypothesis of Proposition \ref{CD178309-4EC6-4386-9C1E-34F625BEBC29}.

Theorem D shows that classifying the Borel subgroups of $\Aut(X)$ up to conjugacy is equivalent to classifying
the maximal elements of $\DJL(B)$ up to conjugacy. In a sequel to this paper, we use this equivalence together with the theory of locally nilpotent derivations to obtain explicit results about Borel subgroups of $\Aut(\Aff^3)$ and their conjugacy classes.

\section{Preliminaries}\label{SectionPreliminaries}

\begin{noname}   \label {9823hf276v38ghfvcd9387tg}
We use  ``$\subset$'' for strict inclusion and  ``$\subseteq$'' for general inclusion.
We follow the convention that $0 \in \Nat$.
All rings are assumed to be associative, commutative and unital.
If $A$ is a ring then $A^*$ denotes the set of units of $A$.
If $B$ is an algebra over a ring $A$, the notation $B = A^{[n]}$ (where $n \in \Nat$) means that $B$ is isomorphic as an $A$-algebra to a polynomial
ring in $n$ variables over $A$. If $L/K$ is a field extension then $L = K^{(n)}$ means that $L$ is a purely transcendental extension of $K$
of transcendence degree $n$.

The word ``domain'' means ``integral domain''.
We write $\Frac A$ for the field of fractions of a domain $A$.
If $\bk$ is a field then a {\it $\bk$-domain\/} is a $\bk$-algebra that is also a domain,
and an {\it affine $\bk$-domain\/} is a finitely generated $\bk$-algebra that is also a domain.
If $A \subseteq B$ are domains then the transcendence degree of $B$ over $A$ is denoted $\trdeg_A(B)$ or $\trdeg(B:A)$.
If $\bk$ is a field and $B$ is a $\bk$-domain then the {\it rational transcendence degree of $B$ over $\bk$}, denoted $\rtrdeg_\bk(B)$,
is the supremum of the set of $n \in \Nat$ for which there exists a field $K$ satisfying
$\bk \subseteq K \subseteq \Frac(B)$ and $\Frac(B) = K^{(n)}$.  Clearly, $0 \le \rtrdeg_\bk(B) \le \trdeg_\bk(B)$.
\end{noname}

\begin{noname} \label {04E41BA5-443D-48FF-9B44-9279A9BECCA5}
Let $A \subseteq B$ be domains. We say that $A$ is {\it factorially closed in $B$} if
the conditions $x,y \in B \setminus \{0\}$ and $xy \in A$ imply that $x,y \in A$.
Note that if $A$ is factorially closed in $B$ then $B^*=A^*$, $B \cap \Frac(A) = A$, every field included in $B$ is included in $A$,
and if $B$ is a UFD then so is $A$.
An arbitrary intersection of factorially closed subrings of $B$ is factorially closed in $B$.
\end{noname}

\begin{noname} \label {7E1F8E6C-B326-4B86-A442-CA26364C2C1C}
Let $A$ be a subring of a ring $B$. We write $\Der(B)$ for the set of all derivations $D : B \to B$,
and $\Der_A(B)$  for the set of $A$-derivations $D : B \to B$.
Note that $\Der_A(B)$ is a $B$-module, and also a Lie algebra over the ring $A$, with Lie bracket $[D,E] = DE-ED$ for $D,E \in \Der_A(B)$.

If $\Delta \subseteq \Der(B)$ then we define $\ker(\Delta) = \bigcap_{D \in \Delta} \ker(D)$; if $\Delta=\emptyset$ then $\ker(\Delta) = B$.

A derivation $D : B \to B$ is {\it irreducible\/} if for every choice of $b \in B$ and $\delta \in \Der(B)$ satisfying $D = b\delta$, we have $b \in B^*$. A derivation $D : B \to B$ is {\it locally nilpotent} if, for each $b \in B$, there exists $n>0$ such that $D^n(b)=0$.
We write $\lnd(B)$ for the set of locally nilpotent derivations $D : B \to B$, and $\klnd(B) = \setspec{ \ker(D) }{ D \in \lnd(B) \text{ and } D \neq 0 }$.
If $A \subseteq B$, we define $\lnd_A(B) = \lnd(B) \cap \Der_A(B)$. 
It is well known that if $B$ is a domain of characteristic zero then the following hold. 
\begin{enumerate}[\rm(i)]

\item If $\Delta \subseteq \lnd(B)$, then $\ker \Delta$ is factorially closed in $B$;
consequently, every field contained in $B$ is also contained in $\ker \Delta$. 

\item If $D : B \to B$ is a derivation and $S$ is a multiplicative set of $B$, then the derivation
$S^{-1}D : S^{-1}B \to S^{-1}B$ is defined by the rule $(S^{-1}D)(b/s) = ( D(b)s - b D(s) )/ s^2$, where $b \in B$ and $s \in S$.
If $D$ is locally nilpotent and $S \subseteq \ker D$ then $S^{-1}D$ is locally nilpotent and $\ker(S^{-1}D) = S^{-1}(\ker D)$.

\item If $D \in \Der(B) \setminus \{0\}$ and $b \in B \setminus \{0\}$, then
    \begin{equation*}
    \text{$bD$ is locally nilpotent $\iff$ $D$ is locally nilpotent and $b \in \ker D$}.
    \end{equation*}

\end{enumerate}
\end{noname}

\begin{lemma} \label{F45F968B-ECF9-45AC-9626-A51DA798D02B}
Let $B$ be a domain of characteristic zero and $E \in \lnd(B)$.
For each $D \in \Der(B) \setminus \{0\}$ and $b \in B \setminus \{0\}$ we have
$$
[bD,E]=0 \ \iff \ b \in \ker(E) \text{ and } [D,E]=0 .
$$
\end{lemma}

\begin{proof}
There exists $x \in B$ such that
\begin{equation} \label {991262CB-2E3C-4CD0-A389-40F21868AB73}
D(x) \neq 0 \quad \text{and} \quad E(x) \in \ker(D) .
\end{equation}
Indeed, if  $\ker(E) \nsubseteq \ker(D)$ then any $x \in \ker(E) \setminus \ker(D)$ satisfies \eqref{991262CB-2E3C-4CD0-A389-40F21868AB73},
and if  $\ker(E) \subseteq \ker(D)$ then the fact that $D \neq 0$ implies that $\ker(D) = \ker(E)$,
so any local slice\footnote{A {\it local slice\/} of $E$ is an $x \in B$ such that $E(x)\neq0$ and $E^2(x)=0$. Since $E$ is locally nilpotent and nonzero, it has a local slice.}
$x$ of $E$ satisfies \eqref{991262CB-2E3C-4CD0-A389-40F21868AB73} (because $x \notin \ker(E) = \ker(D)$ and $E(x) \in \ker(E) = \ker(D)$).

Suppose that $[bD,E]=0$ and choose any $x$ satisfying \eqref{991262CB-2E3C-4CD0-A389-40F21868AB73}. Then
$E( b D(x) ) = (E \circ bD)(x) = (bD \circ E)(x) = bD( E(x) )=0$,
so $0 \neq b D(x) \in \ker(E)$. As $\ker(E)$ is factorially closed in $B$ by \ref{7E1F8E6C-B326-4B86-A442-CA26364C2C1C}(i), we have $b \in \ker(E)$.
It then follows that for each $y \in B$ we have 
$b E(D(y)) = E( b D(y) ) = (E \circ bD)(y) = (bD \circ E)(y) = bD(E(y))$,
so $E(D(y)) = D(E(y))$, showing that $[D,E]=0$.
This proves $(\Rightarrow)$, and the converse is easily verified.
\end{proof}

\begin{corollary} \label {FE9079CB-02CC-4FD7-8149-9E402ABC4851}
Let $B$ be a domain of characteristic zero.
If $D,E \in \Der(B) \setminus \{0\}$ and $a,b \in B \setminus \{0\}$ then the following are equivalent:
\begin{enumerate}[\rm(i)]

\item $aD, bE \in \lnd(B)$ and $(aD) \circ (bE) = (bE) \circ (aD)$

\item $D,E \in \lnd(B)$, $D\circ E = E \circ D$ and $a,b \in \ker(D) \cap \ker(E)$.

\end{enumerate}

\end{corollary}

\begin{noname}
If $B$ is an algebra over a field $\bk$,
a $\bk$-derivation $D : B \to B$ is {\it locally finite\/} if, for each $b \in B$,
there exists a finite dimensional $\bk$-vector subspace $U$ of $B$ such that $b \in U$ and $D(U) \subseteq U$.
We write $\lfd_\bk(B)$ for the set of locally finite $\bk$-derivations $D : B \to B$.
\end{noname}

The sets $\lnd(B)$ and $\lfd_\bk(B)$ are in general neither closed under addition nor under the Lie bracket.

\section{Locally nilpotent sets and (weakly) locally finite sets}\label{sectionLocallyNilpotentandLocallyFinite}

\newcommand{\id}{\text{\rm id}}

If $V$ is a vector space over a field $\bk$,
we write $\Leul(V)$ for the associative $\bk$-algebra whose elements are the linear maps $F : V \to V$ and whose multiplication is the composition of maps.
Note that $\Leul(V)$ is also a Lie algebra over $\bk$, with $[F,G] = F \circ G - G \circ F$ for all $F,G \in \Leul(V)$.

\begin{definition} \label {582852DA-9D57-47AC-B351-A3A2CC446141}
Let $V$ be a vector space over a field $\bk$ and let $\Delta$ be a subset of $\Leul(V)$.
\begin{enumerate}[\rm(a)]

\item The set $\Delta$ is {\it locally nilpotent\/} if for each $x \in V$ and every infinite sequence $(F_1,F_2,\dots)$ of elements of $\Delta$,
there exists $n>0$ such that $(F_n F_{n-1} \cdots F_1)(x)=0$.

\item The set $\Delta$ is {\it locally finite\/} if for each $x \in V$ there exists a finite dimensional subspace $U$ of $V$
such that $x \in U$ and $F(U) \subseteq U$ for all $F \in \Delta$.

\item The set $\Delta$ is {\it weakly locally finite\/} if every finite subset of $\Delta$ is locally finite.

\end{enumerate}
\end{definition}

It is clear that if $\Delta$ is locally nilpotent (resp.\ locally finite, weakly locally finite) then so is every subset of $\Delta$.

\begin{remark} \label {03442952-E213-4085-B07A-FABC7408E85A}
Let $V$ be a vector space over a field $\bk$ and $\Delta$ a subset of $\Leul(V)$.
We write $\AssAlg{\Delta}$ for the subalgebra of $\Leul(V)$ generated by $\Delta$, i.e.,
the intersection of all $\bk$-subspaces $A$ of $\Leul(V)$ that are closed under composition of maps and satisfy $\Delta \subseteq A$.
We write $\LieAlg{\Delta}$ for the Lie subalgebra of $\Leul(V)$ generated by $\Delta$.
It is easy to see that if $\Delta$ is locally nilpotent (resp.\ locally finite, weakly locally finite)
then so is each one of $\Span_\bk(\Delta) \subseteq \LieAlg{\Delta} \subseteq \AssAlg{\Delta}$.
\end{remark}

If $B$ is an algebra over a field $\bk$ and $\Delta$ is a subset of $\Der_\bk(B)$,
we say that $\Delta$ is a {\it locally nilpotent\/} (resp.\ {\it locally finite, weakly locally finite\/}) subset of $\Der_\bk(B)$ if it is so
when viewed as a set of $\bk$-linear endomorphisms of $B$.
Clearly, if $\Delta$ is a locally nilpotent subset of $\Der_\bk(B)$ then $\Delta \subseteq \lnd(B)$,
and if $\Delta$ is a weakly locally finite subset of $\Der_\bk(B)$ then $\Delta \subseteq \lfd_\bk(B)$.
For background on locally nilpotent sets of derivations, see \cite{Skutin_LND_2021} and \cite{Daigle_LNSets_2020};
for (weakly) locally finite sets of derivations, see \cite{ChDaRe2026}.

\begin{lemma}\cite[Lemma 2]{Skutin_LND_2021}\label {SkutinLemma2}
Let $V$ be a vector space over a field $\bk$ and $\Delta$ a locally nilpotent subset of $\Leul(V)$.
If $F \in \Leul(V)$ is locally nilpotent and $[F,\Delta] \subseteq \Delta$ then $\{F\} \cup \Delta$ is locally nilpotent.
\end{lemma}

\begin{noname}  \label {ACF4309F-5E77-4FB8-BDC4-45B3485FFDB6}
Given a Lie algebra $\Lscr$, let $\Lscr'$ temporarily denote the Lie subalgebra of $\Lscr$ generated by the set $\setspec{ [x,y] }{ x,y \in \Lscr }$,
and define $\Lscr^{(0)} = \Lscr$ and $\Lscr^{(n+1)} = \big( \Lscr^{(n)} \big)'$ for $n\ge0$.
(The prime notation $\Lscr'$ will not be used after this point.)
Recall that if there exists $d \in \Nat$ such that $\Lscr^{(d)} = 0$ then $\Lscr$ is said to be {\it solvable,}
and the smallest $d \in \Nat$ satisfying $\Lscr^{(d)} = 0$ is called the {\it derived length\/} of $\Lscr$, denoted $\dl(\Lscr)$.
\end{noname}

\begin{lemma}  \label {A4919C28-E3E4-4D7F-AE02-8F709B2AF0B7}
Let $B$ be a domain of finite transcendence degree over a field $\bk$ of characteristic zero.
\begin{enumerate}[\rm(a)]

\item \label {SkutinProp}
If $\Lscr$ is a Lie subalgebra of $\Der_\bk(B)$ and $\Lscr \subseteq \lnd(B)$ then $\Lscr$ is a locally nilpotent subset of $\Der_\bk(B)$.

\item \label {SkutinCor1} If $\Delta$ is a locally nilpotent subset of $\Der_\bk(B)$ then
the Lie subalgebra of $\Der_\bk(B)$ generated by $(\ker\Delta)\Delta$ is solvable and is a locally nilpotent subset of $\Der_\bk(B)$.

\item \label {SkutinThm4} If $\Delta \subseteq \Der_\bk(B)$ and every finite subset of $\Delta$ is locally nilpotent, then $\Delta$ is locally nilpotent.

\item \label {SkutinThm6} If $\Ascr$ is a Lie subalgebra of $\Der_\bk(B)$ such that $\Ascr \subseteq \lnd(B)$ and $\ker\Ascr = \bk$,
there exist $x_1,\dots,x_m \in B$ algebraically independent over $\bk$
such that $B = \bk[x_1,\dots,x_m]$ and $\Ascr \subseteq \sum_{i=1}^m \bk[x_1,\dots,x_{i-1}] \partial_{x_i}$.

\end{enumerate}
\end{lemma}

Parts (a), (b), (c) and (d) of Lemma \ref{A4919C28-E3E4-4D7F-AE02-8F709B2AF0B7} are, respectively,
the Proposition, Corollary 1, Theorem 4 and Theorem 6 of \cite{Skutin_LND_2021}.
The following is analogous to part (a) of Lemma \ref{A4919C28-E3E4-4D7F-AE02-8F709B2AF0B7}:

\begin{lemma}\cite[Corollary 1]{ChDaRe2026} \label {21AFB0A0-0700-4E7F-8AE2-D3557F9656CC}
Let $B$ be a domain of finite transcendence degree over a field $\bk$ of characteristic zero
and let $\Lscr$ be a Lie subalgebra of $\Der_\bk(B)$ such that $\Lscr \subseteq \lfd_\bk(B)$.
\begin{enumerate}[\rm(a)]

\item If there exists $n \in \Nat$ such that $\Lscr^{(n)}$ is weakly locally finite, then $\Lscr$ is weakly locally finite.

\item If $\Lscr$ is solvable then it is weakly locally finite.

\end{enumerate}
\end{lemma}

\begin{notation} \label {67BC6576-67FE-41A8-975F-597EE29DC120}
Let $B$ be a domain of finite transcendence degree over a field $\bk$ of characteristic zero,
and view $\Der_\bk(B)$ as a Lie algebra over $\bk$.
We define:
\begin{align*}
\LieLN_\bk(B)  &= \ \begin{minipage}[t]{.7\textwidth}
set of Lie subalgebras $\Lscr$ of $\Der_\bk(B)$ such that $\Lscr$ is a locally nilpotent subset of $\Der_\bk(B)$,
\end{minipage} \\
\LieWLF_\bk(B)  &= \ \begin{minipage}[t]{.75\textwidth}
 set of Lie subalgebras $\Lscr$ of $\Der_\bk(B)$ such that $\Lscr$ is a weakly locally finite subset of $\Der_\bk(B)$,
\end{minipage} \\
\LieWLF_\bk(B)_{\text{\rm sol}} &= \setspec{ \Lscr \in \LieWLF_\bk(B) }{ \text{$\Lscr$ is solvable} } . 
\end{align*}
Observe that if $\Lscr \in \LieLN_\bk(B)$ then $\Lscr$ is solvable by Lemma \ref{A4919C28-E3E4-4D7F-AE02-8F709B2AF0B7}\ref{SkutinCor1},
and weakly locally finite by \cite[Remark 4.2]{ChDaRe2026}.
This gives the first inclusion in:
\begin{equation}  \label {1354F7B8-F676-4D8A-9F67-8316E680B009}
\LieLN_\bk(B)  \subseteq \LieWLF_\bk(B)_{\text{\rm sol}} \subseteq \LieWLF_\bk(B) .
\end{equation}
Also, Lemmas \ref{A4919C28-E3E4-4D7F-AE02-8F709B2AF0B7}\ref{SkutinProp} and \ref{21AFB0A0-0700-4E7F-8AE2-D3557F9656CC} give the following equalities:
\begin{align}
\label {D6C764BE-C72A-4279-B9F3-CFD5F5E639ED}
\LieLN_\bk(B) &= \text{ set of Lie subalgebras $\Lscr$ of $\Der_\bk(B)$ such that $\Lscr \subseteq \lnd(B)$,} \\
\label {A949AD42-B8A3-4566-B093-7183EBF9B9B5}
\LieWLF_\bk(B)_{\text{\rm sol}} &= \text{ set of solvable Lie subalgebras $\Lscr$ of $\Der_\bk(B)$ such that $\Lscr \subseteq \lfd_\bk(B)$.}
\end{align}
\end{notation}

The following is a direct consequence of  \cite[Corollary 3]{ChDaRe2026}:

\begin{lemma} \label {CA556DEE-FE7D-4364-A217-A983EBE59700}
Let $B$ be a domain of finite transcendence degree over a field $\bk$ of characteristic zero
and let  $\Lscr \in \LieWLF_\bk(B)$.
If $T$ is a subset of $\lfd_\bk(B)$ such that $[T,T] \subseteq \Lscr$ and $[T,\Lscr] \subseteq \Lscr$,
then  $\Mscr := \LieAlg{T \cup \Lscr}$ belongs to $\LieWLF_\bk(B)$ and $[\Mscr,\Mscr] \subseteq \Lscr$. 
\end{lemma}

\begin{notation}  \label {EDCCFF3C-F3DF-4BCD-82AC-363D73AFAA4B}
If $\Lscr$ is a Lie subalgebra of $\Der_\bk(B)$,
let $\Norm(\Lscr) = \setspec{D \in \Der_\bk(B)}{[D, \Lscr] \subseteq \Lscr }$ denote the normalizer of $\Lscr$ in $\Der_\bk(B)$.
Recall that $\Lscr \subseteq \Norm(\Lscr) \subseteq \Der_\bk(B)$,
$\Norm(\Lscr)$ is a Lie subalgebra of $\Der_\bk(B)$ and $\Lscr$ is an ideal of $\Norm(\Lscr)$.
\end{notation}

\begin{lemma} \label {CAA13785-2B53-4A53-80E3-3857241586CB}
Let $B$ be a domain of finite transcendence degree over a field $\bk$ of characteristic zero.
\begin{enumerate}[\rm(a)]

\item If $\Lscr \in \LieLN_\bk(B)$ and $D \in \Norm(\Lscr) \cap \lnd(B)$ then $\bk D + \Lscr \in \LieLN_\bk(B)$.

\item If $\Lscr$ is a maximal element of $\LieLN_\bk(B)$ then $\Norm(\Lscr) \cap \lnd(B) = \Lscr$.

\end{enumerate}
\end{lemma}

\begin{proof}
(a) By Lemma \ref{SkutinLemma2}, $\{D\} \cup \Lscr$ is a locally nilpotent subset of $\Der_\bk(B)$.
So Remark \ref{03442952-E213-4085-B07A-FABC7408E85A} implies that $\LieAlg{ \{D\} \cup \Lscr } = \bk D + \Lscr$ is a locally nilpotent subset of $\Der_\bk(B)$.
This proves (a), and (b) follows.
\end{proof}

The following observation is useful in conjunction with Zorn's Lemma:

\begin{lemma}  \label {DEA84E5E-79BC-4D4B-BF2D-C5735262EA6D}
Let $B$ be a domain of finite transcendence degree over a field $\bk$ of characteristic zero
and let $\Xscr$  be a nonempty collection of subsets of $\Der_\bk(B)$ that is totally ordered by inclusion.
\begin{enumerate}[\rm(a)]

\item If each element of $\Xscr$ is a locally nilpotent subset of $\Der_\bk(B)$ then so is $\bigcup_{X \in \Xscr} X$.

\item If each element of $\Xscr$ is a weakly locally finite subset of $\Der_\bk(B)$ then so is $\bigcup_{X \in \Xscr} X$.

\item If each element of $\Xscr$ is a solvable Lie subalgebra of $\Der_\bk(B)$ then so is $\bigcup_{X \in \Xscr} X$.

\end{enumerate}
\end{lemma}

\begin{proof}
Part (a) follows from Lemma \ref{A4919C28-E3E4-4D7F-AE02-8F709B2AF0B7}\ref{SkutinThm4} and part (b) is obvious.
For part (c), let $d = \trdeg_\bk(B)$ and recall that each solvable Lie subalgebra of $\Der_\bk(B)$ has derived length $\le 2d$
by \cite[Theorem 2]{MakedonskyiPetravchuk2014};
since each finitely generated subalgebra of $\Lscr = \bigcup_{X \in \Xscr} X$ is a subalgebra of some member of $\Xscr$
and hence has derived length $\le 2d$, we must have $\dl(\Lscr) \le 2d$.
\end{proof}

\section{De Jonqui\`eres-like algebras}\label{sectionDJL}
 
Given a domain $B$ of characteristic zero, we study the set $\DJL(B)$ of ``de Jonqui\`eres-like'' subalgebras of $\Der(B)$
(Definition \ref{E1B94B5F-4823-4C89-95DD-4299351970B9}).
These algebras were first defined by Perepechko in \cite{perepechko2025},
where he described their connection with the nested unipotent subgroups of the automorphism group of an affine variety.

\begin{definition}
If $(D_1,\dots,D_m)$ (where $m\ge0$) is an arbitrary sequence of derivations of a ring $B$ then the {\it kernel sequence\/} of 
$(D_1,\dots,D_m)$ is the sequence of rings $(A_1,\dots,A_m)$
defined by $A_i = \bigcap_{j = i}^m \ker(D_j)$ for each $i \in \{1,\dots,m\}$.
We will often say that the 
kernel sequence of $(D_1,\dots,D_m)$ is $A_1 \subseteq A_2 \subseteq \cdots \subseteq A_m$, instead of $(A_1,\dots,A_m)$.
The two notations are interchangeable.
\end{definition}

\begin{lemma}  \label {C0519934-54B1-477F-8677-BF8168C22DE9}
Let $B$ be a domain of characteristic zero and let $(D_1,\dots,D_m) \in \lnd(B)^m$ be such that $m\ge0$
and $D_iD_j = D_jD_i$ for all $i,j \in \{1,\dots,m\}$.
Let $A_1 \subseteq A_2 \subseteq \cdots \subseteq A_m$ be the kernel sequence of $(D_1,\dots,D_m)$ and also define $A_{m+1} = B$.
The following are equivalent:
\begin{enumerate}[\rm(i)]

\item $D_1, \dots, D_m$ are linearly independent over $B$;

\item $A_i \neq A_{i+1}$ for all $i \in \{1,\dots,m\}$;

\item there exist $s_1,\dots,s_m \in B$ such that for all $i,j \in \{1,\dots,m\}$ we have
$$
\textstyle
\text{$D_i(s_j) \in A_1 \setminus \{0\}$ if $i=j$
\quad and \quad $D_i(s_j)=0$ if $i \neq j$.}
$$

\end{enumerate}
\end{lemma}

\begin{proof}
This is well known. See for instance \cite{Maubach_03}.
\end{proof}

\begin{definition}
Let $B$ be a domain of characteristic zero.
We say that $(D_1,\dots,D_m)$ is an {\it admissible sequence of derivations of $B$\/} if
$m\ge0$, $D_1,\dots,D_m \in \lnd(B)$, $D_iD_j = D_jD_i$ for all $i,j \in \{1,\dots,m\}$,
and the equivalent conditions (i--iii) of Lemma \ref{C0519934-54B1-477F-8677-BF8168C22DE9} are satisfied.
We will sometimes abbreviate ``is an admissible sequence of derivations of $B$'' to ``is admissible''.

If $(D_1,\dots,D_m)$ is an admissible sequence of derivations of $B$ then
a {\it system of local slices\/} of $(D_1,\dots,D_m)$ is a tuple $(s_1,\dots,s_m) \in B^m$ that satisfies
condition (iii) of Lemma \ref{C0519934-54B1-477F-8677-BF8168C22DE9}.
Clearly, each admissible sequence of derivations has at least one system of local slices.

Note that the empty sequence ``$(\,)$'' of derivations is admissible and that its kernel sequence is the empty sequence ``$(\,)$'' of rings.
\end{definition}

\begin{definition}  \label {new-900E2A3C-9A1F-44A9-AB77-981525DA7901}
Let $B$ be a domain of characteristic zero.
If $(D_1,\dots,D_m)$ is an admissible sequence of derivations of $B$ and 
$A_1 \subset \cdots \subset A_m$ is its kernel sequence
then we define $L( D_1,\dots,D_m ) = \sum_{i=1}^m A_i D_i$.
Given a  multiplicatively closed subset $S$ of $A_1 \setminus \{0\}$, we define
\begin{equation*}
\Lscr_S( D_1,\dots,D_m ) = \setspec{ \delta \in \Der(B) }{ s \delta \in L(D_1,\dots,D_m) \text{ for some $s \in S$} }.
\end{equation*}
In the special case where $S = A_1 \setminus \{0\}$, we omit the ``$S$'' from the notation:
\begin{equation*}
\Lscr( D_1,\dots,D_m ) = \setspec{ \delta \in \Der(B) }{ s \delta \in L(D_1,\dots,D_m) \text{ for some $s \in A_1 \setminus\{0\}$} }.
\end{equation*}
Observe that $L(D_1,\dots,D_m) \subseteq \Lscr_S(D_1,\dots,D_m) \subseteq \Lscr(D_1,\dots,D_m)$
are Lie subalgebras of $\Der_{A_1}(B)$, where we view $\Der_{A_1}(B)$ as a Lie algebra over the ring $A_1$.
We will see in Lemma \ref{BAD0996F-5C99-407B-B702-0A97D9510736} that $\Lscr(D_1,\dots,D_m) \subseteq \lnd(B)$.
\end{definition}

\begin{example} \label {B6D5B613-5754-4ECF-8B69-9CFB86A83A7D}
Let $B$ be a domain of characteristic zero. If $D \in \lnd(B) \setminus \{0\}$ then the sequence $(D)$ is admissible and $\Lscr(D) = \lnd_A(B)$, where $A = \ker(D)$.
This describes the $m=1$ case of $\Lscr(D_1,\dots,D_m)$.
The $m=0$ case is described by $\Lscr(\,) = \{0\}$ where $0$ is the zero derivation of $B$.
\end{example}

Before the next Lemma, we need to recall some facts and some terminology.
Part (a) of the following fact appeared in \cite[Section 1]{Nouaze-Gabriel_1967} and \cite[Proposition~2.1]{Wright:JacConj},
and (b) follows from (a).

\begin{noname} \label {SliceThm}
Let $B$ be a $\Rat$-algebra, $D \in \lnd(B)$ and $A = \ker(D)$.
\begin{enumerate}[\rm(a)]

\item If $s\in B$ satisfies $D(s) \in A^*$ then $B = A[s] = A^{[1]}$.

\item Let $t \in B$ be such that $D(t) \neq 0$ and $D^2(t)=0$, and let $a = D(t) \in A\setminus\{0\}$.
Then $B_a = A_a[t] = (A_a)^{[1]}$.

\end{enumerate}
\end{noname}

\begin{noname} \label {CA51E646-7FBD-40C8-9978-60D269C64CB6}
Let $A$ be a domain of characteristic zero, $n\ge1$ and $B = A^{[n]}$.
A {\it coordinate system\/} of $B$ over $A$ is a tuple $(x_1,\dots,x_n) \in B^n$ satisfying $B = A[x_1,\dots,x_n]$.
If $(x_1,\dots,x_n)$ is a coordinate system of $B$ over $A$,
an $A$-derivation $D \in \Der_A(B)$ is {\it triangular with respect to $(x_1,\dots,x_n)$} if $D(x_i) \in A[x_1,\dots,x_{i-1}]$ for all $i=1,\dots,n$
(in particular $D(x_1) \in A$).
An $A$-derivation $D \in \Der_A(B)$ is {\it triangular\/} if it is triangular with respect to some coordinate system of $B$ over $A$.
It is well known that if an $A$-derivation of $B$ is triangular then it is locally nilpotent.
\end{noname}

\begin{noname} \label {CDEAE042-8838-49BE-A23B-4213BCCCD10C}
Let $A \subseteq B$ be domains of characteristic zero, $S = A \setminus \{0\}$, $K = S^{-1}A$ and $\Beul = S^{-1}B$.
Note that $\Der_A(B)$ (resp.\ $\Der_K(\Beul)$) is a Lie algebra over $A$ (resp.\ $K$).

For each $D \in \Der_A(B)$, $S^{-1}D : \Beul \to \Beul$ belongs to $\Der_K(\Beul)$ and satisfies $\ker(S^{-1}D) = S^{-1}\ker(D)$.

If $L$ is a Lie subalgebra of $\Der_A(B)$ then $L^{(S)} := \setspec{ S^{-1}D }{ D \in L }$ is a Lie algebra over $A$ and a Lie subalgebra of $\Der_K(\Beul)$ over $A$
(where we view $\Der_K(\Beul)$ as a Lie algebra over $A$),
and the map $D \mapsto S^{-1}D$ is an isomorphism $L \to L^{(S)}$ of Lie algebras over $A$.
Moreover, $\ker( L^{(S)} ) = S^{-1} \ker(L)$ and the equality $(L^{(i)})^{(S)} = (L^{(S)})^{(i)}$ is valid for all $i \in \Nat$.
Clearly, $L^{(S)} \subseteq \lnd(\Beul)$ if and only if $L \subseteq \lnd(B)$.
\end{noname}

\begin{lemma}  \label{BAD0996F-5C99-407B-B702-0A97D9510736}
Let $B$ be a domain of characteristic zero,
let $(D_1,\dots,D_m)$ be an admissible sequence of derivations of $B$,
let $A_1 \subset \cdots \subset A_m$ be its kernel sequence and let $A_{m+1} = B$. 
Let $S = A_1 \setminus \{0\}$, $\Beul = S^{-1}B$, $\Aeul_i = S^{-1}A_i$ $(1 \le i \le m+1)$ and $\Deul_i = S^{-1}D_i \in \lnd(\Beul)$  $(1 \le i \le m)$.
Then $(\Deul_1,\dots,\Deul_m)$ is an admissible sequence of derivations of $\Beul$
and $\Aeul_1 \subset \cdots \subset \Aeul_{m}$ is its kernel sequence.
Moreover, if $(s_1,\dots,s_m) \in B^m$ is a system of local slices of $(D_1,\dots,D_m)$ then:
\begin{enumerate}[\rm(a)]

\item For each $i \in \{1,\dots,m \}$, $\Aeul_{i+1} = \Aeul_1[s_1,\dots,s_{i}] = \Aeul_i[s_i] = \Aeul_i^{[1]}$ and $B \cap \Aeul_i = A_i$.
In particular, $\Beul = \Aeul_1[s_1,\dots,s_m] = \Aeul_1^{[m]}$.

\item For each $i \in \{1,\dots,m\}$, if we define $u_i = D_i(s_i) \in S$ then $\Deul_i = u_i \frac{\partial}{\partial s_i}$. Moreover,
$$
\textstyle  \Lscr(D_1,\dots,D_m)^{(S)} \subseteq L(\Deul_1,\dots,\Deul_m) = L( \frac{\partial}{\partial s_1} , \dots , \frac{\partial}{\partial s_m} )
= \Lscr( \frac{\partial}{\partial s_1} , \dots , \frac{\partial}{\partial s_m} ) 
$$
where ``\,${}^{(S)}$'' is defined in \ref{CDEAE042-8838-49BE-A23B-4213BCCCD10C} and 
$\Lscr(D_1,\dots,D_m)^{(S)}$ is a Lie subalgebra of $L(\Deul_1,\dots,\Deul_m)$ over $A_1$.

\item $\trdeg_{A_i}(A_{i+1}) = 1$ for every $i = 1, \dots, m$.

\item $\Lscr(D_1,\dots,D_m) \subseteq \lnd(B)$

\end{enumerate}
\end{lemma}

\begin{proof}
The claim that $(\Deul_1,\dots,\Deul_m)$ is admissible and $\Aeul_1 \subset \cdots \subset \Aeul_{m}$ is its kernel sequence
easily follows from \ref{7E1F8E6C-B326-4B86-A442-CA26364C2C1C}(ii).  
The derivation $\Deul_i|_{\Aeul_{i+1}} : \Aeul_{i+1} \to \Aeul_{i+1}$ is locally nilpotent and has kernel $\Aeul_i$;
since $s_i \in \Aeul_{i+1}$ and $\Deul_i|_{\Aeul_{i+1}}(s_i) \in S \subseteq \Aeul_1^*$,
\ref{SliceThm} implies that $\Aeul_{i+1} = \Aeul_i[s_i] = \Aeul_i^{[1]}$.
Since $A_i$ is factorially closed in $B$, $B \cap \Aeul_i = A_i$ (see \ref{04E41BA5-443D-48FF-9B44-9279A9BECCA5}, \ref{7E1F8E6C-B326-4B86-A442-CA26364C2C1C}).
This proves (a), and assertions (b) and (c) easily follow.
To prove (d), note that $L( \frac{\partial}{\partial s_1} , \dots , \frac{\partial}{\partial s_m} )$ is the set of $\Aeul_1$-derivations of $\Aeul_1[s_1,\dots,s_m] = \Aeul_1^{[m]}$
that are triangular with respect to $(s_1,\dots,s_n)$ (see \ref{CA51E646-7FBD-40C8-9978-60D269C64CB6});
in particular, $L( \frac{\partial}{\partial s_1} , \dots , \frac{\partial}{\partial s_m} ) \subseteq \lnd(\Beul)$ and hence $\Lscr(D_1,\dots,D_m)^{(S)} \subseteq \lnd(\Beul)$ by (b).
 As noted in \ref{CDEAE042-8838-49BE-A23B-4213BCCCD10C}, it follows that $\Lscr(D_1,\dots,D_m) \subseteq \lnd(B)$.
\end{proof}

\begin{lemma} \label {E881691D-8557-4672-B859-92E3341ADF4B}
Let $B$ be a domain of characteristic zero, $(D_1,\dots,D_m)$ an admissible sequence of derivations of $B$, $L = L(D_1,\dots,D_m)$ and $\Lscr = \Lscr(D_1,\dots,D_m)$.
Let $(A_1,\dots,A_m)$ be the kernel sequence of $(D_1,\dots,D_m)$.
\begin{enumerate}[\rm(a)]

\item For each $i \in \{1,\dots,m\}$, there exists $(a_i,\dots,a_m) \in A_i \times \cdots \times A_m$ such that
\begin{itemize}

\item $a_i D_i, \, a_{i+1}D_{i+1}, \, \dots, \, a_m D_m \in L^{(i-1)} \setminus \{0\}$;

\item for each $j$ such that $i < j \le m$, $D_i D_{i+1} \cdots D_{j-1} a_j \neq 0$.

\end{itemize}

\item $\dl(L) = m = \dl(\Lscr)$

\item $A_i = \ker( L^{(i-1)} ) = \ker( \Lscr^{(i-1)} )$ for all $i = 1,\dots,m$.

\end{enumerate}
\end{lemma}

\begin{proof}
We define $S = A_1 \setminus \{0\}$ and $\Deul_1, \dots, \Deul_m, \Aeul_1, \dots, \Aeul_m$ as in Lemma \ref{BAD0996F-5C99-407B-B702-0A97D9510736}.
If $1 \le i \le j \le m$ then $[A_i D_i, A_j D_j] \subseteq D_i(A_j)D_j$; this implies that $L^{(1)} \subseteq L(D_2,\dots,D_m)$,
and it follows by induction that  
\begin{equation} \label {704C628C-2F09-4EFE-8EC1-102471D43361}
L^{(i-1)} \subseteq L(D_i,\dots,D_m) \quad \text{for all $i=1,\dots,m$.}
\end{equation}
In particular, $L^{(m-1)} \subseteq L(D_m)$ and hence $L^{(m)} \subseteq [L(D_m), L(D_m)] = \{0\}$. Thus,
\begin{equation}  \label {9698D969-97D0-4F85-B2E6-D40107A4E87A}
\dl( L ) \le m .
\end{equation}
Refer to \ref{CDEAE042-8838-49BE-A23B-4213BCCCD10C} for the notation ``${}^{(S)}$'' and its properties.
Let $i \in \{1,\dots,m\}$; 
we have  $\Lscr^{(S)} \subseteq L(\Deul_1,\dots,\Deul_m)$ by Lemma \ref{BAD0996F-5C99-407B-B702-0A97D9510736}(b),
so  $( \Lscr^{(i-1)} )^{(S)} = ( \Lscr^{(S)} )^{(i-1)} \subseteq L(\Deul_1,\dots,\Deul_m)^{(i-1)} \subseteq L(\Deul_i,\dots,\Deul_m)$
where the last inclusion follows by applying \eqref{704C628C-2F09-4EFE-8EC1-102471D43361} to $(\Deul_1,\dots,\Deul_m)$.  Hence,
$S^{-1}\ker ( \Lscr^{(i-1)} ) = \ker \big( ( \Lscr^{(i-1)} )^{(S)} \big)  \supseteq \ker L(\Deul_i,\dots,\Deul_m) = \Aeul_i = S^{-1}A_i$.
Since $\ker ( \Lscr^{(i-1)} )$ and $A_i$ are factorially closed in $B$, we have 
$B \cap S^{-1}\ker ( \Lscr^{(i-1)} ) = \ker ( \Lscr^{(i-1)} )$ and  $B \cap S^{-1}A_i = A_i$, so
\begin{equation} \label {EFA8928A-7E76-4924-B6DA-5C05CCAB5D36}
A_i \subseteq \ker( \Lscr^{(i-1)} ) \quad \text{for all $i=1,\dots,m$.}
\end{equation}
Applying \eqref{9698D969-97D0-4F85-B2E6-D40107A4E87A} to $(\Deul_1,\dots,\Deul_m)$ gives $\dl( L(\Deul_1,\dots,\Deul_m) ) \le m$,
and since $\Lscr^{(S)} \subseteq L(\Deul_1,\dots,\Deul_m)$ it follows that $\dl(\Lscr^{(S)}) \le m$.
We have $\dl(\Lscr) = \dl(\Lscr^{(S)})$ because $\Lscr$ and $\Lscr^{(S)}$ are isomorphic Lie algebras by \ref{CDEAE042-8838-49BE-A23B-4213BCCCD10C}, so
\begin{equation} \label {68AAEB0E-5174-4CBB-A95E-BAB0CA1159BE}
\dl(L) \le \dl(\Lscr) \le m .
\end{equation}
For each $i \in \{1,\dots,m\}$, let $\Pscr(i)$ be the assertion that 
there exists $(a_i,\dots,a_m) \in A_i \times \cdots \times A_m$ such that
\begin{enumerate}

\item[\rm (1$_i$)] $a_i D_i, \, a_{i+1}D_{i+1}, \, \dots, \, a_m D_m \in L^{(i-1)} \setminus \{0\}$;

\item[\rm (2$_i$)] for each $j$ such that $i < j \le m$, $D_i D_{i+1} \cdots D_{j-1} a_j \neq 0$.

\end{enumerate}
To prove that $\Pscr(1)$ is true, consider a system of local slices $(s_1, \dots, s_m) \in B^m$ of $(D_1,\dots,D_m)$.
We have $D_i(s_i) \in A_1 \setminus \{0\}$ for all $i=1,\dots,m$ and $D_i(s_j) = 0$ whenever $i \neq j$.
Define $a_1=1 \in A_1 \setminus \{0\}$ and $a_j = s_1 \cdots s_{j-1} \in A_j \setminus \{0\}$ for each $j \in \{2, \dots, m\}$. 
Then $(a_1,\dots,a_m) \in A_1 \times \cdots \times A_m$ and  $a_1 D_1, \, \dots, \, a_m D_m \in L^{(0)} \setminus \{0\}$, so $(1_1)$ is true.
For each $j \in \{2, \dots, m\}$,
$D_1 \cdots D_{j-1} a_j = D_1 \cdots D_{j-1}(s_1 \cdots s_{j-1}) = D_1(s_1) \cdots D_{j-1}(s_{j-1}) \in A_1 \setminus \{0\}$,
so $(2_1)$ is true. This shows that $\Pscr(1)$ is true.

Suppose that $i$ is such that $1 \le i < m$ and $\Pscr(i)$ is true.
Let $(a_i,\dots,a_m) \in A_i \times \cdots \times A_m$ be such that (1$_i$) and (2$_i$) are satisfied.
For each $j \in \{i+1, \dots, m\}$, define $a_j' = a_i D_i(a_j)$.
It is clear that $(a_{i+1}', \dots, a_m') \in A_{i+1} \times \cdots \times A_m$.
Let us check that $(a_{i+1}', \dots, a_m')$ satisfies (1$_{i+1}$) and (2$_{i+1}$).
Conditions $(1_{i})$ and $(2_{i})$ imply that $a_i \neq 0$ and $D_{i}(a_{i+1}) \neq 0$, so $a_{i+1}' = a_i D_i(a_{i+1}) \neq 0$ and hence $a_{i+1}'D_{i+1} \neq 0$.
If $j$ is such that $i+1 < j \le m$ then we have $a_i \neq 0$ by $(1_i)$ and $D_i D_{i+1} \cdots D_{j-1} a_j \neq 0$ by $(2_i)$, so
$$
D_{i+1} D_{i+2} \cdots D_{j-1} a_j'
= D_{i+1} D_{i+2} \cdots D_{j-1}(a_i D_i(a_j)) 
= a_i D_i D_{i+1} D_{i+2} \cdots D_{j-1}( a_j) \neq 0 ,
$$
showing that $(2_{i+1})$ holds. Since $(2_{i+1})$ implies that $a_{i+2}', \dots, a_m' \neq 0$, we have $a_{i+1}'D_{i+1}, \dots, a_{m}'D_{m} \neq 0$.
If $j \in \{i+1, \dots, m\}$ then $a_i D_i, a_j D_j \in L^{(i-1)}$, so $L^{(i)} \ni [a_i D_i, a_j D_j] = a_i D_i(a_j) D_j = a_j' D_j \neq 0$,
showing that (1$_{i+1}$) holds.  Hence, $\Pscr(i+1)$ is true.
By induction, $\Pscr(1), \dots, \Pscr(m)$ are true; this proves (a).

By $\Pscr(m)$, there exists $\alpha_m \in A_m$ such that $\alpha_m D_m \in L^{(m-1)} \setminus \{0\}$;
so $L^{(m-1)} \neq 0$.
Together with \eqref{68AAEB0E-5174-4CBB-A95E-BAB0CA1159BE}, this implies that (b) is true.
For each $i \in \{1,\dots,m\}$, we have $a_i D_i, \, a_{i+1}D_{i+1}, \, \dots, \, a_m D_m \in L^{(i-1)} \setminus \{0\}$ by (a),
so $A_i = \ker\{a_iD_i, a_{i+1}D_{i+1}, \dots, a_mD_m\} \supseteq \ker L^{(i-1)} \supseteq \ker \Lscr^{(i-1)} \supseteq A_i$,
where the last inclusion is \eqref{EFA8928A-7E76-4924-B6DA-5C05CCAB5D36}.
This proves (c), and completes the proof of the Lemma.
\end{proof}

\begin{notation}  \label {900E2A3C-9A1F-44A9-AB77-981525DA7901}
Let $B$ be a domain of characteristic zero.
Given an arbitrary increasing sequence $A_1 \subseteq \cdots \subseteq A_m$ of subrings of $B$,
we define $\Lscr(A_1,\dots,A_m)$ to be the set of all $D \in \lnd_{A_1}(B)$
satisfying $D(A_{i}) \subseteq A_i$ for all $i=1,\dots,m$. Note that, in this generality, $\Lscr(A_1,\dots,A_m)$ is not necessarily closed under addition.
% \footnote{For instance, let $B = \Comp^{[2]}$, $A_1=\Comp$ and $A_2=B$; then $\Lscr(A_1,A_2) = \lnd(B)$ is not closed under addition.}
({\it Remark:} If $m=0$ then $\Lscr(A_1,\dots,A_m) = \Lscr( \, ) = \lnd_{A_1}(B) = \lnd_B(B) = \{0\}$ because $A_1= A_{m+1}=B$.)
\end{notation}

\begin{proposition} \label {0615A9A0-CFB7-445C-B34E-44BA30B13CFA}
Let $B$ be a domain of characteristic zero, $(D_1,\dots,D_m)$ an admissible sequence of derivations of $B$,
and $A_1 \subset \cdots \subset A_m$ the kernel sequence of $(D_1,\dots,D_m)$.
Then $\Lscr(D_1,\dots,D_m) = \Lscr(A_1,\dots,A_m)$.
\end{proposition}

\begin{proof}
We claim that if $i,j \in \{1,\dots,m\}$ and $a_i \in A_i$ then $a_i D_i(A_j) \subseteq A_j$.
Indeed, if $j \le i$ then $A_j \subseteq \ker(D_i)$, so $a_iD_i(A_j) = \{0\} \subseteq A_j$;
and if $j > i$ then $a_i \in A_i \subseteq A_j$ and $D_i(A_j) \subseteq A_j$, so $a_iD_i(A_j) \subseteq A_j$.
This proves the above claim, and it follows that each element of $L(D_1,\dots,D_m)$ maps each $A_j$ into itself;
since $L(D_1,\dots,D_m) \subseteq \lnd_{A_1}(B)$, we obtain that $L(D_1,\dots,D_m) \subseteq \Lscr(A_1,\dots,A_m)$.
Using that each $A_j$ is factorially closed in $B$, we deduce that if $\delta \in \Der(B)$ and $s \in A_1 \setminus \{0\}$
are such that $s \delta \in\Lscr(A_1,\dots,A_m)$, then $\delta$ maps each $A_j$ into itself;
since $s\delta \in \lnd_{A_1}(B)$ by definition of $\Lscr(A_1,\dots,A_m)$,
we have $\delta \in \lnd_{A_1}(B)$ by \ref{7E1F8E6C-B326-4B86-A442-CA26364C2C1C}(iii), so $\delta \in\Lscr(A_1,\dots,A_m)$.
Consequently, $\Lscr(D_1,\dots,D_m) \subseteq \Lscr(A_1,\dots,A_m)$.

To prove the reverse inclusion, we use the notation of Lemma \ref{BAD0996F-5C99-407B-B702-0A97D9510736} (including $A_{m+1}=B$).
Consider $D \in \Lscr(A_1,\dots,A_m)$ and its localization $\Deul = S^{-1}D : \Beul \to \Beul$;
clearly, $\Deul \in \Lscr( \Aeul_1, \dots, \Aeul_m )$.

If $i \in \{ 1,\dots,m \}$ then $\Aeul_{i+1} = \Aeul_i[s_i] = \Aeul_i^{[1]}$;
since $\Deul$ is locally nilpotent and satisfies $\Deul(\Aeul_i) \subseteq \Aeul_i$ and $\Deul(\Aeul_i[s_i]) \subseteq \Aeul_i[s_i]$,
we have $\Deul(s_i) \in \Aeul_i$ (see for instance Principle 6 on page 23 of \cite{Freud:Book-new}).
Consequently, $D(s_i) \in \Aeul_i \cap B = A_i$ where the last equality follows from the fact that $A_i$ is factorially closed in $B$.
Recall from Lemma \ref{BAD0996F-5C99-407B-B702-0A97D9510736} that (for each $i=1,\dots,m$)
$\Deul_i = u_i \frac{\partial}{\partial s_i}$, where $u_i = D_i(s_i) \in S$.
Since $\Deul$ is an $\Aeul_1$-derivation of $\Beul = \Aeul_1[s_1,\dots,s_m] = \Aeul_1^{[m]}$,
we have $\Deul = \sum_{i=1}^m \Deul(s_i) \frac{\partial}{\partial s_i}
= \sum_{i=1}^m \Deul(s_i) \frac1{u_i} \Deul_i$.
Let $u = u_1 \cdots u_m \in S$ and (for each $i=1,\dots,m$) $\hat u_i = u/u_i \in S$.
Then $u \Deul = \sum_{i=1}^m \Deul(s_i) \hat u_i \Deul_i$
and consequently $u D = \sum_{i=1}^m D(s_i) \hat u_i D_i$.
As $D(s_i) \hat u_i \in A_i$ for all $i$, we have $u D \in L(D_1,\dots,D_m)$ and hence $D \in \Lscr(D_1,\dots,D_m)$.
This shows that $\Lscr(A_1,\dots,A_m) \subseteq \Lscr(D_1,\dots,D_m)$, which completes the proof of the Lemma.
\end{proof}

\begin{corollary} \label {ujd3876g20oigjh2vy68}
Let $B$ be a domain of characteristic zero and let $(D_1, \dots, D_m)$ and $(E_1, \dots, E_n)$ be admissible sequences of derivations of $B$.
$$
\Lscr(D_1, \dots,D_m) = \Lscr(E_1, \dots, E_n) \ \iff \ \begin{minipage}[t]{.5\textwidth}
$(D_1, \dots, D_m)$ and $(E_1, \dots, E_n)$ have the same\\ kernel sequence.
\end{minipage}
$$
Moreover, if these conditions hold then $m=n$.
\end{corollary}

\begin{proof}
This follows from Lemma \ref{E881691D-8557-4672-B859-92E3341ADF4B}(c) and
Proposition \ref{0615A9A0-CFB7-445C-B34E-44BA30B13CFA}.
\end{proof}

\begin{corollary} \label {58D561A8-FF48-417A-9D8F-635BA02F689F}
Let $B$ be a domain of characteristic zero.
If $(D_1,\dots,D_m)$ is admissible
and, for each $i=1,\dots,m$, $D_i = a_i \delta_i$ for some $a_i \in B$ and $\delta_i \in \Der(B)$,
then $(\delta_1,\dots,\delta_m)$ is admissible and $\Lscr(D_1,\dots,D_m) = \Lscr(\delta_1,\dots,\delta_m)$.
\end{corollary}

\begin{proof}
Corollary \ref{FE9079CB-02CC-4FD7-8149-9E402ABC4851} implies that $\delta_1,\dots,\delta_m$ belong to $\lnd(B)$ and pairwise commute,
and clearly $(\delta_1,\dots,\delta_m)$ has the same kernel sequence as $(D_1,\dots,D_m)$.
So $(\delta_1,\dots,\delta_m)$ is admissible
and Corollary \ref{ujd3876g20oigjh2vy68} gives $\Lscr(D_1,\dots,D_m) = \Lscr(\delta_1,\dots,\delta_m)$.
\end{proof}

\begin{definition}   \label {E1B94B5F-4823-4C89-95DD-4299351970B9}
Given a domain $B$ of characteristic zero, we define the set
$$
\DJL(B) = \setspec{ \Lscr(D_1,\dots,D_m) }{ \text{$(D_1,\dots,D_m)$ is an admissible sequence of derivations of $B$} } .
$$
If $\Lscr \in \DJL(B)$ and $A = \ker(\Lscr)$ then $\Lscr$ is a Lie subalgebra of $\Der_{A}(B)$
and we have $\Lscr \subseteq \lnd_{A}(B)$ by Lemma \ref{BAD0996F-5C99-407B-B702-0A97D9510736}(d). The elements of $\DJL(B)$ are called ``de Jonqui\`eres-like subalgebras of $\Der(B)$'',
but instead of using that name we systematically write ``$\Lscr \in \DJL(B)$''.
\end{definition}

We are interested in the poset $(\DJL(B),\subseteq)$, and in particular in its maximal elements.

\begin{remark} \label {D9D95302-2C52-4CBD-A475-2C8367B13B8F}
Let $B$ be a domain of characteristic zero and $\Lscr \in \DJL(B)$.
Then $\Lscr = \Lscr(D_1,\dots,D_m)$ for some admissible sequence $(D_1,\dots,D_m)$ of derivations of $B$.
For any such $(D_1,\dots,D_m)$, we have
$$
m = \trdeg(B : \ker(\Lscr) )  = \dl(\Lscr) .
$$
Indeed, let $A_1 \subset \cdots \subset A_m$ be the kernel sequence of $(D_1,\dots,D_m)$ and let $A_{m+1} = B$.
We have $A_1 = \ker(\Lscr)$ and Lemma \ref{BAD0996F-5C99-407B-B702-0A97D9510736} shows that $\trdeg(A_{i+1} : A_i) = 1$ for all $i=1,\dots,m$,
so $m = \trdeg(B : \ker(\Lscr) )$. The fact that $m = \dl(\Lscr)$ is Lemma \ref{E881691D-8557-4672-B859-92E3341ADF4B}(b).
By part (c) of the same Lemma,
all admissible sequences $(D_1,\dots,D_m)$ satisfying $\Lscr = \Lscr(D_1,\dots,D_m)$ have the same kernel sequence.
So the following definition is logically sound.
\end{remark}

\begin{definition}  \label {2B03F8AD-D2DB-4E88-9526-69E7D6B6E44E}
Let $B$ be a domain of characteristic zero and let $\Lscr \in \DJL(B)$.
We define the {\it kernel sequence of $\Lscr$}, denoted $\kerseq(\Lscr)$,
to be the kernel sequence of any admissible sequence $(D_1,\dots,D_m)$ satisfying $\Lscr = \Lscr(D_1,\dots,D_m)$.
\end{definition}

\begin{notation}
Given a domain $B$ of characteristic zero,
let $\KS(B)$ be the set of tuples $(A_1,\dots,A_m)$ such that $m\ge0$,
$A_1 \subseteq \cdots \subseteq A_m$ are subrings of $B$,
and if we define $A_{m+1}=B$ and $S = A_1 \setminus \{0\}$ then  
$$
\text{$B \cap S^{-1}A_i = A_i$ 
and $S^{-1}A_{i+1} = (S^{-1}A_i)^{[1]}$\ \ for all $i = 1,\dots, m$.}
$$
We define a partial order $\preceq$ on $\KS(B)$ by declaring that,
given $(A_1,\dots,A_m), (B_1,\dots,B_n) \in \KS(B)$,
$$
(A_1,\dots,A_m) \preceq (B_1,\dots,B_n) \iff \ m \le n \text{ and } (A_1, \dots, A_{m}) = (B_{n-m+1}, \dots, B_{n}) .
$$
\end{notation}

We will see in Proposition \ref{6CFC46F0-0103-4BE1-8ACB-BDA7B53CB093}
that the map $\Lscr \mapsto \kerseq(\Lscr)$ is an injective homomorphism of posets from $(\DJL(B), \subseteq)$ to $(\KS(B),\preceq)$,
and that it is an isomorphism of posets when $B$ is finitely generated as an algebra over a field.

\begin{remark} \label {B32B4DC5-A0C7-4017-8426-C822141AF8D4}
Let $B$ be a domain of characteristic zero.
\begin{enumerate}[\rm(a)]

\item If $(A_1,\dots,A_m) \in \KS(B)$ then $A_1,\dots,A_m$ are factorially closed subrings of $B$.
This follows from $A_i = B \cap S^{-1}A_i$ together with the fact that $S^{-1}A_i$ is factorially closed in $S^{-1}B = (S^{-1}A_i)^{[m+1-i]}$.
Moreover, $\trdeg_{A_i}(A_{i+1}) = 1$ for all $i = 1, \dots, m$, where we define $A_{m+1}=B$.
Also note that if $\bk$ is a field and $\bk \subseteq B$ then $\bk \subseteq A_1$, because $A_1$ is factorially closed in $B$.

\item If $(A_1,\dots,A_m), (B_1,\dots,B_n) \in \KS(B)$ and $(A_1,\dots,A_m) \prec (B_1,\dots,B_n)$ then
$m<n$, $B_1 \subset A_1$, and $\trdeg_{B_1}(A_1) = n-m > 0$.

\end{enumerate}
\end{remark}

\begin{lemma}  \label {860EF548-B386-4610-B6AC-3864D47DDD15}
Let $B$ be a domain of characteristic zero.
\begin{enumerate}[\rm(a)]

\item If $\Lscr \in \DJL(B)$ then $\kerseq(\Lscr) \in \KS(B)$.

\end{enumerate}
For parts {\rm(b)} and {\rm(c)}, assume that $B$ is a finitely generated algebra over a field.
\begin{enumerate}[\rm(a)]
\addtocounter{enumi}{1}

\item Each element of $\KS(B)$ is the kernel sequence of some element of $\DJL(B)$.

\item If $(A_1,\dots,A_m) \in \KS(B)$ and $A_{m+1} :=B$ then $A_{i} \in \klnd(A_{i+1})$ for all $i = 1,\dots,m$.

\end{enumerate}
\end{lemma}

\begin{proof}
Part (a) follows from Lemma \ref{BAD0996F-5C99-407B-B702-0A97D9510736}.

(b) Let $(A_1,\dots,A_m) \in \KS(B)$ and define $A_{m+1}=B$.
Let $\bk \subseteq B$ be a field over which $B$ is a finitely generated algebra, and note that $\bk \subseteq A_1$ by Remark  \ref{B32B4DC5-A0C7-4017-8426-C822141AF8D4}.
Let $S = A_1 \setminus \{0\}$, $\Beul = S^{-1}B$, and $\Aeul_i = S^{-1}A_i$ $(1 \le i \le m+1)$.
For each $i=1,\dots,m$, we have $\Aeul_{i+1} = \Aeul_i^{[1]}$, so we can choose $x_i \in \Aeul_{i+1}$ such that 
$\Aeul_{i+1} = \Aeul_i[x_i] = \Aeul_i^{[1]}$.
So $\Beul = \Aeul_1[x_1,\dots,x_m] = \Aeul_1^{[m]}$, and we may consider the partial derivatives $\partial_{x_1}, \dots, \partial_{x_m} : \Beul \to \Beul$.
For each $i = 1,\dots,m$, the fact that $\bk \subseteq \ker( \partial_{x_i} )$ and that $B$ is finitely generated as a $\bk$-algebra
implies that there exists $\sigma_i \in S$ such that $\sigma_i \partial_{x_i} : \Beul \to \Beul$ maps $B$ into itself;
so we may consider the restriction $D_i : B \to B$ of $\sigma_i \partial_{x_i}$.
The derivations $D_1,\dots,D_m \in \lnd(B)$ pairwise commute,
and since $\ker(D_j) = B \cap \ker(\partial_{x_j})$ for each $j = 1,\dots,m$, we have
$\bigcap_{j\ge i} \ker(D_j) = B \cap \bigcap_{j\ge i} \ker(\partial_{x_j}) = B \cap \Aeul_1[x_1,\dots,x_{i-1}] = B \cap \Aeul_i = A_i$ for each $i=1,\dots,m$.
That is, the kernel sequence of $(D_1,\dots,D_m)$ is $(A_1,\dots,A_m)$.
As $A_i \neq A_{i+1}$ for all $i=1,\dots,m$, $(D_1,\dots,D_m)$ is admissible.
So $\Lscr(D_1,\dots,D_m)$ is an element of $\DJL(B)$ whose kernel sequence is $(A_1,\dots,A_m)$. This proves (b).

(c) Part (b) implies that there exists an admissible sequence $(D_1,\dots,D_m)$ of derivations of $B$ whose kernel sequence is $(A_1,\dots,A_m)$.
If $i \in \{1,\dots,m\}$ then $D_i|_{A_{i+1}} : A_{i+1} \to A_{i+1}$ is a nonzero locally nilpotent derivation of $A_{i+1}$ whose kernel is $A_i$;
so $A_i \in \klnd(A_{i+1})$, which proves (c).
\end{proof}

\begin{lemma}  \label {BC319DE3-1827-4876-96D9-C0E8DBBECECB}
Let $B$ be a domain of characteristic zero and $(D_1,\dots,D_m)$ an admissible sequence of derivations of $B$
with kernel sequence $A_1 \subset \cdots \subset A_{m}$. 
For each $j \in \{1,\dots,m\}$,
$$
\setspec{ D \in \Lscr(D_1,\dots,D_m) }{ A_j \subseteq \ker(D) } = \Lscr_{A_1 \setminus \{0\}}(D_j, D_{j+1}, \dots, D_m) .
$$
\end{lemma}
\begin{proof}
Let $j \in \{1,\dots,m\}$.
The inclusion ``$\supseteq$'' is clear.
To prove ``$\subseteq$'', consider $D \in \Lscr(D_1,\dots,D_m)$ such that $A_j \subseteq \ker(D)$.
There exists $\alpha \in A_1 \setminus \{0\}$ such that $\alpha D \in L(D_1,\dots,D_m)$.
So we have $\alpha D = \sum_{i=1}^m a_i D_i$ for some $a_1,\dots,a_m$ such that $a_i \in A_i$ for all $i$.
We claim that 
\begin{equation} \label {10E2C932-E592-4025-A889-416F26AC4B0A}
a_1 = \cdots = a_{j-1} = 0 .
\end{equation}
This is obvious if $D=0$, so let us assume that $D \neq 0$.
Then there exists $r \in \{1,\dots,m\}$ such that $\alpha D = \sum_{i=r}^m a_i D_i$ and $a_r \neq 0$.
We can choose $x \in A_{r+1} \setminus A_r$ (as usual, define $A_{m+1}=B$);
then $\alpha D(x) = a_r D_r(x) \neq 0$, so $A_{r+1} \nsubseteq \ker(D)$ and consequently $j \le r$,
from which \eqref{10E2C932-E592-4025-A889-416F26AC4B0A} follows.
Now \eqref{10E2C932-E592-4025-A889-416F26AC4B0A} implies that $\alpha D \in L(D_j, \dots, D_m)$,
and since $\alpha \in A_1 \setminus \{0\}$ we obtain $D \in \Lscr_{A_1 \setminus \{0\}}(D_j, D_{j+1}, \dots, D_m)$.
\end{proof}

\begin{lemma} \label {00A23F46-9C1D-449A-BE41-68ACD0236AC5}
Let $B$ be a domain of characteristic zero and let
$(D_1,\dots,D_m)$ and $(E_1,\dots,E_n)$ be admissible sequences of derivations of $B$,
with kernel sequences $A_1 \subset \cdots \subset A_{m}$ and $B_1 \subset \cdots \subset B_{n}$ respectively.
\begin{enumerate}[\rm(a)]

\item $(A_1,\dots,A_m), (B_1,\dots,B_n) \in \KS(B)$

\item If $L(D_1, \dots, D_m) \subseteq \Lscr(E_1,\dots,E_n)$ then $(A_1,\dots,A_m) \preceq (B_1,\dots,B_n)$ in $\KS(B)$.

\item If $\Lscr(D_1, \dots, D_m) \subseteq \Lscr(E_1,\dots,E_n)$ then  $(A_1,\dots,A_m) \preceq (B_1,\dots,B_n)$ in $\KS(B)$ and
\begin{align*}
\Lscr(D_1,\dots,D_m) &= \Lscr_{B_1 \setminus \{0\}}( E_{n-m+1}, \dots, E_n ) = \Lscr( E_{n-m+1}, \dots, E_n ) \\
&= \Lscr(A_1, \dots, A_{m}) = \Lscr(B_{n-m+1}, \dots, B_{n}) .
\end{align*}

\end{enumerate}
\end{lemma}

\begin{proof}
Part (a) is a reiteration of Lemma \ref{860EF548-B386-4610-B6AC-3864D47DDD15}(a).

(b) Suppose that $L(D_1, \dots, D_m) \subseteq \Lscr(E_1,\dots,E_n)$.
Since $A_1$ (resp.\ $B_1$) is the intersection of the kernels of all elements of 
$L(D_1, \dots, D_m)$ (resp.\ $\Lscr(E_1,\dots,E_n)$), we see that $A_1 \supseteq B_1$; consequently, $\trdeg_{A_1}(B) \le \trdeg_{B_1}(B)$.
By Lemma \ref{BAD0996F-5C99-407B-B702-0A97D9510736}, $\trdeg_{A_1}(B) = m$ and $\trdeg_{B_1}(B) = n$; so 
$$
m \le n.
$$
To prove (b), it remains to show that $(A_1, \dots, A_{m}) = (B_{n-m+1}, \dots, B_{n})$.
Define $A_{m+1} = B$ and $B_{n+1}=B$.
For each $d \in \{0, \dots, m\}$, we define $\Pscr(d)$ to be the assertion ``$A_{m+1-d} = B_{n+1-d}$\,''.
So it suffices to show that $\Pscr(0), \dots, \Pscr(m)$ are true.
Since $A_{m+1} = B = B_{n+1}$, $\Pscr(0)$ is true.
Assume that $d \in \{0, \dots, m-1\}$ is such that $\Pscr(d)$ is true, and let us prove $\Pscr(d+1)$.
So we assume that $A_{m+1-d} = B_{n+1-d}$ and we have to prove that $A_{m-d} = B_{n-d}$.

Since $D_{m-d} \in L(D_1,\dots,D_m) \subseteq \Lscr(E_1,\dots,E_n)$, there exists $\alpha \in B_1 \setminus \{0\}$
such that $\alpha D_{m-d} \in L(E_1,\dots,E_n)$. So we have
$\alpha D_{m-d} = \sum_{i=1}^n \alpha_i E_i$ for some $\alpha_1 \in B_1$, \dots, $\alpha_n \in B_n$.
Since $D_1,\dots,D_m$ are linearly independent over $B$ and $\alpha \neq 0$, we have $\alpha D_{m-d} \neq 0$.
So there exists $j \in \{1,\dots,n\}$ such that
\begin{equation} \label {835841C8-3139-49F4-AA49-1B100580CA8F}
\textstyle
\alpha D_{m-d} = \sum_{i=j}^n \alpha_i E_i \quad \text{and} \quad \alpha_j \neq 0 .
\end{equation}

Consider $a \in A_{m-d} \setminus \{0\}$.
Since $\alpha \in B_1 \setminus \{0\} \subseteq A_1 \setminus \{0\} \subseteq A_{m-d} \setminus \{0\}$, we have $a\alpha \in A_{m-d} \setminus \{0\}$,
so $a\alpha D_{m-d} \in L(D_1,\dots,D_m) \subseteq \Lscr(E_1,\dots,E_n)$ and there exists $\alpha' \in B_1 \setminus \{0\}$
such that $\alpha' a \alpha D_{m-d} \in L(E_1,\dots,E_n)$.
Consequently,
$$
\textstyle
\alpha' a \alpha D_{m-d} = \sum_{i=1}^n \hat\alpha_i E_i \quad \text{for some $\hat\alpha_1 \in B_1$, \dots, $\hat\alpha_n \in B_n$.}
$$
On the other hand, \eqref{835841C8-3139-49F4-AA49-1B100580CA8F} gives $\alpha' a \alpha D_{m-d} = \sum_{i=j}^n (\alpha' a \alpha_i) E_i$.
Since $E_1,\dots,E_n$ are linearly independent over $B$, it follows that $\alpha' a \alpha_j = \hat \alpha_j \in B_j$.
Since $\alpha_j \neq 0$, we have $\alpha' a \alpha_j \in B_j \setminus \{0\}$ and hence $a \in B_j$, because $B_j$ is factorially closed in $B$.
This being true for an arbitrary $a \in A_{m-d}$, it follows that $A_{m-d} \subseteq B_j$.
Since $\trdeg_{A_{m-d}}(B) = d+1$ and $\trdeg_{B_j}(B) = n+1-j$, we get $j \ge n-d$.
In other words, we have $\alpha_i = 0$ for all $i < n-d$, so $\alpha D_{m-d} = \sum_{i=n-d}^n \alpha_i E_i$.
This implies that $B_{n-d} = \bigcap_{i = n-d}^n \ker(E_i) \subseteq \ker(D_{m-d})$.
We also have  $B_{n-d} \subseteq B_{n+1-d} = A_{m+1-d}$ (the equality is $\Pscr(d)$), so $B_{n-d} \subseteq \ker(D_{m-d}) \cap A_{m+1-d} = A_{m-d}$.
Since $\trdeg_{B_{n-d}}(B) = \trdeg_{A_{m-d}}(B)$, $A_{m-d}$ is algebraic over $B_{n-d}$; as $B_{n-d}$ is algebraically closed in $B$,
we conclude that $A_{m-d} = B_{n-d}$, i.e., $\Pscr(d+1)$ is true. This proves (b).

(c)  Suppose that $\Lscr(D_1, \dots, D_m) \subseteq \Lscr(E_1,\dots,E_n)$.
Then the hypothesis of (b) is satisfied, so $m \le n$ and $(A_1, \dots, A_{m}) = (B_{n-m+1}, \dots, B_{n})$.
If $D \in \Lscr(D_1,\dots,D_m)$ then $D \in \Lscr(E_1,\dots,E_n)$ and $A_1 = B_{n-m+1}$ is included in $\ker(D)$,
so $D \in \Lscr_{B_1 \setminus \{0\}}( E_{n-m+1}, \dots, E_n )$ by Lemma \ref{BC319DE3-1827-4876-96D9-C0E8DBBECECB}.
This gives the first inclusion in:
\begin{multline*}
\Lscr(D_1,\dots,D_m) \subseteq \Lscr_{B_1 \setminus \{0\}}( E_{n-m+1}, \dots, E_n ) \subseteq \Lscr( E_{n-m+1}, \dots, E_n ) \\
= \Lscr(B_{n-m+1}, \dots, B_{n}) = \Lscr(A_1, \dots, A_{m}) = \Lscr(D_1,\dots,D_m) ,
\end{multline*}
where the first and last equalities follow from Proposition \ref{0615A9A0-CFB7-445C-B34E-44BA30B13CFA}.
Assertion (c) follows.
\end{proof}

\begin{remark} \label {CEC40263-FB7A-4204-8EFF-32CC255344BB}
Let $B$ be a domain of characteristic zero and let $(D_1,\dots,D_m)$ and $(E_1,\dots,E_m)$ be admissible sequences of derivations of $B$ (of the same length).
\begin{enumerate}[\rm(a)]

\item If $\Lscr(D_1,\dots,D_m) \subseteq \Lscr(E_1,\dots,E_m)$ then $\Lscr(D_1,\dots,D_m) = \Lscr(E_1,\dots,E_m)$.

\item If $\Lscr(D_1,\dots,D_m) \subseteq L(E_1,\dots,E_m)$ then $\Lscr(D_1,\dots,D_m) = L(E_1,\dots,E_m) = \Lscr(E_1,\dots,E_m)$.

\end{enumerate}
Indeed, part (a) follows from Lemma \ref{00A23F46-9C1D-449A-BE41-68ACD0236AC5}(c) and part (b) follows from part (a).
\end{remark}

\begin{proposition} \label {6CFC46F0-0103-4BE1-8ACB-BDA7B53CB093}
Let $B$ be a domain of characteristic zero. Then the map
$$
\kerseq : \DJL(B) \to \KS(B)
$$
is an injective homomorphism of posets, from $(\DJL(B), \subseteq)$ to $(\KS(B),\preceq)$.
Moreover, if $B$ is a finitely generated algebra over a field then $\kerseq$ is an isomorphism of posets
and its inverse is the map that sends $(A_1,\dots,A_m) \in \KS(B)$ to $\Lscr(A_1,\dots,A_m) \in \DJL(B)$.
\end{proposition}

\begin{proof}
If $\Lscr \in \DJL(B)$ then $\kerseq(\Lscr) \in \KS(B)$ by Lemma \ref{860EF548-B386-4610-B6AC-3864D47DDD15}(a);
so $\kerseq : \DJL(B) \to \KS(B)$ is a well defined set map.
If $\Lscr_1, \Lscr_2 \in \DJL(B)$ have the same kernel sequence then $\Lscr_1 = \Lscr_2$ by Proposition  \ref{0615A9A0-CFB7-445C-B34E-44BA30B13CFA},
so the map $\kerseq$ is injective.  It is a homomorphism of posets by Lemma \ref{00A23F46-9C1D-449A-BE41-68ACD0236AC5}(c).
Assume that $B$ is a finitely generated algebra over a field.
By Lemma \ref{860EF548-B386-4610-B6AC-3864D47DDD15}(b), $\kerseq : \DJL(B) \to \KS(B)$ is surjective and hence bijective. 
By Proposition \ref{0615A9A0-CFB7-445C-B34E-44BA30B13CFA},
the inverse of $\kerseq$ is the map $(A_1,\dots,A_m) \mapsto \Lscr(A_1,\dots,A_m)$, which is easily seen to be a homomorphism of posets;
so $\kerseq$ is an isomorphism of posets.
\end{proof}

\begin{remark}  \label{5A4B8846-954B-45B2-AD99-00914B6F31F0}
Let $B$ be a domain of characteristic zero and suppose that we have a strict inclusion $\Lscr_1 \subset \Lscr_2$ with $\Lscr_1, \Lscr_2 \in \DJL(B)$.
Then $\ker(\Lscr_1) \supset \ker(\Lscr_2)$ and $\trdeg( \ker(\Lscr_1) : \ker(\Lscr_2) ) > 0$. 
Indeed, let $(A_1,\dots,A_m) = \kerseq(\Lscr_1)$ and $(B_1,\dots,B_n) = \kerseq(\Lscr_2)$, and note that $\ker(\Lscr_1) = A_1$ and  $\ker(\Lscr_2) = B_1$.
Proposition \ref{6CFC46F0-0103-4BE1-8ACB-BDA7B53CB093} implies that $(A_1,\dots,A_m) \prec (B_1,\dots,B_n)$ in $\KS(B)$,
so Remark \ref{B32B4DC5-A0C7-4017-8426-C822141AF8D4} gives $\trdeg(A_1:B_1)>0$.
\end{remark}

\begin{remark}  \label {48B77BA7-E2B4-4887-BF9D-B75658D0CC7C}
If $B$ is a domain of finite transcendence degree over a field $\bk$ of characteristic zero then each
element of $\DJL(B)$ is included in a maximal element.
(This follows from Remark \ref{5A4B8846-954B-45B2-AD99-00914B6F31F0} and the fact that $\ker(\Lscr) \supseteq \bk$ for each $\Lscr \in \DJL(B)$.)
\end{remark}

%%%%%%%%%%%%%%%%%%%%%%%%%%%%%%%%%%%%%%%%%%%%%%%%%%%%%%%%%%%%%%%%%%%%%%%%%%%%%%%%%%%%%%%%%%%%%%%%%%%%%%%%%%%%%%%%%%%%%%%%%%%%%%%%%%
%%%%%%%%%%%%%%%%%%%%%%%%%%%%%%%%%%%%%%%%%%%%%%%%%%%%%%%%%%%%%%%%%%%%%%%%%%%%%%%%%%%%%%%%%%%%%%%%%%%%%%%%%%%%%%%%%%%%%%%%%%%%%%%%%%

The following result is a digression that serves to illustrate how Proposition \ref{6CFC46F0-0103-4BE1-8ACB-BDA7B53CB093} can be applied.
Lemma \ref{DJLNagata1} will be used in paragraph \ref{6A24CF4A9C6743B6} where the Nagata automorphism of $\Aff^3$
is shown to belong to a unique Borel subgroup of $\Aut(\Aff^3)$.
For the notion of {\it rank of a derivation}, see \cite[3.2.1]{Freud:Book-new}.

\begin{lemma} \label {DJLNagata1}
Let $\bk$ be a field of characteristic zero and $B = \kk3$.
Let $\Delta \in \lnd(B)$ be irreducible of rank $2$, let $z$ be a variable of $B$ that belongs to $\ker(\Delta)$, and let $f \in \ker(\Delta) \setminus  \bk[z]$.
Then $\Lscr(\bk[z],\ker(\Delta))$ is the unique maximal element of $\DJL(B)$ that contains $f\Delta$.
\end{lemma}

\begin{proof}
Let $A = \ker\Delta$ and note that the $1$-term sequence $(A)$ is an element of $\KS(B)$. We claim:
\begin{equation}  \label {2FFCB4F0-5ECF-4A7B-A19A-D4AF5CC60398}
\text{for each $\Lscr \in \DJL(B)$, \quad  $f\Delta \in \Lscr \iff \kerseq(\Lscr) \succeq (A)$ in $\KS(B)$. }
\end{equation}
To prove $(\Leftarrow)$, consider the element $\Lscr_1 = \Lscr(\Delta) \in \DJL(B)$
and note that $f\Delta \in \Lscr_1$ and $\kerseq(\Lscr_1) = (A)$.
If $\Lscr \in \DJL(B)$ is such that $\kerseq(\Lscr) \succeq (A)$, then $\Lscr \supseteq \Lscr_1 \ni f\Delta$ because $\kerseq$ is an isomorphism of posets
by Proposition \ref{6CFC46F0-0103-4BE1-8ACB-BDA7B53CB093}.
Conversely, suppose that $\Lscr \in \DJL(B)$ is such that $f\Delta \in \Lscr$.
We have $\Lscr = \Lscr(D_1, \dots, D_n)$ for some admissible sequence $(D_1,\dots,D_n)$, where $n \in \{1,2,3\}$.
Since $D_n$ commutes with every element of $\Lscr$,\footnote{If $D \in \Lscr$ then $sD \in L(D_1, \dots, D_n)$ for some $s \in \ker(\Lscr)\setminus\{0\}$, so $D_n$ commutes
with $sD$ and hence also with $D$, by Corollary \ref{FE9079CB-02CC-4FD7-8149-9E402ABC4851}.}
it commutes with $f\Delta$,
so Corollary \ref{FE9079CB-02CC-4FD7-8149-9E402ABC4851} implies that $[\Delta,D_n]=0$ and $f \in \ker(\Delta) \cap \ker(D_n)$.
If $\ker D_n \neq \ker \Delta$ then $\ker D_n \cap \ker\Delta = \bk[v]$ for some variable $v$ of $B$, by \cite[Proposition 5.40]{Freud:Book-new};
since $v$ and $z$ are variables of $B$ that belong to $\ker\Delta$, where $\rank\Delta=2$,
\cite[Theorem 2.5]{daigle1996necessary} implies that $\bk[v] = \bk[z]$, which contradicts the hypothesis that $f \notin \bk[z]$.
So we must have $\ker D_n = \ker\Delta = A$, which implies that $\kerseq(\Lscr) \succeq (A)$.
This proves \eqref{2FFCB4F0-5ECF-4A7B-A19A-D4AF5CC60398}.
Next, we show that
\begin{equation} \label {96D5A066-5AD3-4961-A64F-B526B62916DB}
\text{$(\bk[z],A)$ is the unique element $\Agoth$ of $\KS(B)$ that satisfies $\Agoth \succ (A)$.}
\end{equation}
By \cite[Theorem 2.4]{daigleFreudenburgRank2}, we have $A = \bk[z]^{[1]}$ and $S^{-1}B = (S^{-1}A)^{[1]}$ where $S = \bk[z] \setminus \{0\}$;
this implies that $(\bk[z],A) \in \KS(B)$, and $(\bk[z],A) \succ (A)$ is clear.
For uniqueness, let $\Agoth$ be an element of $\KS(B)$ such that $\Agoth \succ (A)$. Clearly, $\Agoth = (A_1, \dots, A_m)$ where $A_m = A$ and $m \in \{2,3\}$.
If $m=3$ then $A_1=\bk$ and hence $B = A^{[1]}$, which is not possible because $\rank(\Delta)\neq1$ implies $B \neq A^{[1]}$.
So $\Agoth = (A_1, A)$ for some $A_1$,
and $\Lscr(A_1, A) = \Lscr(D,E)$ for some admissible sequence $(D,E)$.
Again, \cite[Proposition 5.40]{Freud:Book-new} implies that $A_1 = \ker D \cap \ker E = \bk[v]$ for some variable $v$ of $B$
and \cite[Theorem 2.5]{daigle1996necessary} implies that $\bk[v] = \bk[z]$ since $v$ and $z$ are variables of $B$ belonging to $A$.
So $A_1 = \bk[z]$ and hence $\Agoth = (\bk[z],A)$, proving \eqref{96D5A066-5AD3-4961-A64F-B526B62916DB}.

By \eqref{96D5A066-5AD3-4961-A64F-B526B62916DB} and Proposition \ref{6CFC46F0-0103-4BE1-8ACB-BDA7B53CB093},
$\Lscr(\bk[z],A)$ is the unique $\Lscr \in \DJL(B)$ such that $\kerseq(\Lscr) \succ (A)$;
so \eqref{2FFCB4F0-5ECF-4A7B-A19A-D4AF5CC60398} implies that  $\Lscr(\bk[z],A)$ is the unique maximal element of $\DJL(B)$ that contains $f\Delta$.
\end{proof}

We now consider the problem of finding all immediate successors (\ref{ACF5CAEBFDFD404C}) of a given element of $\DJL(B)$.
More generally, our goal is to describe the posets $(\DJL(B), \subseteq)$ and $(\KS(B),\preceq)$.

%%%%%%%%%%%%%%%%%%%%%%%%%%%%%%%%%%%%%%%%%%%%%%%%%%%%%%%%%%%%%%%%%%%%%%%%%%%%%%%%%%%%%%%%%%%%%%%%%%%%%%%%%%%%%%%%%%%%%%%%
%%%%%%%%%%%%%%%%%%%%%%%%%%%%%%%%%%%%%%%%%%%%%%%%%%%%%%%%%%%%%%%%%%%%%%%%%%%%%%%%%%%%%%%%%%%%%%%%%%%%%%%%%%%%%%%%%%%%%%%%

\begin{lemmadef} \label{E18C9915-1DF7-4A48-A602-582A546C96F5}
Let $\bk$ be a field of characteristic zero, $B$ an affine $\bk$-domain and $\Agoth = (A_1,\dots,A_m) \in \KS(B)$.
For a  $\bk$-subalgebra $A_0$ of $A_1$, if we define $S = A_0 \setminus \{0\}$ and $A_{m+1}=B$ then the following conditions are equivalent:
\begin{enumerate}[\rm(i)]

\item $A_0 \in \klnd(A_1)$ and $S^{-1}A_{i+1} = (S^{-1}A_i)^{[1]}$ for all $i=1,\dots,m$;

\item $B \cap S^{-1}A_0 = A_0$ and $S^{-1}A_{i+1} = (S^{-1}A_i)^{[1]}$ for all $i=0,\dots,m$;

\item $(A_0,A_1,\dots,A_m) \in \KS(B)$.

\end{enumerate}
We define $\Omega(\Agoth)$ to be the set of $\bk$-subalgebras $A_0$ of $A_1$ that satisfy {\rm(i--iii)}.
\end{lemmadef}

\begin{proof}
If $A_0$ satisfies (i)  then $A_0 \in \klnd(A_1)$, so $A_0$ is factorially closed in $A_1$;
since $A_1$ is factorially closed in $B$ by Remark \ref{B32B4DC5-A0C7-4017-8426-C822141AF8D4},
$A_0$ is factorially closed in $B$ and hence $B \cap S^{-1}A_0 = A_0$.
The fact that $A_0 \in \klnd(A_1)$ also implies that  $S^{-1}A_1 = (S^{-1}A_0)^{[1]}$, so $A_0$ satisfies (ii).

Suppose that $A_0$ satisfies (ii).
The fact that $(A_1,\dots,A_m) \in \KS(B)$ implies (by Remark \ref{B32B4DC5-A0C7-4017-8426-C822141AF8D4})
that $A_1, \dots, A_m$ are factorially closed in $B$; so $B \cap S^{-1}A_i = A_i$ for all $i=1,\dots,m$.
Since (ii) holds, we also have $B \cap S^{-1}A_0 = A_0$ and $S^{-1}A_{i+1} = (S^{-1}A_i)^{[1]}$ for all $i=0,\dots,m$.
It follows that $(A_0,A_1,,\dots,A_m) \in \KS(B)$, i.e., (iii) holds.

If (iii) holds then Lemma \ref{860EF548-B386-4610-B6AC-3864D47DDD15}(c) gives $A_0 \in \klnd(A_1)$ and the definition of $\KS(B)$
implies that $S^{-1}A_{i+1} = (S^{-1}A_i)^{[1]}$ for all $i=0,\dots,m$.
So (i) holds.
\end{proof}

\begin{definition} \label {ACF5CAEBFDFD404C}
Let $(X,\le)$ be a poset and let $x \in X$.
An {\it immediate successor\/} of $x$ is an element $y \in X$ such that $x < y$ and there is no $y' \in X$ satisfying $x < y' < y$.
An {\it immediate predecessor\/} of $x$ is an element $w \in X$ such that $w < x$ and there is no $w' \in X$ satisfying $w < w' < x$.
\end{definition}

See \ref{9823hf276v38ghfvcd9387tg} for the definition of $\rtrdeg_\bk(B)$.

\begin{proposition} \label {E024C289-4F11-48E4-831C-C94E59C34ACE}
Let $\bk$ be a field of characteristic zero and $B$ an affine $\bk$-domain.
Then the poset \mbox{$(\DJL(B),\subseteq)$} has the following properties.
\begin{enumerate}[\rm(a)]

\item If $\Lscr_0 \subset \Lscr_1 \subset \cdots \subset \Lscr_d$ is a strictly increasing sequence in $\DJL(B)$ then $d \le \rtrdeg_\bk(B)$.

\item The set of minimal elements of $\DJL(B) \setminus \{ \mathbf{0} \}$ is $\setspec{ \lnd_A(B) }{ A \in \klnd(B) }$.

\item Let $\Lscr$ be a non-minimal element of $\DJL(B)$ with kernel sequence $(A_1, \dots, A_m)$.
Then $m \ge 1$ and $\Lscr' = \Lscr(A_2, \dots, A_m)$ is the unique immediate predecessor of $\Lscr$.
Moreover, if $(D_1, \dots, D_m)$ is an admissible sequence satisfying $\Lscr = \Lscr(D_1,\dots,D_m)$
then $\Lscr' = \Lscr(D_2,\dots,D_m)$.

\item If $\Lscr \in \DJL(B)$ then the set $\setspec{ \Lscr' \in \DJL(B) }{ \Lscr' \subseteq \Lscr }$ is finite and totally ordered by inclusion.

\item Let $\Lscr \in \DJL(B)$, let $\Agoth = (A_1,\dots,A_m) \in \KS(B)$ be the kernel sequence of $\Lscr$
and let $\Omega(\Agoth)$ be the set defined in Lemma \ref{E18C9915-1DF7-4A48-A602-582A546C96F5}.
Then the set of immediate successors of $\Lscr$ in $\DJL(B)$ is 
$$
\setspec{ \Lscr(A_0,A_1,\dots,A_m) }{ A_0 \in \Omega(\Agoth) } .
$$
More precisely, $A_0 \mapsto\Lscr(A_0,A_1,\dots,A_m)$ is a bijection from $\Omega(\Agoth)$ to the set of immediate successors of $\Lscr$.
In particular, $\Lscr$ is maximal in $\DJL(B)$ if and only if $\Omega(\Agoth) = \emptyset$.

\end{enumerate}
\end{proposition}

\begin{proof}
As a preliminary remark, note that if $(A_1,\dots,A_m) \in \KS(B)$ then $\bk \subseteq A_1$ by Remark \ref{B32B4DC5-A0C7-4017-8426-C822141AF8D4}(a).
Also note that Proposition \ref{6CFC46F0-0103-4BE1-8ACB-BDA7B53CB093} implies:
\begin{equation} \label {3F50FE5B-F075-454D-94AD-0E9FB69E171B}
\begin{minipage}{.8\textwidth}
$\kerseq : (\DJL(B),\subseteq) \to (\KS(B),\preceq)$ is an isomorphism of posets
and its inverse is the map that sends $(A_1,\dots,A_m) \in \KS(B)$ to $\Lscr(A_1,\dots,A_m) \in \DJL(B)$.
\end{minipage}
\end{equation}

(a) The definitions of $\KS(B)$ and $\preceq$ immediately imply that 
if $\Agoth_0 \prec \Agoth_1 \prec \cdots \prec \Agoth_d$ is a strictly increasing sequence in $\KS(B)$ then $d \le \rtrdeg_\bk(B)$.
In view of \eqref{3F50FE5B-F075-454D-94AD-0E9FB69E171B}, (a) follows.

(b) We claim:
\begin{equation} \label {FBDDFCEE-04CE-4FB3-83F1-393B5DD53FD1}
\text{the set of minimal elements of $\KS(B) \setminus \{ (\,) \}$ is $\setspec{ (A) }{ A \in \klnd(B) }$.}
\end{equation}
By definition, the minimal elements of $\KS(B) \setminus \{ (\,) \}$
are the one-term sequences $(A)$  where $A$ is a subring of $B$ that satisfies the conditions
$B \cap S^{-1}A = A$ and $S^{-1}B = (S^{-1}A)^{[1]}$, with $S = A \setminus \{0\}$.
Since $B$ is an affine $\bk$-domain, it is well known that these rings $A$ are precisely the elements of $\klnd(B)$.
(See for instance \cite[Proposition 1.4]{Dai:smooth}.)
So \eqref{FBDDFCEE-04CE-4FB3-83F1-393B5DD53FD1} is true. 
Since $\Lscr(A) = \lnd_A(B)$ for each $A \in \klnd(B)$,
assertion (b) follows from \eqref{3F50FE5B-F075-454D-94AD-0E9FB69E171B} and \eqref{FBDDFCEE-04CE-4FB3-83F1-393B5DD53FD1}.

(c) By \eqref{3F50FE5B-F075-454D-94AD-0E9FB69E171B}, $(A_1, \dots, A_m)$ is a non-minimal element of $\KS(B)$.
Clearly, it follows that $m \ge 1$ and that $(A_2, \dots, A_m)$ is the unique immediate predecessor of $(A_1, \dots, A_m)$ in $\KS(B)$.
So, by \eqref{3F50FE5B-F075-454D-94AD-0E9FB69E171B}, $\Lscr' := \Lscr(A_2, \dots, A_m)$ is the unique immediate predecessor of $\Lscr$ in $\DJL(B)$.
If $(D_1, \dots, D_m)$ is an admissible sequence satisfying $\Lscr = \Lscr(D_1,\dots,D_m)$
then the kernel sequence of $(D_1, \dots, D_m)$ is $(A_1, \dots, A_m)$ by definition of $\kerseq(\Lscr)$ (see \ref{2B03F8AD-D2DB-4E88-9526-69E7D6B6E44E}).
Clearly, $(D_2,\dots,D_m)$ is an admissible sequence with kernel sequence $(A_2, \dots, A_m)$;
so $\Lscr'$ and $\Lscr(D_2, \dots, D_m)$ are elements of $\DJL(B)$ that have the same kernel sequence.
By injectivity of $\kerseq$, it follows that $\Lscr' = \Lscr(D_2,\dots,D_m)$.

(d) Part (c) implies that each non-minimal element of $\DJL(B)$ has a unique immediate predecessor, 
and part (a) implies that the lengths of strictly increasing sequences in $\DJL(B)$ are uniformly bounded;
any poset with these two properties has the property stated in (d).

(e) We claim that 
\begin{equation}  \label {4BDE3DF7-8D8B-4438-9BBA-4AE73FC4B76E}
\text{the set of immediate successors of $\Agoth$ in $\KS(B)$ is $\setspec{ (A_0,A_1,\dots,A_m) }{ A_0 \in \Omega(\Agoth) }$.}
\end{equation}
Indeed, suppose that $\Agoth'$ is an immediate successor of $\Agoth$.
Then $\Agoth$ is the unique immediate predecessor of $\Agoth'$,
so $\Agoth' = (A_0,A_1,,\dots,A_m)$ for some subring $A_0$ of $A_1$, where $\bk \subseteq A_0$ by Remark \ref{B32B4DC5-A0C7-4017-8426-C822141AF8D4}(a).
Since $A_0$ satisfies condition \ref{E18C9915-1DF7-4A48-A602-582A546C96F5}(iii), we have $A_0 \in \Omega(\Agoth)$.

Conversely, if $A_0 \in \Omega(\Agoth)$ then $A_0$ satisfies \ref{E18C9915-1DF7-4A48-A602-582A546C96F5}(iii), i.e., $(A_0,A_1,\dots,A_m) \in \KS(B)$.
Clearly, $\Agoth=(A_1,,\dots,A_m)$ is the unique immediate predecessor of $(A_0,A_1,,\dots,A_m)$.
Equivalently, $(A_0,A_1,,\dots,A_m)$ is an immediate successor of $\Agoth$.

This proves \eqref{4BDE3DF7-8D8B-4438-9BBA-4AE73FC4B76E}.
Assertion (e) follows from \eqref{4BDE3DF7-8D8B-4438-9BBA-4AE73FC4B76E} and \eqref{3F50FE5B-F075-454D-94AD-0E9FB69E171B}.
\end{proof}

The following example illustrates how Proposition \ref{E024C289-4F11-48E4-831C-C94E59C34ACE} can be used to find all immediate successors of a given element of $\DJL(B)$.
The Lie algebras that arise in Example \ref{BF979604-4D5B-4A13-A6AA-E9EDC00540C7} will be of particular interest
and will be further studied in Example \ref{C60B4C67-6814-44F3-8DDF-E43EEF900AFA}.

\begin{example} \label {BF979604-4D5B-4A13-A6AA-E9EDC00540C7}
Let $\bk$ be a field of characteristic zero and $B = \bk[u,v,y,z] = \bk^{[4]}$.
Consider the $\bk$-derivations $D = \partial_y$ and $E = u \partial_z - v \partial_y$ of $B$
and note that $\ker(D) = \bk[u,v,z]$ and $\ker(E) = \bk[u,v,Q]$ where $Q = uy+vz$.
It is easily checked that $(D,E)$ is admissible; thus, $\Lscr(D,E) \in \DJL(B)$.
Let us describe the immediate successors (if any) of $\Lscr(D,E)$ in $\DJL(B)$.

The kernel sequence of $\Lscr(D,E)$ is $\Agoth = (A_1,A_2)$ where  $A_1=\bk[u,v] \subset A_2=\bk[u,v,Q]$.
By condition (i) of Lemma \ref{E18C9915-1DF7-4A48-A602-582A546C96F5},
the elements of the set $\Omega(\Agoth)$ are all $R \in \klnd(A_1)$ such that
\begin{equation}  \label {F3285086-45C7-4FA1-9041-7CA1DEC37268}
\text{if we define $S = R \setminus \{0\}$, then $S^{-1}B = (S^{-1}A_2)^{[1]}$.} 
\end{equation}
(Condition (i) of Lemma \ref{E18C9915-1DF7-4A48-A602-582A546C96F5} also requires that $S^{-1}A_2 = (S^{-1}A_1)^{[1]}$, 
but this condition holds automatically because $A_2 = A_1^{[1]}$.)
For each $R \in \Omega(\Agoth)$, define $\Lscr_R = \Lscr(R,A_1,A_2)$. We claim:
\begin{enumerate}[\rm(a)]

\item $\Omega(\Agoth) = \setspec{ \bk[f] }{ \text{$f$ is a variable of $\bk[u,v]$} }$

\item  The set of immediate successors of $\Lscr(D,E)$ in $\DJL(B)$ is $\setspec{ \Lscr_R }{ R \in \Omega(\Agoth) }$.
Moreover, if $R, R' \in \Omega(\Agoth)$ are distinct then $\Lscr_{R} \neq \Lscr_{R'}$.

\item For each $R \in \Omega(\Agoth)$, $\Lscr_R$ is a maximal element of $\DJL(B)$ and has derived length $3$.

\end{enumerate}
Let us prove (a). If $R \in \Omega(\Agoth)$ then $R \in \klnd( \bk[u,v] )$, so Rentschler's Theorem \cite{rentschler1968operations}
gives $R = \bk[f]$ for some variable $f$ of $\bk[u,v]$.
Conversely,  let $R = \bk[f]$ where $f$ is a variable of $\bk[u,v]$. Clearly, $R \in \klnd(\bk[u,v])$.
Let $S = R \setminus \{0\}$; we have to prove \eqref{F3285086-45C7-4FA1-9041-7CA1DEC37268}.
If $\lambda \in \bk$ is the constant term of $f(u,v)$, then $f-\lambda = \alpha u + \beta v$ for some $\alpha, \beta \in \bk[u,v]$.
Then $E(\alpha z - \beta y) = f-\lambda \in S$,
so the locally nilpotent derivation $S^{-1}E : S^{-1}B \to S^{-1}B$ satisfies 
$(S^{-1}E)\big( \frac{\alpha z - \beta y}{f-\lambda} \big) = 1$.
By \ref{SliceThm}(a), it follows that $S^{-1}B = ( \ker( S^{-1}E ) )^{[1]} = (S^{-1}A_2)^{[1]}$, i.e., \eqref{F3285086-45C7-4FA1-9041-7CA1DEC37268} holds
and consequently $R \in \Omega(\Agoth)$.
This proves (a).

Assertion (b) follows from Proposition  \ref{E024C289-4F11-48E4-831C-C94E59C34ACE}(e).
As $\Omega(\Agoth)$ is an infinite set by (a), (b) implies that $\Lscr(D,E)$ has infinitely many immediate successors.
In particular, $\Lscr(D,E)$ is not maximal in $\DJL(B)$.

If $R \in \Omega(\Agoth)$ then $\dl( \Lscr_R ) = 3$ is clear, for instance by Lemma \ref{E881691D-8557-4672-B859-92E3341ADF4B}(b).
So, to prove (c), it suffices to show that $\Omega( \Agoth_R ) = \emptyset$ where $\Agoth_R = (R,A_1,A_2)$ is the kernel sequence of $\Lscr_R$.
Each element of $\Omega( \Agoth_R )$ is an element of $\klnd(R)$, where $R = \kk1$; so $\Omega( \Agoth_R )$ is either empty or equal to $\{\bk\}$.
If $\bk \in \Omega( \Agoth_R )$ then $(\bk,R,A_1,A_2) \in \KS(B)$ and hence $B = A_2^{[1]}$, which is false.
So $\Omega( \Agoth_R ) = \emptyset$ and consequently $\Lscr_R$ is maximal in $\DJL(B)$ by Proposition \ref{E024C289-4F11-48E4-831C-C94E59C34ACE}(e).
This proves (c).
\end{example}

Our next objective is to compare the posets $(\DJL(B),\subseteq)$ and $(\LieLN_\bk(B),\subseteq)$ 
(see Definition \ref{67BC6576-67FE-41A8-975F-597EE29DC120}).
The main result in this direction is Proposition \ref{6D6E22A4-236A-4B28-96D9-86157812D34A}.
We begin with a trivial observation:

\begin{remark} \label {D8E0AD3A-6104-479D-B7E2-D38E67645BBD}
{\it
If $B$ is a domain of finite transcendence degree over a field $\bk$ of characteristic zero then $\DJL(B) \subseteq \LieLN_\bk(B)$.}
Indeed, if $\Lscr \in \DJL(B)$ then (as noted in Definition \ref{E1B94B5F-4823-4C89-95DD-4299351970B9}) $\Lscr \subseteq \lnd(B)$,
so $\bk \subseteq \ker(\Lscr)$ by \ref{7E1F8E6C-B326-4B86-A442-CA26364C2C1C}(i);
since (as noted in Definition \ref{E1B94B5F-4823-4C89-95DD-4299351970B9}) $\Lscr$ is a Lie algebra over the ring $\ker(\Lscr)$,
it is therefore a Lie algebra over $\bk$.
So $\Lscr$ is a Lie subalgebra of $\Der_\bk(B)$ over $\bk$ and  $\Lscr \subseteq \lnd(B)$.
By statement \eqref{D6C764BE-C72A-4279-B9F3-CFD5F5E639ED} in Definition \ref{67BC6576-67FE-41A8-975F-597EE29DC120}, $\Lscr \in \LieLN_\bk(B)$.
\end{remark}

\begin{definition}
An element $\Lscr$ of $\LieLN_\bk(B)$ is {\it relatively maximal\/} if it is a maximal element of the set
$\setspec{ \Lscr' \in \LieLN_\bk(B) }{ \ker(\Lscr') = \ker(\Lscr) }$.
\end{definition}

\begin{proposition}  \label {6D6E22A4-236A-4B28-96D9-86157812D34A}
Let $\bk$ be a field of characteristic zero and $B$ an affine $\bk$-domain.
\begin{enumerate}[\rm(a)]

\item $\DJL(B) = \setspec{ \Lscr \in \LieLN_\bk(B) }{ \text{$\Lscr$ is relatively maximal} }$

\item If $\Lscr \in \LieLN_\bk(B)$ then there exists  $\Lscr' \in \DJL(B)$
such that $\Lscr \subseteq \Lscr'$ and $\ker(\Lscr) = \ker(\Lscr')$.

\item The posets $( \LieLN_\bk(B) , \subseteq )$ and $( \DJL(B) , \subseteq )$ have the same maximal elements.

\item Each element of $\DJL(B)$ (resp.\ $\LieLN_\bk(B)$) is included in a maximal element.

\end{enumerate}
\end{proposition}

\begin{proof}
We have $\DJL(B) \subseteq \LieLN_\bk(B)$ by Remark \ref{D8E0AD3A-6104-479D-B7E2-D38E67645BBD}.
We prove (b), (c), (d) and (a), in this order.
We first prove the following statement, which is stronger than (b):
\begin{enumerate}
\item[(b$'$)] If $\Delta$ is a locally nilpotent subset of $\Der_\bk(B)$ then there exists  $\Lscr \in \DJL(B)$
such that $\Delta \subseteq \Lscr$ and $\ker(\Delta) = \ker(\Lscr)$.
\end{enumerate}
Let $S = \ker(\Delta) \setminus \{0\}$, $K = S^{-1} \ker(\Delta)$ and $\Beul = S^{-1}B$.
For each $D \in \Delta$, we consider the derivation $S^{-1}D : \Beul \to \Beul$.
The fact that $\Delta$ is a locally nilpotent subset of $\Der_\bk(B)$ implies that
$S^{-1}\Delta = \setspec{ S^{-1}D }{ D \in \Delta }$ is a locally nilpotent subset of $\Der_K(\Beul)$.
Clearly, $\ker\big( S^{-1}\Delta \big) = K$.
Let $\Ascr$ be the Lie subalgebra of $\Der_K(\Beul)$ over $K$ generated by $S^{-1}\Delta$.
By Remark \ref{03442952-E213-4085-B07A-FABC7408E85A}, $\Ascr$ is a locally nilpotent subset of $\Der_K(\Beul)$;
since $\ker(\Ascr) = K$,
Lemma \ref{A4919C28-E3E4-4D7F-AE02-8F709B2AF0B7}\ref{SkutinThm6} implies that 
there exist $x_1,\dots,x_m \in B$ algebraically independent over $K$
such that $B = K[x_1,\dots,x_m]$ and $\Ascr \subseteq \sum_{i=1}^m K[x_1,\dots,x_{i-1}] \partial_{x_i}$.
For each $i=1,\dots,m$, define $A_i = B \cap K[x_1,\dots,x_{i-1}]$ and note that $S^{-1}A_i = K[x_1,\dots,x_{i-1}]$;
it follows that $(A_1,\dots,A_m) \in \KS(B)$,
so Proposition \ref{6CFC46F0-0103-4BE1-8ACB-BDA7B53CB093} implies that there exists $\Lscr \in \DJL(B)$ such that $\kerseq(\Lscr)=(A_1,\dots,A_m)$.
Note that $\ker(\Lscr) = A_1 = B \cap K = \ker(\Delta)$ and that (by Proposition \ref{0615A9A0-CFB7-445C-B34E-44BA30B13CFA}) $\Lscr = \Lscr(A_1,\dots,A_m)$.
So, to prove (b$'$), it suffices to check that $\Delta \subseteq \Lscr(A_1,\dots,A_m)$.
Consider $\delta \in \Delta$. Then $S^{-1}\delta \in \Ascr \subseteq \sum_{i=1}^m K[x_1,\dots,x_{i-1}] \partial_{x_i}$,
so for each $i=1,\dots,m$, $S^{-1}\delta$ maps $K[x_1,\dots,x_{i-1}]$ into itself.
It follows that for each $i=1,\dots,m$, $\delta$ maps $B \cap K[x_1,\dots,x_{i-1}] = A_i$ into itself.
As $\delta \in \lnd_{A_1}(B)$, we have $\delta \in \Lscr(A_1,\dots,A_m)$, which proves that $\Delta \subseteq \Lscr(A_1,\dots,A_m) = \Lscr$.
This proves (b$'$), and (b) follows.

Assertion (c) follows from (b) and the inclusion $\DJL(B) \subseteq \LieLN_\bk(B)$.

(d) By Remark \ref{48B77BA7-E2B4-4887-BF9D-B75658D0CC7C}, each element of $\DJL(B)$ is included in a maximal element.
This together with (b) and (c) implies that each element of $\LieLN_\bk(B)$ is included in a maximal element.

(a) Suppose that $\Lscr \in \LieLN_\bk(B)$ is relatively maximal and let $A = \ker(\Lscr)$;
so $\Lscr$ is a maximal element of the set $S = \setspec{ \Lscr' \in \LieLN_\bk(B) }{ \ker(\Lscr')=A }$.
By (b), we have $\Lscr \subseteq \Lscr'$ for some $\Lscr' \in \DJL(B) \cap S$.
By maximality of $\Lscr$, we have $\Lscr = \Lscr'$, so $\Lscr \in \DJL(B)$. This proves ``$\supseteq$'' in (a).
For the reverse inclusion, consider an arbitrary $\Lscr \in \DJL(B)$ and let $A = \ker(\Lscr)$.
Proceeding by contradiction, suppose that $\Lscr$ is not relatively maximal. 
Then we have a strict inclusion $\Lscr \subset \Lscr'$ for some $\Lscr' \in \LieLN_\bk(B)$ such that $\ker(\Lscr')=A$.
By (b), we have  $\Lscr' \subseteq \Lscr''$ for some $\Lscr'' \in \DJL(B)$ such that $\ker(\Lscr'')=A$.
So we have a strict inclusion $\Lscr \subset \Lscr''$ where $\Lscr, \Lscr'' \in \DJL(B)$ and $\ker(\Lscr'')=\ker(\Lscr)$,
which contradicts Remark \ref{5A4B8846-954B-45B2-AD99-00914B6F31F0}.
This shows that every element of $\DJL(B)$ is a relatively maximal element of $\LieLN_\bk(B)$.
\end{proof}

%%%%%%%%%%%%%%%%%%%%%%%%%%%%%%%%%%%%%%%%%%%%%%%%%%%%%%%%%%%%%%%%
%%%%%%%%%%%%%%%%%%%%%%%%%%%%%%%%%%%%%%%%%%%%%%%%%%%%%%%%%%%%%%%%

%%%%%%%%%%%%%%%%%%%%%%%%%%%%%%%%%%%%%%%%%%%%%%%%%%%%%%%%%%%%%%%%
%%%%%%%%%%%%%%%%%%%%%%%%%%%%%%%%%%%%%%%%%%%%%%%%%%%%%%%%%%%%%%%%

\section{Integral elements of $\DJL(B)$}
\label {SecintegralBorels}

\begin{definition}  \label {3D422524-4083-45E5-BA2D-D0736D2D3CF2}
Let $B$ be a domain of characteristic zero.
An element $\Lscr$ of $\DJL(B)$ is {\it integral\/} if there exists an admissible sequence $(D_1,\dots,D_m)$ of derivations of $B$
such that $\Lscr = L(D_1,\dots,D_m)$.
\end{definition}

The word ``integral'' in Definition \ref{3D422524-4083-45E5-BA2D-D0736D2D3CF2} suggests the absence of denominators
when elements of $\Lscr$ are expressed as linear combinations of $D_1,\dots,D_m$.
The main results of this section are:

\begin{proposition}\label{dim3Straight}
If $\bk$ is an algebraically closed field of characteristic zero and $B$ a factorial affine $\bk$-domain of dimension at most $3$
then each element of $\DJL(B)$ is integral.
\end{proposition}

\begin{proposition} \label {6BC46AAB-2493-4034-AE60-7FCC9264ECB8}
Let $\bk$ be a field of characteristic zero.
For each $n\ge4$, there exists a maximal element of $\DJL( \bk^{[n]})$ that is not integral.
\end{proposition}

We also give examples of $3$-dimensional smooth affine varieties $X$ such that  $\DJL( \bk[X] )$ contains non-integral maximal elements,
thus showing that the factoriality assumption in Proposition \ref{dim3Straight} cannot be entirely relaxed.

\begin{remark}  \label {214ED250-CA2F-4270-BC00-AEC2B0FADA07}
Let $(D_1,\dots,D_m)$ and $(E_1,\dots,E_n)$ be admissible sequences of derivations of a domain of characteristic zero.
If $\Lscr(D_1,\dots,D_m) = L(E_1,\dots,E_n)$, then $m=n$ by Lemma \ref{E881691D-8557-4672-B859-92E3341ADF4B}(b).
\end{remark}

\begin{lemma}  \label {3E5072F3-B726-4AEF-9B98-0AC1619775F2}
Let $B$ be a domain of characteristic zero and $(D_1,\dots,D_m)$ an admissible sequence of derivations of $B$.
\begin{enumerate}[\rm(a)]

\item $L(D_1,\dots,D_m) \in \DJL(B)$ if and only if $L(D_1,\dots,D_m) = \Lscr(D_1,\dots,D_m)$. 

\item If $L(D_1,\dots,D_m) = \Lscr(D_1,\dots,D_m)$ then $D_1, \dots, D_m$ are irreducible.

\end{enumerate}
\end{lemma}

\begin{proof}
(a) If $L(D_1,\dots,D_m) \in \DJL(B)$ then there exists an admissible sequence $(E_1,\dots,E_n)$ such that $L(D_1,\dots,D_m) = \Lscr(E_1,\dots,E_n)$.
Note that $m=n$ by Remark \ref{214ED250-CA2F-4270-BC00-AEC2B0FADA07};
this equality together with $\Lscr(E_1,\dots,E_n) \subseteq \Lscr(D_1,\dots,D_m)$ implies,
by Remark \ref{CEC40263-FB7A-4204-8EFF-32CC255344BB}(a), that $\Lscr(E_1,\dots,E_n) = \Lscr(D_1,\dots,D_m)$
and hence $L(D_1,\dots,D_m) = \Lscr(D_1,\dots,D_m)$.
This proves $(\Rightarrow)$, and the converse is obvious.

(b) Suppose that $L(D_1,\dots,D_m) = \Lscr(D_1,\dots,D_m)$.
Let $A_1 \subset \cdots \subset A_m$ be the kernel sequence of $(D_1,\dots,D_m)$.
Let $i \in \{1,\dots,m\}$ and suppose that $D_i = b \delta_i$ for some $b \in B$ and $\delta_i \in \Der(B)$.
Then $b \delta_i \neq 0$, so Lemma \ref{F45F968B-ECF9-45AC-9626-A51DA798D02B} implies that 
$b \in A_1$, because $b\delta_i$ commutes with each of $D_1, \dots, D_m$.
As $b \delta_i \in L(D_1,\dots,D_m)$ and $b \in A_1 \setminus \{0\}$, we have $\delta_i \in \Lscr(D_1,\dots,D_m) = L(D_1,\dots,D_m)$ and hence
$\delta_i = \sum_{j=1}^m a_j D_j$ with $a_j \in A_j$ for all $j=1,\dots,m$.
Thus, $D_i = b \delta_i = \sum_{j=1}^m (b a_j) D_j$;
as $D_1,\dots,D_m$ are linearly independent over $B$, we get $b a_i = 1$ and hence $b \in B^*$, showing that $D_i$ is irreducible.
\end{proof}

\begin{noname}  \label {DBC3814A-C9F4-4861-99E0-5A680C6D397E}
Recall the following facts (see, e.g., \cite[Section 2.1]{Freud:Book-new}).
Let $B$ be a domain of characteristic zero.
\begin{enumerate}[\rm(a)]

\item If $B$ satisfies the ascending chain condition (ACC) on principal ideals
then for each $D \in \lnd(B) \setminus \{0\}$ we have $D = a \delta$ for some $a \in \ker(D)$ and some irreducible $\delta \in \lnd(B)$.
(Recall that every UFD and every Noetherian ring satisfies the ACC on principal ideals.)

\item Suppose that $B$ is a UFD and let $A \in \klnd(B)$.
Then the nonzero $A$-module $\lnd_A(B)$ is cyclic, i.e., there exists $\delta \in \lnd_A(B)$ such that $\lnd_A(B) = A \delta$. 
Moreover, an element $\delta \in \lnd_A(B)$ satisfies $\lnd_A(B) = A \delta$ if and only if it is irreducible.

\end{enumerate}
\end{noname}

\begin{remark}  \label{0DFAFE39-F150-42CB-A719-FA7CFC5F62BE}
Let $B$ be a domain of characteristic zero that satisfies the ACC on principal ideals.
For each $\Lscr \in \DJL(B)$,
there exists an admissible sequence $(\delta_1,\dots,\delta_m)$ of derivations of $B$
such that $\delta_1,\dots,\delta_m$ are irreducible and $\Lscr = \Lscr(\delta_1,\dots,\delta_m)$.
(This follows from \ref{DBC3814A-C9F4-4861-99E0-5A680C6D397E}(a) and Corollary \ref{58D561A8-FF48-417A-9D8F-635BA02F689F}(b).)
\end{remark}

\begin{remark}  \label {5E82FF80-4404-441B-93BD-AB4A430ABB65}
Let $\bk$ be a field of characteristic zero, $B$ a $\bk$-domain with $\trdeg_\bk(B) = n < \infty$,
and $(D_1,\dots,D_m)$ an admissible sequence of derivations of $B$.
Assume that either $m=n$ or
$$
\text{$m=1$ and $B$ is a UFD.}
$$
Then $\Lscr(D_1,\dots,D_m)$ is integral.
Indeed, if $m=n$ then $A_1 = \bigcap_{i=1}^n \ker(D_i)$ is algebraic over $\bk$
(because $m = \trdeg_{A_1}(B)$ by Remark \ref{D9D95302-2C52-4CBD-A475-2C8367B13B8F}) and hence is a field;
so $\Lscr(D_1,\dots,D_m) = L(D_1,\dots,D_m)$.
Suppose that $m=1$ and $B$ is a UFD, and let $A_1 = \ker(D_1)$. 
We have $\Lscr(D_1) = \lnd_{A_1}(B)$ by Example \ref{B6D5B613-5754-4ECF-8B69-9CFB86A83A7D},
and \ref{DBC3814A-C9F4-4861-99E0-5A680C6D397E}(b) implies that there exists $\delta \in \lnd_{A_1}(B)$ such that $\lnd_{A_1}(B) = A_1 \delta$.
So $\Lscr(D_1) = A_1 \delta = L(\delta)$, which shows that $\Lscr(D_1)$ is integral.
\end{remark}

\begin{lemma}  \label {2345A6EB-195E-46CB-BBBF-A4CB58DA2639}
Let $\bk$ be an algebraically closed field of characteristic zero and $B$ a factorial affine $\bk$-domain of dimension $3$.
Suppose that $(D,E)$ is an admissible sequence of derivations of $B$ such that $E$ is irreducible.
Then there exists $D'$ such that $(D',E)$ is admissible and $\Lscr(D,E) = L(D',E)$.
\end{lemma}

\begin{proof}
The kernel sequence of $(D,E)$ is $A_1 \subset A_2$ where $A_1 = \ker(D) \cap \ker(E)$ and $A_2 = \ker(E)$.
Note that $A_1$ is affine over $\bk$ (because $\trdeg_\bk(A_1)=1$ and $A_1$ is contained in the affine $\bk$-domain $B$),
and is a PID (because it is factorially closed in $B$).
Since $\bk$ is algebraically closed, it follows that for every prime element $p \in A_1$
the composition $\bk \to A_1 \to A_1/pA_1$ is an isomorphism.
If $\alpha \in A_1 \setminus \{0\}$ and $\alpha = u p_1 \cdots p_n$ where $p_1,\dots,p_n$ are prime elements of $A_1$ and $u \in A_1^*$,
we define the {\it length\/} of $\alpha$ to be $\ell(\alpha) = n$ (so $\ell(u)=0$ for all $u \in A_1^*$).
Let $\Dscr$ be the set of $D' \in \Der(B)$ that satisfy the two conditions:
\begin{itemize}

\item $(D',E)$ is admissible,

\item $L(D,E) \subseteq L(D',E) \subseteq \Lscr(D',E) = \Lscr(D,E)$.

\end{itemize}
Note that $\Dscr \neq \emptyset$ (because $D \in \Dscr$) and that if $D' \in \Dscr$ then $(D',E)$ has kernel sequence $A_1 \subset A_2$
(the same as that of $(D,E)$).
Fix $\beta \in A_2 \setminus A_1$ and consider the set $\Omega = \setspec{ D'(\beta)B }{ D' \in \Dscr }$.
Then $\Omega \neq \emptyset$ and each element of $\Omega$ is a nonzero principal ideal of $B$.
Let us choose $D' \in \Dscr$ such that $D'(\beta)B$ is a maximal element of the poset $(\Omega,\subseteq)$
(we can do this because $B$ is Noetherian).

We claim that $L(D',E) = \Lscr(D',E)$.
Proceeding by contradiction, we assume that  $L(D',E) \neq \Lscr(D',E)$.
Then there exists $\delta_1 \in \Lscr(D',E) \setminus L(D',E)$.
Choose $\alpha \in A_1 \setminus \{0\}$ of minimal length such that $\alpha \delta_1 \in L(D',E)$,
and observe that $\ell(\alpha)>0$.
So $\alpha = p \alpha_1$ for some prime element $p$ of $A_1$ and some $\alpha_1 \in A_1 \setminus \{0\}$;
since $\ell(\alpha_1)<\ell(\alpha)$, we have $\alpha_1 \delta_1 \notin L(D',E)$.
Let $\delta = \alpha_1 \delta_1$ and note that
$$
\delta \in \Lscr(D',E) \setminus L(D',E) \quad \text{and} \quad p \delta \in L(D',E) .
$$
Write $p\delta = a_1 D' + a_2 E$ with $a_1 \in A_1$ and $a_2 \in A_2$.
We claim that $p \nmid a_1$ in $A_1$.
To see this, first note that $p$ is prime in $B$, because $A_1$ is factorially closed in $B$;
since $p \notin B^*$ and $E$ is irreducible, there exists $b \in B$ such that $E(b) \notin pB$.
If $p \mid a_1$ then the equation $p\delta(b) = a_1 D'(b) + a_2 E(b)$ shows that $a_2 E(b) \in pB$,
so $a_2 \in pB$; as $A_2$ is factorially closed in $B$, we then have $a_2 \in p A_2$
and consequently $\delta = \frac{a_1}{p} D' + \frac{a_2}{p} E \in L(D',E)$ is a contradiction.
So $p \nmid a_1$ in $A_1$;
since the composition $\bk \to A_1 \to A_1/pA_1$ is an isomorphism,
we have $a_1 = \mu + p a_1'$ for some $\mu \in \bk^*$ and $a_1' \in A_1$.
Define $D_1 = \delta -  a_1' D' \in \Der(B)$ and note that 
$p D_1  = p \delta - pa_1'D' = a_1 D' + a_2 E - pa_1'D'$, i.e., 
\begin{equation} \label {EC566A71-504F-4E00-8390-5B36AB4B458F}
p D_1 = \mu D' + a_2 E .
\end{equation}
It is easy to see that $(D_1,E)$ is admissible and that
$$
L(D',E) \subseteq L(D_1,E) \subseteq \Lscr(D_1,E) = \Lscr(D',E) ;
$$
so $D_1 \in \Dscr$ and consequently $D_1(\beta)B \in \Omega$.
Equation \eqref{EC566A71-504F-4E00-8390-5B36AB4B458F} also shows that $p D_1(\beta) = \mu D'(\beta)$,
so we have a strict inclusion $D'(\beta)B \subset D_1(\beta)B$ of elements of $\Omega$.
This contradicts the fact that $D'(\beta)B$ is a maximal element of $(\Omega,\subseteq)$.
This contradiction shows that  $L(D',E) = \Lscr(D',E)$, as required.
\end{proof}

\begin{noname} \label {C7BF28E5-4695-4BF8-9DA1-D3A730D28EAD}
{\bf Proof of Proposition \ref{dim3Straight}.}
We have to show that each element $\Lscr(D_1,\dots,D_m)$ of $\DJL(B)$ is integral.
By Remark \ref{5E82FF80-4404-441B-93BD-AB4A430ABB65}, we may assume that $\dim(B)=3$ and $m=2$.
So it suffices to show that if $(D_1,D_2)$ is an admissible sequence of derivations of $B$ then $\Lscr(D_1,D_2)$ is integral.
By Remark \ref{0DFAFE39-F150-42CB-A719-FA7CFC5F62BE},
there is an admissible sequence $(D,E)$ such that $\Lscr(D_1,D_2) = \Lscr(D,E)$ and $D,E$ are irreducible.
Lemma \ref{2345A6EB-195E-46CB-BBBF-A4CB58DA2639} implies that there exists $D'$ such that  $\Lscr(D,E) = L(D',E)$.
So $\Lscr(D_1,D_2)$ is integral.
\hfill\qedsymbol
\end{noname}

Our next objective is to give examples of maximal elements of $\DJL(B)$ that are not integral, when $B$ is a $\bk$-domain with good properties.

\begin{lemma}  \label {33206224-E440-46C6-BE36-BAB82E58772E}
Let $B$ be a domain of characteristic zero.
\begin{enumerate}[\rm(a)]

\item Let $(D_1,\dots,D_m)$ and $(E_1,\dots,E_n)$ be admissible sequences of derivations of $B$.
If \linebreak $\Lscr(D_1,\dots,D_m) \subseteq L(E_1,\dots,E_n)$
then $m \le n$ and $\Lscr(D_1,\dots,D_m) = L( E_{n-m+1}, \dots, E_n )$, so in particular $\Lscr(D_1,\dots,D_m)$ is integral.

\item Let $\Lscr_1, \Lscr_2 \in \DJL(B)$. If $\Lscr_1 \subseteq \Lscr_2$ and $\Lscr_2$ is integral, then $\Lscr_1$ is integral.

\item Assume that $B$ has finite transcendence degree over a field.
If $\DJL(B)$ contains a non-integral element, then it contains a non-integral maximal element.

\end{enumerate}
\end{lemma}

\begin{proof}
(a) Since $\Lscr(D_1,\dots,D_m) \subseteq L(E_1,\dots,E_n) \subseteq \Lscr(E_1,\dots,E_n)$,
Lemma \ref{00A23F46-9C1D-449A-BE41-68ACD0236AC5} implies that $m \le n$ and $\Lscr(D_1,\dots,D_m) = \Lscr( E_{n-m+1}, \dots, E_n )$.
Thus,
$$
\Lscr(D_1,\dots,D_m) = \Lscr( E_{n-m+1}, \dots, E_n ) \cap L(E_1,\dots,E_n) .
$$
To finish the proof of (a), it therefore suffices to show that 
\begin{equation} \label {B3C9E005-6521-448B-A3B5-30BBBFC3B98E}
\Lscr( E_{n-m+1}, \dots, E_n ) \cap L(E_1,\dots,E_n) = L( E_{n-m+1}, \dots, E_n ) ,
\end{equation}
where ``$\supseteq$'' is obvious.
Let $B_1 \subset \cdots \subset B_{n}$ be the kernel sequence of $(E_1,\dots,E_n)$.
If $\delta$ belongs to the left side of \eqref{B3C9E005-6521-448B-A3B5-30BBBFC3B98E} then $\delta = \sum_{i=1}^n a_i E_i$
with $a_i \in B_i$ for all $i = 1,\dots,n$,
and for some $\alpha \in B_{n-m+1} \setminus \{0\}$ 
we have $\alpha \delta = \sum_{i=n-m+1}^n \hat a_i E_i$ with $\hat a_i \in B_i$ for all $i=n-m+1, \dots, n$;
since $\alpha\delta =  \sum_{i=1}^n (\alpha a_i) E_i$ and $E_1,\dots,E_n$ are linearly independent over $B$,
we obtain $\alpha a_i = 0$ (so $a_i=0$) for all $i < n-m+1$,
so $\delta = \sum_{i=n-m+1}^n a_i E_i \in L( E_{n-m+1}, \dots, E_n )$.
This proves the inclusion ``$\subseteq$'' in \eqref{B3C9E005-6521-448B-A3B5-30BBBFC3B98E}.
It follows that $\Lscr(D_1,\dots,D_m) = L( E_{n-m+1}, \dots, E_n )$, so $\Lscr(D_1,\dots,D_m)$ is integral.
This proves (a), and (b) easily follows. 
If $B$ has finite transcendence degree over a field then Remark \ref{48B77BA7-E2B4-4887-BF9D-B75658D0CC7C} guarantees
that each element of $\DJL(B)$ is included in a maximal element.
This together with (b) implies (c).
\end{proof}

\begin{remark}  \label {3B19CFB2-4364-4D68-B7B0-CB01D06CC1A6}
Let $B$ be a domain of characteristic zero and $A \in \klnd(B)$.
Then $\lnd_A(B)$ is a nonzero $A$-module and Example \ref{B6D5B613-5754-4ECF-8B69-9CFB86A83A7D} implies that
$\lnd_A(B) \in \DJL(B)$. Clearly,
$$
\text{$\lnd_A(B)$ is an integral element of $\DJL(B)$ if and only if it is a cyclic $A$-module.}
$$
\end{remark}

\begin{lemma}  \label {2423D678-4B8A-4F44-831B-268ACD0A67EA}
Suppose that $R$ is a domain of characteristic zero and that $p,q,a,b \in R$ satisfy:
\begin{equation} \label {3310A422-FE1C-43BE-B309-BC0698363C53}
\tag{$*$}
\text{$p$ and $q$ are irreducible and not associates, \quad $pa = qb$, \quad $p \nmid b$ \quad and \quad $q \nmid a$.}
\end{equation}
Let $B = R[x,y] = R^{[2]}$ and define $D, E \in \Der_R(B)$ by
$$
D = a \partial_x + q \partial_y \quad \text{and} \quad  E = b \partial_x + p \partial_y .
$$
Then $D$ and $E$ are locally nilpotent and irreducible, $\ker(D) = \ker(E)$, but no unit $u \in B^*$ is such that $uD=E$.
Consequently, if we define $A = \ker(D) = \ker(E)$ then the $A$-module $\lnd_A(B)$ is not cyclic.
\end{lemma}

\begin{proof}
Since $D^2(x) = D^2(y) = E^2(x) = E^2(y) = 0$, $D$ and $E$ are locally nilpotent.
If $c \in B$ is such that $D(B) \subseteq cB$ then $c$ is a common divisor of $D(x) = a$ and $D(y)= q$ in $B$,
so $c \mid a$ and $c \mid q$ in $R$; since $q$ is irreducible in $R$ and $q \nmid a$ in $R$, it follows that $c \in R^*$.
This shows that $D$ is irreducible, and a similar argument shows that $E$ is irreducible.
As $pD = qE$, we have $\ker(D) = \ker(E)$.
If $u \in B^*$ ($=R^*$) satisfies $uD=E$ then $uq = uD(y) = E(y) = p$, which contradicts the fact that $p,q$ are not associates in $R$.
So no unit $u \in B^*$ is such that $uD=E$.
The fact that $\lnd_A(B)$ is not cyclic easily follows.
\end{proof}

\begin{corollary}\label{nonstrightdim3}
Let $\bk$ be a field of characteristic zero, $R$ an affine $\bk$-domain that is not a UFD, and $B = R^{[2]}$. 
\begin{enumerate}[\rm(a)]

\item There exists $A \in \klnd(B)$ such that the $A$-module $\lnd_A(B)$ is not cyclic.

\item There exists a maximal element of $\DJL(B)$ that is not integral.

\end{enumerate}
\end{corollary}

\begin{proof}
Since $R$ is a Noetherian domain, each nonzero nonunit element of $R$ is a product of irreducible elements.
As $R$ is not a UFD, there exists an equation $p_1 \cdots p_m = q_1 \cdots q_n$ in $R$
with $m,n \ge 2$, all $p_i, q_j$ irreducible and no $p_i$ associate to any $q_j$.
Since $p_1 \mid q_1 \cdots q_{n}$ and  $p_1 \nmid q_1$,
there exists a unique $j \in \{2,\dots,n\}$ such that $p_1 \mid q_1 \cdots q_{j}$ and  $p_1 \nmid q_1 \cdots q_{j-1}$.
Let $a = q_1 \cdots q_{j}/p_1 \in R$ and $b = q_1 \cdots q_{j-1} \in R$.
Then $p_1 a = q_j b$, and since $p_1 \nmid b$ it follows that $q_j \nmid a$.
So, with $p = p_1$ and $q = q_j$, the elements $p,q,a,b \in R$ satisfy the condition \eqref{3310A422-FE1C-43BE-B309-BC0698363C53}
of Lemma \ref{2423D678-4B8A-4F44-831B-268ACD0A67EA}.
Since $B = R^{[2]}$, that Lemma implies that there exists $A \in \klnd(B)$ such that $\lnd_A(B)$ is not cyclic. This proves (a).
By  Remark \ref{3B19CFB2-4364-4D68-B7B0-CB01D06CC1A6}, $\lnd_A(B)$ is a non-integral element of $\DJL(B)$, so (b) follows from Lemma \ref{33206224-E440-46C6-BE36-BAB82E58772E}(c).
\end{proof}

\begin{example} \label {nonIntegralExample}
There exist smooth affine varieties $X$ of dimension $3$ such that $\DJL( \bk[X] )$ has a non-integral maximal element.
For instance, let $X = C \times \Aff^{2}$ where $C$ is a non-rational smooth affine curve.
Since $\bk[C]$ is not a UFD, Corollary \ref{nonstrightdim3} implies that $\DJL( \bk[X] )$ has a non-integral maximal element.
\end{example}

\medskip

\begin{center}
\textbf{Non-integral elements of $\DJL(\bk^{[n]})$, $n \geq 4$.}
\end{center}

\smallskip

The following Example proves the case $n=4$ of Proposition \ref{6BC46AAB-2493-4034-AE60-7FCC9264ECB8}.

\begin{example} \label {C60B4C67-6814-44F3-8DDF-E43EEF900AFA}
Let $\bk$ be a field of characteristic zero and $B = \bk[u,v,y,z] = \bk^{[4]}$.
We give examples of maximal elements of $\DJL(B)$ that are not integral and have derived length $3$.

Consider the derivations $D = \partial_y$ and $E = u \partial_z - v \partial_y$ of Example \ref{BF979604-4D5B-4A13-A6AA-E9EDC00540C7}.
Recall that $\Lscr(D,E) \in \DJL(B)$ and that the kernel sequence of $\Lscr(D,E)$ is $A_1=\bk[u,v] \subset A_2=\bk[u,v,Q]$, where  $Q = uy+vz$.
We claim:
\begin{equation}  \label {1C38309B-7681-4D98-83B1-071614B914AD}
\text{the element $\Lscr(D,E)$ of $\DJL(B)$ is not integral.}
\end{equation}
Arguing by contradiction, suppose that $\Lscr(D,E)$ is integral, i.e., that there exists an admissible sequence $(D',E')$ of derivations of $B$
such that $\Lscr(D,E) = L(D',E')$.
By Lemma \ref{00A23F46-9C1D-449A-BE41-68ACD0236AC5}, $(D,E)$ and $(D',E')$ have the same kernel sequence. 
In particular, $\ker(E') = \ker(E)$.
Since $E$ is irreducible, $E' = h E$ for some $h \in \ker(E)$.
As $\partial_y = D \in \Lscr(D,E) = L(D',E')$, we have $\partial_y = a D' + b E'$ for some $a \in A_1$ and $b \in A_2$.
Since $u \partial_z = v \partial_y + E \in L(D,E)$ and $u \in A_1 \setminus \{0\}$,
we have $\partial_z \in \Lscr(D,E) = L(D',E')$ and hence $\partial_z = c D' + d E'$ for some $c \in A_1$ and $d \in A_2$.
Thus,
\begin{align*}
\begin{pmatrix} 1 & 0 \\ 0 & 1 \end{pmatrix}
&= \begin{pmatrix} \partial_y(y) & \partial_y(z) \\ \partial_z(y) & \partial_z(z) \end{pmatrix}
= \begin{pmatrix} a D'(y) + b E'(y) & a D'(z) + b E'(z) \\ c D'(y) + d E'(y) & c D'(z) + d E'(z) \end{pmatrix} \\
&= \begin{pmatrix} a & b \\ c & d \end{pmatrix} \begin{pmatrix} D'(y) & D'(z) \\ E'(y) & E'(z) \end{pmatrix}
= \begin{pmatrix} a & b \\ c & d \end{pmatrix} \begin{pmatrix} D'(y) & D'(z) \\ -hv & hu \end{pmatrix} .
\end{align*}
Taking determinants gives $1 = (ad-bc)(D'(y)u + D'(z)v)h$ where $a,b,c,d,h, D'(y), D'(z) \in \bk[u,v,y,z]$,
so $1$ belongs to the ideal $(u,v)$ of $\bk[u,v,y,z]$, a contradiction.  
This proves \eqref{1C38309B-7681-4D98-83B1-071614B914AD}.

Example \ref{BF979604-4D5B-4A13-A6AA-E9EDC00540C7} shows that $\Lscr(D,E)$ is not maximal in $\DJL(B)$
and that each one of its immediate successors is maximal and has derived length $3$.
It follows from \eqref{1C38309B-7681-4D98-83B1-071614B914AD} and Lemma \ref{33206224-E440-46C6-BE36-BAB82E58772E}(b)
that each immediate successor of $\Lscr(D,E)$  is a maximal element of $\DJL(B)$ that is not integral.

The immediate successors of $\Lscr(D,E)$ are described in Example \ref{BF979604-4D5B-4A13-A6AA-E9EDC00540C7},
but let us display one of them as a concrete example.
If $\Agoth = (A_1,A_2)$ denotes the kernel sequence of $\Lscr(D,E)$
then assertion (a) of Example \ref{BF979604-4D5B-4A13-A6AA-E9EDC00540C7} shows that $\bk[u] \in \Omega(\Agoth)$,
so assertion (b) implies that $\Lscr( \bk[u], A_1, A_2 )$ is an immediate successor of $\Lscr(D,E)$.
For the benefit of readers who prefer to specify an element of $\DJL(B)$ by giving an admissible sequence,
we note that  $\Lscr( \bk[u], A_1, A_2) = \Lscr(F,D,E)$ where $F = z \partial_y - u \partial_v$.
Hence, $\Lscr(F,D,E)$ is a maximal element of $\DJL(B)$ that is not integral.
\end{example}

\smallskip

The following trivial observation is used in the proof of Lemma \ref{A7B19E49-A7C3-46A5-B235-23A05DBE5BC7}:

\begin{noname} \label {2C2C5C78-9245-4643-8CF6-03695085BD9F}
\it Let $B$ be a domain of characteristic zero and $D_1,\dots,D_m, D'_1, \dots, D'_m \in \Der(B)$.
Suppose that $D_1, \dots, D_m$ belong to the $B$-submodule of $\Der(B)$ generated by $D'_1, \dots, D'_m$
and that $D_1, \dots, D_m$ are linearly independent over $B$. Then $D'_1, \dots, D'_m$ are linearly independent over $B$. 
\end{noname}

% \begin{proof}
% Let $K = \Frac B$ and extend $D_1,\dots,D_m, D'_1, \dots, D'_m$ to $\hat D_1,\dots,\hat D_m, \hat D'_1, \dots, \hat D'_m \in \Der(K)$.
% Then $\Der(K)$ is a vector space over $K$;
% since the vectors $\hat D_1,\dots,\hat D_m$ are linearly independent and belong to the subspace spanned by $\hat D'_1, \dots, \hat D'_m$,
% it follows that $\hat D'_1, \dots, \hat D'_m$ are linearly independent over $K$
% and hence that $D'_1, \dots, D'_m$ are linearly independent over $B$. 
% \end{proof}

\begin{noname} \label {C2B8A618-47F7-474D-B819-5799968DAAB5}
We introduce the notation for Lemma \ref{A7B19E49-A7C3-46A5-B235-23A05DBE5BC7} and its proof.
Let $B$ be a domain of characteristic zero and $\tilde B = B[x_1,\dots,x_n] = B^{[n]}$.
If $A$ is a subring of $B$, let $\tilde A$ be the subring $A[x_1,\dots,x_n]$ of $\tilde B$.
Let $\Der_x( \tilde B ) = \setspec{ D \in \Der(\tilde B) }{ D(x_1) = \cdots = D(x_n) = 0 }$.
Each $D \in \Der(B)$ extends uniquely to an element $\tilde D$ of $\Der_x(\tilde B)$,
and the map
$$
\Der(B) \to \Der_x(\tilde B), \quad D \mapsto \tilde D
$$
has the following properties:
\begin{itemize}

\item if $b \in B$ and $D \in \Der(B)$ then $\widetilde{bD} = b \, \tilde D$;

\item if $D_1,D_2 \in \Der(B)$ then $(D_1 + D_2){\mbox{\!\large$\widetilde{\ \ }$}} = \tilde D_1 + \tilde D_2$
and $[D_1,D_2]{\mbox{\large$\widetilde{\ \ }$}} = [\tilde D_1, \tilde D_2]$;

\item if $D \in \lnd(B)$ then $\tilde D \in \lnd(\tilde B)$.

\end{itemize}
There is also a map
\begin{equation*}
\Der_x(\tilde B) \to \Der(B), \quad E \mapsto \bar E , 
\end{equation*}
where given $E \in \Der_x(\tilde B)$ we define $\bar E$ to be the composition $B \xrightarrow{\nu} \tilde B \xrightarrow{E} \tilde B \xrightarrow{\pi} B$,
where $\nu$ is the inclusion map and $\pi$ sends $f(x_1,\dots,x_n)$ to $f(0,\dots,0)$.
The map $E \mapsto \bar E$ satisfies:
\begin{itemize}

\item if $E \in \Der_x(\tilde B)$ then $\pi \circ E = \bar E \circ \pi$;

\item if $f \in \tilde B$ and $E \in \Der_x(\tilde B)$ then $\overline{fE} = \pi(f) \bar E$;

\item if $E_1,E_2 \in \Der_x(\tilde B)$ then $\overline{E_1+E_2}= \bar E_1 + \bar E_2$ and $\overline{[E_1,E_2]} = [ \bar E_1, \bar E_2 ]$,

\item if $E \in \Der_x(\tilde B) \cap \lnd(\tilde B)$ then $\bar E \in \lnd(B)$.

\end{itemize}
Finally, if $D \in \Der(B)$ then $\overline{ \tilde D } = D$.
\end{noname}

\begin{lemma} \label {A7B19E49-A7C3-46A5-B235-23A05DBE5BC7}
Let $B$ be a domain of characteristic zero and
$(D_1,\dots,D_m)$ an admissible sequence of derivations of $B$ with kernel sequence $A_1 \subset \cdots \subset A_{m}$.
Let $\tilde B = B[x_1,\dots,x_n] = B^{[n]}$ and let us use the notation of \ref{C2B8A618-47F7-474D-B819-5799968DAAB5}.
\begin{enumerate}[\rm(a)]

\item $(\tilde D_1, \dots, \tilde D_m)$ is an admissible sequence of derivations of $\tilde B$
and $\tilde A_1 \subset \cdots \subset \tilde A_{m}$ is its kernel sequence.

\item If $\Lscr(\tilde D_1, \dots, \tilde D_m)$ is integral then so is $\Lscr(D_1,\dots,D_m)$.

\end{enumerate}
\end{lemma}

\begin{proof}
Verification of (a) is straightforward.
For (b), suppose that $\Lscr(\tilde D_1, \dots, \tilde D_m)$ is integral.
Then $\Lscr(\tilde D_1, \dots, \tilde D_m) = L(E_1,\dots,E_n)$ for some admissible sequence $(E_1,\dots,E_n)$ of derivations of $\tilde B$.
Remark \ref{214ED250-CA2F-4270-BC00-AEC2B0FADA07} gives $m=n$,
so Remark \ref{CEC40263-FB7A-4204-8EFF-32CC255344BB}(b) (or Lemma \ref{3E5072F3-B726-4AEF-9B98-0AC1619775F2}(a)) gives
\begin{equation} \label {11FD99C2-2FDF-4F6D-9BB9-3649D45F7855}
\Lscr(\tilde D_1, \dots, \tilde D_m) = L(E_1,\dots,E_m) = \Lscr(E_1,\dots,E_m).
\end{equation}
By Proposition \ref{6CFC46F0-0103-4BE1-8ACB-BDA7B53CB093}, $(\tilde D_1, \dots, \tilde D_m)$ and $(E_1,\dots,E_m)$ have the same kernel sequence; so
$$
\text{the kernel sequence of $(E_1,\dots,E_m)$ is $\tilde A_1 \subset \cdots \subset \tilde A_{m}$.}
$$
Since $x_1,\dots,x_n \in \tilde A_1$, it follows in particular that $E_i(x_j)=0$ for all $i \in \{1,\dots,m\}$ and $j \in \{1,\dots,n\}$,
so $E_1,\dots,E_m \in \Der_x(\tilde B) \cap \lnd(\tilde B)$;
consequently, $\bar E_1, \dots, \bar E_m \in \lnd(B)$ are defined.
Since $E \mapsto \bar E$ preserves the Lie bracket, $\bar E_1, \dots, \bar E_m$ pairwise commute.
To show that $(\bar E_1, \dots, \bar E_m)$ is an admissible sequence of derivations of $B$, it only remains to check that 
\begin{equation} \label {AAC0D747-4DC2-4966-A852-F2415DAD5125}
\text{$\bar E_1, \dots, \bar E_m$ are linearly independent over $B$.}
\end{equation}
We will prove \eqref{AAC0D747-4DC2-4966-A852-F2415DAD5125} later. Consider an arbitrary $\delta \in \Lscr(D_1,\dots,D_m)$.
There exists $\alpha \in A_1 \setminus \{0\}$ such that $\alpha \delta \in L(D_1,\dots,D_m)$; write $\alpha \delta = \sum_{i=1}^m a_i D_i$ 
with $a_i \in A_i$ for all $i=1,\dots,m$. Since
$$
\textstyle
\alpha \tilde\delta = \widetilde{ \alpha \delta }
= (\sum_{i=1}^m a_i D_i){\mbox{\!\large$\widetilde{\ \ }$}}
= \sum_{i=1}^m a_i \tilde D_i
$$
and $a_i \in A_i \subseteq \tilde A_i$ for $i=1,\dots,m$, we get
$\alpha \tilde\delta \in L(\tilde D_1, \dots, \tilde D_m) \subseteq \Lscr(\tilde D_1, \dots, \tilde D_m) = L(E_1,\dots,E_m)$;
since $\alpha \in A_1 \setminus \{0\} \subseteq \tilde A_1 \setminus \{0\}$, it therefore follows that 
$\tilde\delta \in \Lscr(E_1,\dots,E_m) = L(E_1,\dots,E_m)$, the last equality by \eqref{11FD99C2-2FDF-4F6D-9BB9-3649D45F7855}.
So we can write $\tilde\delta = \sum_{i=1}^m f_i E_i$ with $f_i \in \tilde A_i$ for all $i=1,\dots,m$.
Consequently,
$$
\textstyle
\delta = \overline{ \tilde \delta } = \overline{ \sum_{i=1}^m f_i E_i } = \sum_{i=1}^m \pi(f_i) \bar E_i .
$$
It is obvious that $\pi( \tilde A_i ) = A_i$, so $\pi(f_i) \in A_i$ for all $i=1,\dots,m$, so $\delta \in \sum_{i=1}^m A_i \bar E_i$.
This shows that
\begin{equation}  \label {65ACA082-A0AE-46DD-AB60-304C13C83A14}
\textstyle
\Lscr(D_1,\dots,D_m) \subseteq \sum_{i=1}^m A_i \bar E_i .
\end{equation}
In particular, $D_1,\dots,D_m \in \sum_{i=1}^m A_i \bar E_i$; as $D_1,\dots,D_m$ are linearly independent over $B$,
the fact stated in \ref{2C2C5C78-9245-4643-8CF6-03695085BD9F} implies that \eqref{AAC0D747-4DC2-4966-A852-F2415DAD5125} is true; so 
$$
\text{$(\bar E_1, \dots, \bar E_m)$ is an admissible sequence of derivations of $B$.}
$$
For each $i$, we have $\pi \circ E_i = \bar E_i \circ \pi$, so $\pi( \ker E_i ) \subseteq \ker \bar E_i$.
If $B_1 \subset \cdots \subset B_{m}$ denotes the kernel sequence of 
$(\bar E_1, \dots, \bar E_m)$ then, for each $j \in \{1,\dots,m\}$,
$$
\textstyle
A_j = \pi\big( \tilde A_j \big) = \pi\big( \bigcap_{i=j}^{m} \ker E_i \big) 
\subseteq \bigcap_{i=j}^{m} \pi( \ker E_i ) 
\subseteq \bigcap_{i=j}^{m} \ker \bar E_i = B_j .
$$
We have $\trdeg_{A_j}(B) = m+1-j = \trdeg_{B_j}(B)$ by Lemma  \ref{BAD0996F-5C99-407B-B702-0A97D9510736}(c), 
so $B_j$ is algebraic over $A_j$; since $A_j$ is algebraically closed in $B$, we therefore have $B_j=A_j$ for
all $j = 1,\dots,m$. So the kernel sequence of $(\bar E_1, \dots, \bar E_m)$ is $A_1 \subset \cdots \subset A_{m}$,
and \eqref{65ACA082-A0AE-46DD-AB60-304C13C83A14} means that $\Lscr(D_1,\dots,D_m) \subseteq L(\bar E_1, \dots, \bar E_m)$.
We obtain $\Lscr(D_1,\dots,D_m) = L(\bar E_1, \dots, \bar E_m)$ by Remark \ref{CEC40263-FB7A-4204-8EFF-32CC255344BB}(b),
so $\Lscr(D_1,\dots,D_m)$ is integral.
\end{proof}

Lemma \ref{A7B19E49-A7C3-46A5-B235-23A05DBE5BC7} has the following obvious consequence:

\begin{corollary}  \label {B9348BEF-D82C-4F54-92EA-A6A04ADA2B2B}
Let $B$ be a domain of characteristic zero and $\Lscr$ a non-integral element of $\DJL(B)$.
For each $n \in \Nat$, $\DJL(B^{[n]})$ contains a non-integral element whose derived length equals that of $\Lscr$.
\end{corollary}

\begin{noname}
{\bf Proof of Proposition \ref{6BC46AAB-2493-4034-AE60-7FCC9264ECB8}.}
Let $n\ge4$.
Since $\DJL( \kk4 )$ has a non-integral element by Example \ref{C60B4C67-6814-44F3-8DDF-E43EEF900AFA}, 
it follows from Corollary \ref{B9348BEF-D82C-4F54-92EA-A6A04ADA2B2B} that $\DJL( \kk n )$ has a non-integral element
and hence a non-integral maximal element by Lemma \ref{33206224-E440-46C6-BE36-BAB82E58772E}(c).
\hfill\qedsymbol
\end{noname}

%%%%%%%%%%%%%%%%%%%%%%%%%%%%%%%%%%%%%%%%%%%%%%%%%%%%%%%%%%%%%%%%%%%%%%%%%%%%%%%%%%%%%%%%%%%%%%%%%%%%%%%%%%%%%%%%
%%%%%%%%%%%%%%%%%%%%%%%%%%%%%%%%%%%%%%%%%%%%%%%%%%%%%%%%%%%%%%%%%%%%%%%%%%%%%%%%%%%%%%%%%%%%%%%%%%%%%%%%%%%%%%%%
%%%%%%%%%%%%%%%%%%%%%%%%%%%%%%%%%%%%%%%%%%%%%%%%%%%%%%%%%%%%%%%%%%%%%%%%%%%%%%%%%%%%%%%%%%%%%%%%%%%%%%%%%%%%%%%%

\section{Gradings}  \label {Liealggradsec}

Throughout this section, $\bk$ is a field of characteristic zero and $B$ is an affine $\bk$-domain.

\begin{noname} \label {8736egd9fv092b5}
Let $(G,+)$ be an abelian group.
A {\it grading\/} of $B$ by $G$ (or a {\it $G$-grading\/} of $B$) is a family $\big( B_i \big)_{i \in G}$ of subgroups of $(B,+)$ satisfying $B = \bigoplus_{i \in G} B_i$, 
$B_iB_j\subseteq B_{i+j}$ for all $i,j \in G$, and $\bk \subseteq B_0$.
Suppose that $\ggoth = \big( B_i \big)_{i \in G}$ is a $G$-grading of $B$.
\begin{enumerate}

\item An element of $B$ is {\it $\ggoth$-homogeneous\/} if it belongs to $\bigcup_{i \in G} B_i$.
If $x$ is a nonzero $\ggoth$-homogeneous element of $B$ then $\deg_\ggoth(x)$ denotes the unique $i \in G$ such that $x \in B_i$.  Given $x \in B$, the nonzero $x_i$ that appear in the expression $x = \sum_{i \in G} x_i$ (with $x_i \in B_i$ for all $i$)
are called the {\it $\ggoth$-homogeneous components\/} of $x$.
A subset $A$ of $B$ is called {\it $\ggoth$-homogeneous\/} if, for each $a \in A$, all $\ggoth$-homogeneous components of $a$ belong to $A$.

\item The $G$-grading $\ggoth$ induces a $G$-grading on the $B$-module $\Der_\bk(B)$, given by
$$
\textstyle
\Der_\bk(B) = \bigoplus_{ d \in G } \Der_\bk(B,\ggoth)_d
$$
where for each $d \in G$ we define 
$$
\Der_\bk(B,\ggoth)_d = \setspec{ D \in \Der_\bk(B) }{ \text{$D(B_i) \subseteq B_{i+d}$ for all $i \in G$} } .
$$
An element of $\Der_\bk(B)$ is {\it $\ggoth$-homogeneous\/} if it belongs to  $\bigcup_{ d \in G } \Der_\bk(B,\ggoth)_d$.
If $D \in \Der_\bk(B)$ is nonzero and $\ggoth$-homogeneous then $\deg_\ggoth(D)$ is the unique $d \in G$ such that $D \in \Der_\bk(B,\ggoth)_d$.
We define the set $\hlnd(B,\ggoth) = \setspec{ D \in \lnd(B) }{ \text{$D$ is $\ggoth$-homogeneous} }$.

Consider $D \in \Der_\bk(B)$ and write $D = \sum_{d \in G} D_d$ with $D_d \in \Der_\bk(B,\ggoth)_d$ for all $d$.
We define the finite subset  $\Supp_\ggoth(D) = \setspec{ d \in G }{ D_d \neq 0 }$ of $G$.
The $D_d$ with $d \in \Supp_\ggoth(D)$ are called the {\it $\ggoth$-homogeneous components\/} of $D$
and the expression $D = \sum_{ d \in \Supp_\ggoth(D) } D_d$ is called the {\it $\ggoth$-homogeneous decomposition\/} of $D$.

A subset $\Lscr$ of $\Der_\bk(B)$ is called {\it $\ggoth$-homogeneous\/} if,
for each $D \in \Lscr$, all $\ggoth$-homogeneous components of $D$ belong to $\Lscr$.
We write $\DJL(B,\ggoth) = \setspec{ \Lscr \in \DJL(B) }{ \text{$\Lscr$ is $\ggoth$-homogeneous} }$
and $\LieLN_\bk(B,\ggoth) = \setspec{ \Lscr \in \LieLN_\bk(B) }{ \text{$\Lscr$ is $\ggoth$-homogeneous} }$.

\end{enumerate}
\end{noname}

\begin{remark} \label {51C3D8B8-B46A-47AD-A51F-3A83DCCEF6CA}
An abelian group $G$ is said to be {\it orderable\/} if
there exists a total order $\le$ on the set $G$ such that $(G,\le)$ is a totally ordered abelian group.
It is known that an abelian group is orderable if and only if it is torsion-free.
In particular, $\Integ^n$ is orderable.
\end{remark}

\begin{lemma}\label{homogeneousPrelim}
Let $\ggoth$ be a grading of $B$ by a torsion-free abelian group.
\begin{enumerate}[\rm(a)]
\item If $x,y \in B \setminus \{0\}$ and $xy$ is $\ggoth$-homogeneous then $x$ and $y$ are $\ggoth$-homogeneous.
\item If $x \in B \setminus \{0\}$ and $D \in \Der_\bk(B) \setminus \{0\}$ are such that $xD$ is $\ggoth$-homogeneous, then $x$  and $D$ are $\ggoth$-homogeneous.
\end{enumerate}
\end{lemma}

\begin{proof}
Let $\ggoth = (B_i)_{ i \in G }$.
By Remark \ref{51C3D8B8-B46A-47AD-A51F-3A83DCCEF6CA}, we can choose an order $\le$ such that $(G,\le)$ a totally ordered abelian group.
Given an element $x = \sum_{i \in G} x_i$ of $B \setminus \{0\}$ (with $x_i \in B_i$ for all $i$),
note that $S_x := \setspec{ i \in G }{ x_i \neq 0 }$ is a nonempty finite subset of $G$ and define $w(x) = \max S_x - \min S_x$.
Given $D \in \Der_\bk(B) \setminus \{0\}$, define $W(D) = \max \Supp_\ggoth(D) - \min \Supp_\ggoth(D)$.
This defines maps $w : B \setminus \{0\} \to G_{\ge0}$ and $W : \Der_\bk(B) \setminus \{0\} \to G_{\ge0}$,
where $G_{\ge0} = \setspec{ i \in G }{ i \ge 0 }$,
and one can check that $w(xy) = w(x) + w(y)$ and $W(xD) = w(x) + W(D)$ 
for all $x,y \in B \setminus \{0\}$ and $D \in \Der_\bk(B) \setminus \{0\}$.
So $w(xy)=0$ implies $w(x)=0=w(y)$ and $W(xD)=0$ implies $w(x)=0=W(D)$, as desired.
\end{proof}

\begin{definition} \label {C3246E05-854F-427C-828A-669E2281D276}
Let $(G,\le)$ be a totally ordered abelian group and let $\ggoth = (B_i)_{i \in G}$ be a $G$-grading of $B$. Given $x = \sum_{i \in G} x_i \in B$ ($x_i \in B_i$ for all $i$),
we define $\deg_\ggoth(x) = \max\setspec{ i \in G }{ x_i \neq 0 }$ if $x \neq 0$, and $\deg_\ggoth(x) = -\infty$ if $x = 0$.
Given $D = \sum_{i \in G} D_i \in \Der_\bk(B)$ ($D_i \in \Der_\bk(B,\ggoth)_i$ for all $i$),
we define $\deg_\ggoth(D) = \max \setspec{ i \in G }{ D_i \neq 0 }$ if $D \neq 0$,
and $\deg_\ggoth(D)=-\infty$ if $D=0$.
This defines two maps, $\deg_\ggoth : B \to G \cup \{-\infty\}$ and $\deg_\ggoth : \Der_\bk(B) \to G \cup \{-\infty\}$.

Given a nonzero $D = \sum_{i \in G} D_i \in \Der_\bk(B)$  ($D_i \in \Der_\bk(B,\ggoth)_i$ for all $i$),
we define the {\it highest homogeneous component\/} $\hat D$ of $D$
by $\hat D = D_d$ where $d = \deg_\ggoth(D)$. If $D=0$, we define $\hat D=0$.
It is clear that $D \neq 0$ implies $\hat D \neq 0$, and Lemma \ref{E7261F32-B2ED-4D50-8F3E-E694DB49CF21} recovers the well-known fact
that if $D \in \lnd(B)$ then $\hat D \in \hlnd(B,\ggoth)$. 
\end{definition}

\begin{lemma}  \label {E7261F32-B2ED-4D50-8F3E-E694DB49CF21}
Let $\ggoth$ be a grading of $B$ by a totally ordered abelian group $G$.
Given $D \in \Der_\bk(B)$, let $\hat D$ denote the highest homogeneous component of $D$.
If $S$ is a locally nilpotent subset of $\Der(B)$ then so is $\setspec{ \hat D }{ D \in S }$.
\end{lemma}

\begin{proof}
See Section \ref{sectionLocallyNilpotentandLocallyFinite} for the concept of a locally nilpotent subset of $\Der(B)$.
Let $(D_1,D_2,\dots)$ be an infinite sequence of elements of $\setspec{ \hat D }{ D \in S }$
and let $b$ be a homogeneous element of $B$.
It suffices to show that $(D_n D_{n-1} \cdots D_1)(b)=0$ for $n$ large enough.
We may assume that $b \neq 0$ and that $D_i \neq 0$ for all $i$, otherwise the claim is trivial.

For each $i\ge1$, choose $D_i' \in S$ such that $\widehat{D_i'} = D_i$ and define
$d_i = \deg_\ggoth(D_i') = \deg_\ggoth(D_i)$ and $\dgoth_i = \deg_\ggoth(b) + \sum_{j=1}^i d_j$.
For each $g \in G$, let $p_g : B \to B_g$ be the canonical projection (we use the notation $B = \bigoplus_{g \in G}B_g$ for $\ggoth$).
It is easily proved by induction on $i$
that $\deg_\ggoth[ (D_i' D_{i-1}' \cdots D_1')(b) ] \le \dgoth_i$ for all $i \ge1$, and that
\begin{equation}  \label {CE4E4194-924C-4DBA-8061-AFEFCBDD57BC}
p_{\dgoth_i} [ (D_i' D_{i-1}' \cdots D_1')(b) ] = (D_i D_{i-1} \cdots D_1)(b) \in B_{\dgoth_i} \quad \text{for all $i \ge1$.}
\end{equation}
Since $S$ is a locally nilpotent subset of $\Der(B)$,
there exists $n\ge1$ such that $(D_n' D_{n-1}' \cdots D_1')(b)=0$.
By \eqref{CE4E4194-924C-4DBA-8061-AFEFCBDD57BC}, it follows that $(D_n D_{n-1} \cdots D_1)(b)=0$.
\end{proof}

\begin{proposition} \label{DDB3B9BE-7B31-4FF9-A45C-02A0882AB545}
Let $\Lscr \in \DJL(B)$ and $(A_1,\dots,A_m) = \kerseq(\Lscr)$.
If $\ggoth$ is a $\Integ^d$-grading of $B$ then:
$$
\text{$\Lscr$ is $\ggoth$-homogeneous if and only if $A_1$, \dots, $A_m$ are $\ggoth$-homogeneous subrings of $B$.}
$$
\end{proposition}

\begin{proof}
If $\Lscr$ is $\ggoth$-homogeneous then so are $\Lscr^{(0)}$, $\Lscr^{(1)}$, \dots, $\Lscr^{(m-1)}$.
By Lemma \ref{E881691D-8557-4672-B859-92E3341ADF4B}, $A_i = \ker( \Lscr^{(i-1)} )$ for all $i=1,\dots,m$; so $A_1$, \dots, $A_m$ are $\ggoth$-homogeneous subrings of $B$.

For the converse, choose an order $\le$ that makes $(\Integ^d,\le)$ a totally ordered abelian group
and consider the map $\lnd(B) \to \hlnd(B,\ggoth)$, $D \mapsto \hat D$, determined by that choice (see \ref{C3246E05-854F-427C-828A-669E2281D276}).
Recall that $\Lscr = \Lscr(A_1, \dots, A_m)$ by Proposition \ref{0615A9A0-CFB7-445C-B34E-44BA30B13CFA}.
Suppose that $A_1$, \dots, $A_m$ are $\ggoth$-homogeneous subrings of $B$;
to show that $\Lscr$ is $\ggoth$-homogeneous, it suffices to show that $D \in \Lscr(A_1, \dots, A_m)$ implies $\hat D \in \Lscr(A_1, \dots, A_m)$.
Consider $D \in \Lscr(A_1, \dots, A_m)$ and write $D = \sum_{\delta \in \Delta} \delta$ where
$\Delta$ is a finite set whose elements are nonzero $\ggoth$-homogeneous $\bk$-derivations (of $B$) of distinct degrees.
Let $i \in \{1,\dots,m\}$. We have $D(A_i) \subseteq A_i$ because $D \in \Lscr(A_1, \dots, A_m)$,
so if $h$ is a $\ggoth$-homogeneous element of $A_i$ then $A_i \ni D(h) = \sum_{\delta \in \Delta} \delta(h)$,
where the nonzero terms in the sum are $\ggoth$-homogeneous of distinct degrees;
since $A_i$ is a $\ggoth$-homogeneous subring of $B$, we obtain $\delta(h) \in A_i$ for all $\delta \in \Delta$.
Since the group $(A_i,+)$ is generated by its $\ggoth$-homogeneous elements, this implies that $\delta(A_i) \subseteq A_i$ for all $\delta \in \Delta$,
i.e., each $\delta \in \Delta$ maps each $A_i$ into itself.
Since $D$ is an $A_1$-derivation, a similar argument shows that each $\delta \in \Delta$ is an $A_1$-derivation
(let $h$ be a $\ggoth$-homogeneous element of $A_1$ and write $0 = D(h) = \sum_{\delta \in \Delta} \delta(h)$, etc.).
In particular, $\hat D$ is an $A_1$-derivation that maps each $A_i$ into itself. Since $\hat D$ is locally nilpotent, it belongs to $\Lscr = \Lscr(A_1, \dots, A_m)$.
This proves the claim.
\end{proof}

\begin{proposition}  \label {C70CF166-08AB-4FA3-971E-432644916C40}
Let $\ggoth$ be a $\Integ^d$-grading of $B$.
\begin{enumerate}[\rm(a)]

\item Every maximal element of $\LieLN_\bk(B,\ggoth)$ is a maximal element of $\LieLN_\bk(B)$.

\item Every maximal element of $\DJL(B,\ggoth)$ is a maximal element of $\DJL(B)$.

\end{enumerate}
\end{proposition}

\begin{proof}
(a) Choose an order $\le$ that makes $(\Integ^d,\le)$ a totally ordered abelian group
and consider the map $\lnd(B) \to \hlnd(B,\ggoth)$, $D \mapsto \hat D$, determined by that choice.
By contradiction, suppose that $\Lscr \subset \Lscr'$ is a strict inclusion where
$\Lscr$ is a maximal element of $\LieLN_\bk(B,\ggoth)$ and $\Lscr' \in \LieLN_\bk(B)$.
For each $D \in \Lscr' \setminus \Lscr$, let $\#(D) \in \Nat \setminus \{0\}$ denote the number of $\ggoth$-homogeneous components of $D$.
Let $m = \min\setspec{ \#(D) }{ D \in \Lscr' \setminus \Lscr }$ and choose $D \in \Lscr' \setminus \Lscr$ such that $\#(D) = m$.
Let $H = \setspec{ E \in \Lscr }{ \text{$E$ is $\ggoth$-homogeneous} }$ and observe that $H$ is 
a generating set for the Lie algebra $\Lscr$.
Since $H \cup \{D\} \subseteq \Lscr' \in \LieLN_\bk(B)$, $H \cup \{D\}$ is a locally nilpotent subset of $\Der_\bk(B)$.
So Lemma \ref{E7261F32-B2ED-4D50-8F3E-E694DB49CF21} implies that $H \cup \{ \hat D \}$ is a locally nilpotent subset of $\Der_\bk(B)$.
By Remark \ref{03442952-E213-4085-B07A-FABC7408E85A}, the Lie subalgebra $\bar\Lscr$ of $\Der_\bk(B)$
generated by $H \cup \{ \hat D \}$ is a locally nilpotent subset of $\Der_\bk(B)$, so $\bar\Lscr \in \LieLN_\bk(B)$.
Since $\bar \Lscr$ is generated by homogeneous derivations, we have in fact $\bar\Lscr \in \LieLN_\bk(B,\ggoth)$.
We also have $\Lscr \cup \{\hat D\} \subseteq \bar\Lscr$, so the maximality of $\Lscr$ implies that
$\Lscr = \bar\Lscr$ and $\hat D \in \Lscr$. It follows that $\hat D \in \Lscr'$ and consequently $D-\hat D \in \Lscr'$.
However,  $D-\hat D \notin \Lscr$ because $D \notin\Lscr$.
So $D-\hat D \in \Lscr' \setminus \Lscr$ and $\#(D - \hat D) < \#(D)$, which contradicts our choice of $D$. This proves (a).

(b) Let $\Lscr$ be a maximal element of $\DJL(B,\ggoth)$.
Then $\Lscr \in \LieLN_\bk(B,\ggoth)$ by Remark \ref{D8E0AD3A-6104-479D-B7E2-D38E67645BBD},
so Zorn's Lemma and Lemma \ref{DEA84E5E-79BC-4D4B-BF2D-C5735262EA6D} show that 
there exists a maximal element $\Lscr'$ of $\LieLN_\bk(B,\ggoth)$ such that $\Lscr \subseteq \Lscr'$.
By (a), $\Lscr'$ is a maximal element of $\LieLN_\bk(B)$ and hence (by Proposition \ref{6D6E22A4-236A-4B28-96D9-86157812D34A}(c)) a maximal element of $\DJL(B)$.
As $\Lscr'$ is homogeneous, we have $\Lscr' \in \DJL(B,\ggoth)$, so the fact that $\Lscr$ is maximal in $\DJL(B,\ggoth)$
implies that $\Lscr = \Lscr'$. So $\Lscr$ is a maximal element of $\DJL(B)$.
\end{proof}

\begin{notation} \label {FA176F6C-52D4-43FA-A7AF-B41C94E23E8B}  
Let $\ggoth = (B_i)_{i \in G}$ be a $G$-grading of $B$, where $G = \Integ^d$ for some $d \in \Nat$.

\smallskip

\noindent(i) We define $\rank(\ggoth)$ to be the rank of the subgroup of $G$ generated by $\setspec{ i \in G }{ B_i \neq 0 }$.
It is well known that $\rank(\ggoth) \le \rtrdeg_\bk(B)$.

\smallskip

\noindent(ii)
Let $\Hom(G,\bk)$ denote the $\bk$-vector space whose elements are the group homomorphisms from $G$ to the additive group of $\bk$.
For each $\phi \in \Hom(G,\bk)$, define the $\bk$-derivation $D_\ggoth^\phi : B \to B$ by
declaring that if $i \in G$ and $x \in B_i$ then $D_\ggoth^\phi(x) = \phi(i) x$.
Then $D_\ggoth^\phi$ is semisimple and hence locally finite.
Note that $\phi \mapsto D_\ggoth^\phi$ is a $\bk$-linear map $\Hom(G,\bk) \to \Der_\bk(B)$;
we denote its image by
$$
L(\ggoth) = \setspec{D_\ggoth^\phi}{ \phi \in \Hom(G,\bk) } .
$$
As $[D_\ggoth^\phi , D_\ggoth^\psi ] = 0$ for all $\phi,\psi \in \Hom(G,\bk)$, $L(\ggoth)$ is a commutative Lie subalgebra of $\Der_\bk(B)$.
Moreover, the dimension of $L(\ggoth)$ as a $\bk$-vector space is equal to $\rank(\ggoth)$,
and it is clear that
\begin{equation}  \label {DA19998A-8871-4661-A7A2-9B7A36F645D9}
L(\ggoth) \cap \lnd(B) = \{0\} .
\end{equation}
Also observe that if $\phi \in \Hom(G,\bk)$ then
$D_\ggoth^\phi$ is $\ggoth$-homogeneous of degree $0$,
and if $E \in \Der_\bk(B)$ is $\ggoth$-homogeneous of degree $d \in G$ then $[D_\ggoth^\phi , E] = \phi(d)E$.
\end{notation}

\begin{remark} \label {A76E8B96-799E-45EC-9363-51DAC324DFFD}
Assume that $\bk$ is algebraically closed and let $X = \Spec B$.
We will see in paragraph \ref{92E834E4-3895-42FF-A577-61C5AF52410C}(a) that $L(\ggoth)$ is the Lie algebra of the algebraic torus $T_\ggoth \subseteq \Aut(X)$ determined by $\ggoth$.
\end{remark}

We leave it to the reader to verify the following assertion (see \ref{EDCCFF3C-F3DF-4BCD-82AC-363D73AFAA4B} for the notation  $\Norm(\Lscr)$):

\begin{lemma}  \label {98BA9308-0D1C-49E9-8D0D-7F648E5E2B67}
Let $\ggoth$ be a $\Integ^d$-grading of $B$.
\begin{enumerate}[\rm(a)]

\item A derivation $E \in \Der_\bk(B)$ is $\ggoth$-homogeneous of degree $0$ if and only if $[ E, L(\ggoth) ] = 0$.

\item Let $\Lscr$ be a Lie subalgebra of $\Der_\bk(B)$.
\begin{itemize}
\item If $\Lscr$ is $\ggoth$-homogeneous then so is $\Norm(\Lscr)$.
\item $\Lscr$ is $\ggoth$-homogeneous if and only if $L(\ggoth) \subseteq \Norm(\Lscr)$.
\end{itemize}
\end{enumerate}
\end{lemma}

\begin{definition} \label {B77C75D9-E381-4069-A82A-870BECEA6849}
Let $\Lscr$ be a Lie subalgebra of $\Der_\bk(B)$.
\begin{enumerate}[\rm(i)]

\item Let $\Grad(\Lscr)$ denote the set of all $\Integ^n$-gradings $\ggoth$ of $B$ (with $n$ arbitrary) such that $\Lscr$ is $\ggoth$-homogeneous.
Note that $\Grad(\Lscr)$ is not empty, as it contains the trivial $\Integ$-grading of $B$.

\item We define the {\it grading rank\/} of $\Lscr$ by $\grank(\Lscr) = \max\setspec{ \rank\ggoth }{ \ggoth \in \Grad(\Lscr) }$.

\end{enumerate}
We have $\grank(\Lscr) \le \rtrdeg_\bk(B)$ by \ref{FA176F6C-52D4-43FA-A7AF-B41C94E23E8B}(i).
Also observe that
\begin{equation}  \label {391D2623-B472-4F59-9487-EE08025E5C46}
\text{for each $\ggoth \in \Grad(\Lscr)$,} \qquad
[L(\ggoth),L(\ggoth)] = \{0\} , \quad L(\ggoth) \subseteq \lfd_\bk(B)  \quad  \text{and} \quad L(\ggoth) \subseteq \Norm(\Lscr)
\end{equation}
by paragraph \ref{FA176F6C-52D4-43FA-A7AF-B41C94E23E8B} and Lemma \ref{98BA9308-0D1C-49E9-8D0D-7F648E5E2B67}(b).
\end{definition}

\begin{remark}  \label {3E8A3225-A052-47E0-99B6-C8935E9A664C}
Let $\bbL$ be either $\LieWLF_\bk(B)$ or $\LieWLF_\bk(B)_{\text{\rm sol}}$ and let $\Lscr \in \bbL$.
\begin{enumerate}[\rm(a)]

\item If $L$ is a Lie subalgebra of $\Norm(\Lscr)$ such that $L \subseteq \lfd_\bk(B)$ and $[L,L] \subseteq \Lscr$, then $L + \Lscr \in \bbL$.

\item If $\ggoth \in \Grad(\Lscr)$ then $L(\ggoth) + \Lscr  \in \bbL$.

\item If $\Lscr$ is a maximal element of $\bbL$ then $L(\ggoth) \subseteq \Lscr$ for all $\ggoth \in \Grad(\Lscr)$.

\end{enumerate}
Indeed, part (a) follows by applying Lemma \ref{CA556DEE-FE7D-4364-A217-A983EBE59700} to the set $T = L$
(note that, in Lemma \ref{CA556DEE-FE7D-4364-A217-A983EBE59700}, if $\Lscr$ is solvable then so is $\Mscr$).
Statement \eqref{391D2623-B472-4F59-9487-EE08025E5C46} shows that (b) is a special case of (a), and it is clear that (c) is a special case of (b).
\end{remark}

To prove the next Proposition we require the following well-known fact, whose proof we omit:

\begin{lemma} \label {C6AEC471-BCDE-4A94-B59D-72C4AAF45868}
Let $\ggoth$ be a grading of $B$ by a totally ordered abelian group $(G,\le)$ and let $D \in \lfd_\bk(B)$.
Then the highest homogeneous component $\hat D$ of $D$ belongs to $\lfd_\bk(B)$, and if $\deg_\ggoth( \hat D ) \neq 0$ then $\hat D$ is locally nilpotent.
\end{lemma}

\begin{proposition} \label {8E965CF0-961C-4AC9-859F-398B4CEA6E5E}
Let $\Lscr$ be a maximal element of $\LieLN_\bk(B)$.
\begin{enumerate}[\rm(a)]

\item $[ L(\ggoth), L(\hgoth) ] \subseteq \Lscr$ for all $\ggoth, \hgoth \in \Grad(\Lscr)$.

\item If we define $\Mscr = \Lscr + \sum_{\ggoth \in \Grad(\Lscr)} L(\ggoth)$ then
$$
\Mscr \in \LieWLF_\bk(B)_{\text{\rm sol}} , \quad [\Mscr,\Mscr] \subseteq \Lscr \quad \text{and} \quad \Mscr \cap \lnd(B) = \Lscr .
$$

\end{enumerate}
\end{proposition}

\begin{proof}
(a) We use the abbreviation $\bbL = \LieWLF_\bk(B)_{\text{\rm sol}}$.
We claim that
\begin{equation}  \label {20E16C55-F188-4712-89CB-293495FEDE5E}
\text{for each $\delta \in \lfd_\bk(B) \cap \Norm(\Lscr)$ and each $\ggoth \in \Grad(\Lscr)$,\ \ $[\delta, L(\ggoth)] \subseteq \Lscr$.}
\end{equation}
Indeed, consider a $\Integ^d$-grading $\ggoth \in \Grad(\Lscr)$ and let $\delta \in \lfd_\bk(B) \cap \Norm(\Lscr)$.
Observe that $\Lscr \in \bbL$ by \eqref{1354F7B8-F676-4D8A-9F67-8316E680B009} (in Definition \ref{67BC6576-67FE-41A8-975F-597EE29DC120});
so  Remark \ref{3E8A3225-A052-47E0-99B6-C8935E9A664C}(a) implies that $\bk \delta + \Lscr \in \bbL$.
Let $S = \Supp_\ggoth(\delta)$ and let $\delta = \sum_{i \in S} \delta_i$ be the $\ggoth$-homogeneous decomposition of $\delta$;
this notation assumes that, for each $i \in S$, the derivation $\delta_i \in \Der_\bk(B)$ is nonzero and $\ggoth$-homogeneous of degree $i \in \Integ^d$.
Since $\Lscr$ is $\ggoth$-homogeneous, so is $\Norm(\Lscr)$ by Lemma \ref{98BA9308-0D1C-49E9-8D0D-7F648E5E2B67}(b);
so the fact that $\delta \in \Norm(\Lscr)$ implies that $\delta_i \in \Norm(\Lscr)$ for all $i \in S$.
Let $S' = \setspec{ i \in S }{ \delta_i \in \Lscr }$ and write $\delta = D + E$ where
$E = \sum_{i \in S'} \delta_i \in \Lscr$ and $D = \sum_{i \in S \setminus S'} \delta_i$.
Obviously, $\Supp_\ggoth(D) = S \setminus S'$ and  $D = \sum_{i \in S \setminus S'} \delta_i$ is the $\ggoth$-homogeneous decomposition of $D$.
Since $D \in \bk \delta + \Lscr \in \bbL$, we have $D \in \lfd_\bk(B)$.

We claim that $D$ is $\ggoth$-homogeneous of degree $0$.
Indeed, assume the contrary. Then $\Supp_\ggoth(D) \nsubseteq \{0\}$, so we can choose an order $\le$
that makes $(\Integ^d,\le)$ a totally ordered abelian group and with respect to which $j := \max(\Supp_\ggoth D)$ is not $0$;
since $D$ is locally finite, it then follows that $\delta_j$ is locally nilpotent by Lemma \ref{C6AEC471-BCDE-4A94-B59D-72C4AAF45868},
so $\delta_j \in \Norm(\Lscr) \cap \lnd(B) = \Lscr$ where the equality is Lemma \ref{CAA13785-2B53-4A53-80E3-3857241586CB}(b).
Since $\delta_j \in \Lscr$, we have $j \in S'$, which contradicts $j \in \Supp_\ggoth(D) = S \setminus S'$.
This contradiction proves that $D$ is $\ggoth$-homogeneous of degree $0$.
Then $[D, L(\ggoth)] = 0$ by Lemma \ref{98BA9308-0D1C-49E9-8D0D-7F648E5E2B67}(a).
Since $\Lscr$ is $\ggoth$-homogeneous, we have $L(\ggoth) \subseteq \Norm(\Lscr)$  by Lemma \ref{98BA9308-0D1C-49E9-8D0D-7F648E5E2B67}(b),
so $[E, L(\ggoth)] \subseteq [\Lscr,L(\ggoth)] \subseteq \Lscr$.
Thus, $[\delta, L(\ggoth)] = [D+E, L(\ggoth)] \subseteq \Lscr$,
which proves \eqref{20E16C55-F188-4712-89CB-293495FEDE5E}.
It follows that (a) is true, because \eqref{391D2623-B472-4F59-9487-EE08025E5C46} implies that
$L(\hgoth) \subseteq \lfd_\bk(B) \cap \Norm(\Lscr)$ for each $\hgoth \in \Grad(\Lscr)$.

(b) Let $T = \bigcup_{\ggoth \in \Grad(\Lscr)} L(\ggoth)$; then  $[T,T] \subseteq \Lscr$ by part (a).
Since $T \subseteq \lfd_\bk(B)$ and $[T,\Lscr] \subseteq \Lscr$ by \eqref{391D2623-B472-4F59-9487-EE08025E5C46},
Lemma \ref{CA556DEE-FE7D-4364-A217-A983EBE59700} implies that  $\Nscr := \LieAlg{T \cup \Lscr}$ belongs to $\bbL$ and $[\Nscr,\Nscr] \subseteq \Lscr$. 
Using $[T,T] \subseteq \Lscr$ and $[T,\Lscr] \subseteq \Lscr$, we easily obtain $\Mscr = \Nscr$, so  $\Mscr \in \bbL$ and $[\Mscr,\Mscr] \subseteq \Lscr$. 
The fact that $[\Mscr,\Mscr] \subseteq \Lscr$ implies in particular that $\Mscr \subseteq \Norm(\Lscr)$.
Since $\Norm(\Lscr) \cap \lnd(B) = \Lscr$ by Lemma \ref{CAA13785-2B53-4A53-80E3-3857241586CB}(b), it follows that $\Mscr \cap \lnd(B) = \Lscr$.
\end{proof}

\begin{remark}
If $\Lscr$ is a non-maximal element of $\LieLN_\bk(B)$
then $\Lscr + \sum_{\ggoth \in \Grad(\Lscr)} L(\ggoth)$ is not necessarily weakly locally finite,
contrary to what happens in Proposition \ref{8E965CF0-961C-4AC9-859F-398B4CEA6E5E}. 
For instance one can see that if $B = \kk2$ and $\Lscr = \{0\}$ then
$\Lscr + \sum_{\ggoth \in \Grad(\Lscr)} L(\ggoth)$ is not even included in $\lfd_\bk(B)$.
%
% Consider the following example.
% Let $B = \kk2$ and $\Lscr = \{0\} \in \LieLN_\bk(B)$.
% If $\gamma = (u,v)$ is a coordinate system of $B$
% then let $\ggoth_{\gamma} = (B_i)_{ i \in \Nat }$ be the $\Integ$-grading of $B$ such that $u \in B_1$ and $v \in B_0$,
% and let $D_\gamma$ be the $\bk$-derivation of $B$ such that $D_\gamma(u) = u$ and $D_\gamma(v) = 0$.
% Then for each coordinate system $\gamma$ of $B$ we have $\ggoth_\gamma \in \Grad(\Lscr)$ and $D_\gamma \in L( \ggoth_\gamma )$.
% Let $x,y$ be such that $B = \bk[x,y]$ and let $\partial_x, \partial_y$ be the partial derivatives.
% If $\gamma_1 = (x+y^2,y)$ then $D_{\gamma_1} = (x+y^2) \partial_x$; 
% if $\gamma_2 = (y+x^2,x)$ then $D_{\gamma_2} = (y+x^2) \partial_y$.
% It is easy to check that the derivation $D_{\gamma_1} + D_{\gamma_2} = (x+y^2) \partial_x + (y+x^2) \partial_y$ is not locally finite.
% So $\bigcup_{\ggoth \in \Grad(\Lscr)} L(\ggoth)$ is not a weakly locally finite set.
\end{remark}

\begin{notation} \label {B6DFEABF-3379-4ECF-9C30-26F2BCDB4FD6}
$\Lgoth(B) = \setspec{ \Lscr + L(\ggoth) }{ \Lscr \in \LieLN_\bk(B) \text{ and } \ggoth \in \Grad(\Lscr) }$
\end{notation}

We will need the following fact about inclusions in $\Lgoth(B)$.

\begin{lemma}  \label {F2922F53-50B0-4424-A59A-C3A8D4BE38A3}
Let $\Lscr \in \LieLN_\bk(B)$ and let $\ggoth_1, \ggoth_2 \in \Grad(\Lscr)$ be such that $L(\ggoth_1) + \Lscr \subseteq L(\ggoth_2) + \Lscr$.
Then $\rank(\ggoth_1) \le \rank(\ggoth_2)$, and $L(\ggoth_1) + \Lscr = L(\ggoth_2) + \Lscr$ $\Leftrightarrow$ $\rank(\ggoth_1) = \rank(\ggoth_2)$.
\end{lemma}

\begin{proof} 
For each $i \in \{1,2\}$,
we have $\dim_\bk L(\ggoth_i) = \rank(\ggoth_i)$ by paragraph \ref{FA176F6C-52D4-43FA-A7AF-B41C94E23E8B}
and $L(\ggoth_i) \cap \Lscr = \{0\}$ by \eqref{DA19998A-8871-4661-A7A2-9B7A36F645D9}. 
So this Lemma is a consequence of the following linear algebra fact, whose proof we leave to the reader:
{\it Let $L_1$, $L_2$ and $\Lscr$ be subspaces of some vector space and assume that
$L_1 \cap \Lscr = \{0\} = L_2 \cap \Lscr$, $L_1$ and $L_2$ are finite dimensional, and $L_1 + \Lscr \subseteq L_2 + \Lscr$.
Then $\dim L_1 \le \dim L_2$, and the condition  $L_1 + \Lscr = L_2 + \Lscr$ is equivalent to $\dim L_1 = \dim L_2$.}
\end{proof}

\begin{lemma} \label {36F36D71-D42B-4E92-8612-63FD8842DBA9}
Let $\Lscr \in \LieLN_\bk(B)$ and $\ggoth \in \Grad(\Lscr)$, and consider the element $\Mscr = L(\ggoth) + \Lscr$ of $\Lgoth(B)$.
\begin{enumerate}[\rm(a)]

\item $\Mscr \in \LieWLF_\bk(B)_{\text{\rm sol}}$

\item $\Mscr \cap \lnd(B) = \Lscr$

\item If $\Mscr$ is a maximal element of $\Lgoth(B)$ then $\Lscr$ is a maximal element of $\LieLN_\bk(B)$.

\item If $\Lscr$ is a maximal element of $\LieLN_\bk(B)$ and $\rank(\ggoth) = \grank(\Lscr)$ then $\Mscr$ is a maximal element of $\Lgoth(B)$.

\end{enumerate}
\end{lemma}

\begin{proof}
We have $\Lscr \in \LieWLF_\bk(B)_{\text{\rm sol}}$ by \eqref{1354F7B8-F676-4D8A-9F67-8316E680B009},
so Remark \ref{3E8A3225-A052-47E0-99B6-C8935E9A664C}(b) gives $\Mscr \in \LieWLF_\bk(B)_{\text{\rm sol}}$, proving (a).
To prove (b), consider $\delta \in \Mscr \cap \lnd(B)$.
We have $L(\ggoth) \subseteq \Norm(\Lscr)$ by \eqref{391D2623-B472-4F59-9487-EE08025E5C46},
so $\Mscr \subseteq \Norm(\Lscr)$ and hence $\delta \in \Norm(\Lscr) \cap \lnd(B)$;
this together with $\Lscr \in \LieLN_\bk(B)$ implies that $\bk \delta + \Lscr \in \LieLN_\bk(B)$ by Lemma \ref{CAA13785-2B53-4A53-80E3-3857241586CB}(a);
consequently, $\bk \delta + \Lscr \subseteq \lnd(B)$.
Note that $\delta = E + D$ for some $E \in L(\ggoth)$ and $D \in \Lscr$. Then $E = \delta - D \in \bk \delta + \Lscr \subseteq \lnd(B)$,
i.e., $E \in L(\ggoth) \cap \lnd(B) = \{0\}$ by \eqref{DA19998A-8871-4661-A7A2-9B7A36F645D9}.
So $\delta = D \in \Lscr$, proving that $\Mscr \cap \lnd(B) \subseteq \Lscr$.
Since $\Lscr \subseteq \Mscr \cap \lnd(B)$ is clear, (b) is proved.

(c) Suppose that $\Mscr$ is maximal in $\Lgoth(B)$.
Note that $\Lscr \in \LieLN_\bk(B,\ggoth)$ and consider any $\Lscr' \in \LieLN_\bk(B,\ggoth)$ such that $\Lscr \subseteq \Lscr'$.
Then $\Mscr' := L(\ggoth) + \Lscr'$ belongs to  $\Lgoth(B)$ and $\Mscr \subseteq \Mscr'$,
so maximality of $\Mscr$ implies that $\Mscr = \Mscr'$.
By part (b), $\Lscr = \Mscr \cap \lnd(B) = \Mscr' \cap \lnd(B) = \Lscr'$.
This shows that $\Lscr$ is maximal in $\LieLN_\bk(B,\ggoth)$;
by Proposition \ref{C70CF166-08AB-4FA3-971E-432644916C40}(a), $\Lscr$ is maximal in $\LieLN_\bk(B)$.

(d) Suppose that $\Lscr$ is maximal in $\LieLN_\bk(B)$ and $\rank(\ggoth) = \grank(\Lscr)$.
Consider $\Mscr' \in \Lgoth(B)$ such that $\Mscr \subseteq \Mscr'$.
We have $\Mscr' = L(\ggoth') + \Lscr'$ for some $\Lscr' \in \LieLN_\bk(B)$ and $\ggoth' \in \Grad(\Lscr')$.
Part (b) shows that $\Lscr = \Mscr \cap \lnd(B) \subseteq \Mscr' \cap \lnd(B) = \Lscr'$,
so maximality of $\Lscr$ implies that $\Lscr = \Lscr'$ and hence $\Mscr' = L(\ggoth') + \Lscr$ where  $\ggoth' \in \Grad(\Lscr)$.
Since $\rank(\ggoth') \le \grank(\Lscr) = \rank(\ggoth)$, Lemma \ref{F2922F53-50B0-4424-A59A-C3A8D4BE38A3} gives $\Mscr = \Mscr'$,
showing that $\Mscr$ is maximal in $\Lgoth(B)$.
\end{proof}

If $\bbL$ is any collection of Lie subalgebras of $\Der_\bk(B)$,
or more generally any collection of sets, we write $\max\bbL$ for the set of maximal elements of the poset $(\bbL,\subseteq)$.

\begin{proposition} \label {CD178309-4EC6-4386-9C1E-34F625BEBC29}
If the condition\footnote{We will see in Corollary \ref{CC2270E1-029B-4F26-89F8-81B03963C6CD} that if $\bk$ is algebraically closed then
the assumption $(*)$ of the Proposition is satisfied.}
\begin{equation} \tag{$*$}
\text{each maximal element of $\LieWLF_\bk(B)_{\text{\rm sol}}$ belongs to $\Lgoth(B)$}
\end{equation}
is satisfied, then the following hold.
\begin{enumerate}[\rm(a)]

\item The map
$$
\rho_B : \max \LieWLF_\bk(B)_{\text{\rm sol}} \longrightarrow \max \LieLN_\bk(B) , \quad   \Mscr \mapsto  \Mscr \cap \lnd(B) ,
$$
is a well-defined bijection. The inverse of $\rho_B$ is described in part {\rm (b)}.

\item Let $\Lscr \in \max \LieLN_\bk(B)$. There exists $\ggoth \in \Grad(\Lscr)$ such that $\rank(\ggoth) = \grank(\Lscr)$,
and for any such $\ggoth$ we have $L(\ggoth) + \Lscr \in \max \LieWLF_\bk(B)_{\text{\rm sol}}$ and $\rho_B(L(\ggoth) + \Lscr) = \Lscr$.

\item For every $\Mscr \in \max \LieWLF_\bk(B)_{\text{\rm sol}}$ and $\Lscr \in \max \LieLN_\bk(B)$,
the condition $\Mscr \supseteq \Lscr$ is equivalent to $\rho_B(\Mscr) = \Lscr$.

\item Two elements of $\max \LieWLF_\bk(B)_{\text{\rm sol}}$ are conjugate\footnote{We say that
Lie subalgebras $L_1$, $L_2$ of $\Der_\bk(B)$ are \textit{conjugate} if there exists an automorphism $\theta \in \Aut_\bk(B)$
such that $\theta L_1 \theta^{-1} = L_2$, where we define $\theta L_1 \theta^{-1} = \setspec{\theta D \theta^{-1}}{ D \in L_1}$.}
if and only if their images by $\rho_B$ are conjugate.

\end{enumerate}
\end{proposition}

\begin{proof}
By Zorn's Lemma and parts (b) and (c) of Lemma \ref{DEA84E5E-79BC-4D4B-BF2D-C5735262EA6D},
each element of $\LieWLF_\bk(B)_{\text{\rm sol}}$ is included in a maximal element.
Also, Lemma \ref{36F36D71-D42B-4E92-8612-63FD8842DBA9}(a) gives $\Lgoth(B) \subseteq \LieWLF_\bk(B)_{\text{\rm sol}}$.
By these observations and $(*)$, it follows that 
\begin{equation} \label {9F245DFE-A83C-4A78-84B8-6E61BECB4C18}
\text{the posets $(\Lgoth(B),\subseteq)$ and $(\LieWLF_\bk(B)_{\text{\rm sol}},\subseteq)$ have the same maximal elements.}
\end{equation}

(a) If $\Mscr \in \max \LieWLF_\bk(B)_{\text{\rm sol}}$ ($ = \max\Lgoth(B)$ by \eqref{9F245DFE-A83C-4A78-84B8-6E61BECB4C18})
then parts (b) and (c) of Lemma \ref{36F36D71-D42B-4E92-8612-63FD8842DBA9} imply that $\Mscr \cap \lnd(B) \in \max \LieLN_\bk(B)$; so $\rho_B$ is a well-defined map.
To prove injectivity, let $\Lscr \in \max \LieLN_\bk(B)$ and suppose that 
$\Mscr_1, \Mscr_2 \in \max \LieWLF_\bk(B)_{\text{\rm sol}}$ are such that $\rho_B(\Mscr_1) = \Lscr = \rho_B(\Mscr_2)$.
For each $i=1,2$, we have $\Mscr_i \in \Lgoth(B)$ (by $(*)$) and $\Mscr_i \cap \lnd(B) = \Lscr$,
so Lemma \ref{36F36D71-D42B-4E92-8612-63FD8842DBA9}(b) implies that $\Mscr_i = L(\ggoth_i) + \Lscr$ for some $\ggoth_i \in \Grad(\Lscr)$.
Proposition \ref{8E965CF0-961C-4AC9-859F-398B4CEA6E5E} implies that $\Mscr := \Lscr + \sum_{\ggoth \in \Grad(\Lscr)} L(\ggoth)$ belongs to $\LieWLF_\bk(B)_{\text{\rm sol}}$;
since $\Mscr_i \subseteq \Mscr$ and $\Mscr_i$ is maximal, we obtain $\Mscr_i = \Mscr$ for each $i=1,2$, so $\rho_B$ is injective.
Surjectivity will follow from (b), see below.

(b) By definition of $\grank(\Lscr)$, there exists $\ggoth \in \Grad(\Lscr)$ such that $\rank(\ggoth) = \grank(\Lscr)$.
For any such $\ggoth$, parts (b) and (d) of Lemma \ref{36F36D71-D42B-4E92-8612-63FD8842DBA9}
imply that $\Mscr := L(\ggoth) + \Lscr$ is a maximal element of $\Lgoth(B)$ satisfying $\Mscr \cap \lnd(B) = \Lscr$.
We have $\Mscr \in \max \LieWLF_\bk(B)_{\text{\rm sol}}$ by \eqref{9F245DFE-A83C-4A78-84B8-6E61BECB4C18}, so this proves (b).
It follows that $\rho_B$ is surjective, which completes the proof of (a).

(c) Note that $\Mscr \in \Lgoth(B)$ by $(*)$, so  $\Mscr \cap \lnd(B)$ belongs to $\LieLN_\bk(B)$ by Lemma \ref{36F36D71-D42B-4E92-8612-63FD8842DBA9}(b).
Since $\Lscr$ is maximal in $\LieLN_\bk(B)$, we have
$\Lscr \subseteq \Mscr \cap \lnd(B)$ $\Leftrightarrow$ $\Lscr = \Mscr \cap \lnd(B)$.
Since $\Lscr \subseteq \lnd(B)$, we have 
$\Lscr \subseteq \Mscr$ $\Leftrightarrow$ $\Lscr \subseteq \Mscr \cap \lnd(B)$. This proves (c).

(d) Let $\Mscr \in \max \LieWLF_\bk(B)_{\text{\rm sol}}$, $\Lscr = \rho_B(\Mscr) \in \max\LieLN_\bk(B)$ and $\theta \in \Aut_\bk(B)$.
Note that $\theta \Mscr \theta^{-1}  \in \max \LieWLF_\bk(B)_{\text{\rm sol}}$ and $\theta \Lscr \theta^{-1}  \in \max\LieLN_\bk(B)$.
Since $\Lscr \subseteq \Mscr$, we have $\theta \Lscr \theta^{-1} \subseteq \theta \Mscr \theta^{-1}$,
so part (c) implies that $\rho_B( \theta \Mscr \theta^{-1} ) = \theta \Lscr \theta^{-1}$.  Thus,
\begin{equation}  \label {4F139596-F93E-4428-96AB-03A48C79EDAD}
\rho_B( \theta \Mscr \theta^{-1} ) = \theta \, \rho_B(\Mscr) \, \theta^{-1}
\quad \text{for all $\Mscr \in \max \LieWLF_\bk(B)_{\text{\rm sol}}$ and $\theta \in \Aut_\bk(B)$.}
\end{equation}
It follows that if two elements of $\max \LieWLF_\bk(B)_{\text{\rm sol}}$ are conjugate then so are their images by $\rho_B$.
Since $\rho_B$ is injective, the converse is also a consequence of \eqref{4F139596-F93E-4428-96AB-03A48C79EDAD}.
\end{proof}

%%%%%%%%%%%%%%%%%%%%%%%%%%%%%%%%%%%%%%%%%%%%%%%%%%%%%%%%%%%%%%%%%%%%%%%%%%%%%%%%%%%%%%%%%%%%%%%%%%%%%%%%%%%%%%%%
%%%%%%%%%%%%%%%%%%%%%%%%%%%%%%%%%%%%%%%%%%%%%%%%%%%%%%%%%%%%%%%%%%%%%%%%%%%%%%%%%%%%%%%%%%%%%%%%%%%%%%%%%%%%%%%%
%%%%%%%%%%%%%%%%%%%%%%%%%%%%%%%%%%%%%%%%%%%%%%%%%%%%%%%%%%%%%%%%%%%%%%%%%%%%%%%%%%%%%%%%%%%%%%%%%%%%%%%%%%%%%%%%
%%%%%%%%%%%%%%%%%%%%%%%%%%%%%%%%%%%%%%%%%%%%%%%%%%%%%%%%%%%%%%%%%%%%%%%%%%%%%%%%%%%%%%%%%%%%%%%%%%%%%%%%%%%%%%%%
%%%%%%%%%%%%%%%%%%%%%%%%%%%%%%%%%%%%%%%%%%%%%%%%%%%%%%%%%%%%%%%%%%%%%%%%%%%%%%%%%%%%%%%%%%%%%%%%%%%%%%%%%%%%%%%%

\section{Borel subgroups of $\Aut(X)$}
\label{sectionUniqueBorel}

Throughout this section, $\bk$ is an algebraically closed field of characteristic zero,
$B$ is an affine $\bk$-domain and $X = \Spec(B)$ is the corresponding affine $\bk$-variety.
We view $\Aut(X)$ as an ind-group.
We refer to  \cite{FurterKraft} and \cite[Chapter IV]{kumar2012kac} for background on ind-varieties and ind-groups,
and we assume that the reader is familiar with the basic definitions in this area.

\begin{noname}  \label {E6194646-DD25-4CCB-B616-129EAB72740F}
By \cite[Theorem 5.1.1]{FurterKraft}, $\Aut(X)$ is an affine ind-group.
As such, it has a topology (the ind-topology),
and the notion of an ``algebraic subset of $\Aut(X)$'' is defined.\footnote{Given an
ind-variety $V$ defined by a filtration $V_0 \subseteq V_1 \subseteq V_2 \subseteq \cdots$,
a subset of $V$ is {\it algebraic\/} if it is locally closed and contained in some $V_i$.}
Every algebraic subset of $\Aut(X)$ has a natural structure of algebraic variety.
By an {\it algebraic subgroup\/} of $\Aut(X)$, one means a subgroup $G$ of $\Aut(X)$ that is also a closed algebraic subset of $\Aut(X)$.
Such a $G$ is in particular an affine algebraic variety,
and it is easy to see that the group operations of $G$ are compatible with this variety structure.
In other words, each algebraic subgroup $G$ of $\Aut(X)$ is naturally endowed with the structure of an algebraic group,
and (by \cite[Theorem 5.1.1(1)]{FurterKraft})
the map $G \times X \to X$, $(g,x) \mapsto g(x)$, is an algebraic action of this algebraic group on the algebraic variety $X$.
\end{noname}

\begin{noname}
Let $G$ be a closed subgroup of $\Aut(X)$, i.e., a subgroup that is closed in the ind-topology of $\Aut(X)$.
If $V_0 \subseteq V_1 \subseteq V_2 \subseteq \cdots$ is an admissible filtration of $\Aut(X)$,
then $G \cap V_0 \subseteq G \cap V_1 \subseteq G \cap V_2 \subseteq \cdots$ defines a structure of an affine ind-group on $G$,
and this structure is independent of the choice of the admissible filtration $(V_n)_{n\ge0}$ of $\Aut(X)$.
We write $\clo{G}$ to denote $G$ equipped with this structure of affine ind-group.
\end{noname}

\begin{noname}
A subgroup $G$ of $\Aut(X)$ is said to be {\it nested\/} if it is the union of an increasing sequence
$G_0 \subseteq G_1 \subseteq G_2 \subseteq \cdots$ of algebraic subgroups of $\Aut(X)$.
If this is the case then the filtration $(G_n)_{n\ge0}$ defines a structure of affine ind-group on $G$,
and this structure is independent of the choice of the sequence $G_0 \subseteq G_1 \subseteq G_2 \subseteq \cdots$ of algebraic subgroups of $\Aut(X)$ whose union is $G$.\footnote{The independence claim follows from the fact that, over an algebraically closed field of characteristic zero,
every algebraic group contains a dense finitely generated subgroup.}
We write $\nes{G}$ to denote $G$ equipped with this structure of affine ind-group.
\end{noname}

\begin{noname}  \label {7BACADD4-796D-4491-8D0A-93D9F0132FFC}
If $G$ is a nested subgroup of $\Aut(X)$ then the following hold.
\begin{enumerate}[\rm(a)]

\item $\nes G$ is connected if and only if $G$ is the union of a sequence 
$G_0 \subseteq G_1 \subseteq G_2 \subseteq \cdots$ of connected algebraic subgroups of $\Aut(X)$.

\item If $\nes G$ is connected then $G$ is closed in $\Aut(X)$ and $\nes G = \clo G$.

\item If $\nes G$ is connected then $G$ is connected in $\Aut(X)$, but the converse is in general not true.

\end{enumerate}
In these statements, ``closed in $\Aut(X)$'' and ``connected in $\Aut(X)$'' refer to the ind-topology of $\Aut(X)$.
Part (a) follows from \cite[Proposition 2.2.1]{FurterKraft}.
Part (b) is a consequence of \cite[Proposition 3.13]{perepechko2025} (the fact that $\nes G = \clo G$ is established in the proof but not asserted in the statement);
alternatively, \cite[Proposition 3.13]{perepechko2025} states that $G$ is closed and \cite[Corollary 6.11]{solvableAutGroups} gives $\nes G = \clo G$.
In (c), the implication follows from (b) and the following example shows that the converse is false.
\end{noname}

\begin{example}
Choose an affine variety $X$ such that the ind-group $\Aut(X)$ is the algebraic group $\bbG_{\text{\rm m}}$.  Identify $\Aut(X)$ with $\bk^*$, let $G_n = \setspec{ t \in \bk^* }{ t^{ n! } = 1 }$ ($n \in \Nat$), and let $G$ be the union of $G_0 \subseteq G_1 \subseteq G_2 \subseteq \cdots$.
Then $G$ is a nested subgroup of $\Aut(X)$ and is connected in the ind-topology of $\Aut(X)$, but $\nes G$ has the discrete topology and hence is not connected.
Note that $G$ is not an element of the set $\NC(X)$ defined in \ref{44526848-3818-4BFF-94AD-FBDDA66B18EA}, even though it is nested and connected in $\Aut(X)$.
\end{example}

\begin{noname} \label {44526848-3818-4BFF-94AD-FBDDA66B18EA}
Let $\NC(X)$ denote the set of all subgroups $G$ of $\Aut(X)$ such that $G$ is nested and $\nes G$ is connected.
Each $G \in \NC(X)$ has a natural structure of affine ind-group, namely,
the structure $\nes G = \clo G$ (see \ref{7BACADD4-796D-4491-8D0A-93D9F0132FFC}(b));
we always view $G$ as being equipped with that structure and we will rarely use the notations $\nes G$ or $\clo G$.
If $G \in \NC(X)$ then the inclusion map $G \hookrightarrow \Aut(X)$ is a morphism of affine ind-groups;
if $G,H \in \NC(X)$ and $G\subseteq H$ then the inclusion map $G \hookrightarrow H$ is a morphism of affine ind-groups.
\end{noname}

\begin{noname} \label {2CABEFE5-DBF7-45CE-B5BB-A5569DBFACAF}
Every affine ind-group $G$ has a Lie algebra $\Lie(G)$, and every morphism $\phi : G \to H$ of affine ind-groups
induces a homomorphism of Lie algebras $\Lie(\phi) : \Lie(G) \to \Lie(H)$;
definitions can be found in \cite{FurterKraft} and \cite{kumar2012kac}.
In view of \ref{44526848-3818-4BFF-94AD-FBDDA66B18EA},
it follows that $\Lie(G)$ is defined for each $G \in \NC(X) \cup \big\{ \Aut(X) \big\}$.
By \cite[Theorem 0.3.2]{FurterKraft}, there is a canonical embedding of Lie algebras $\Lie(\Aut X) \hookrightarrow \Der_\bk(B)$. 
It follows that there is a natural embedding $\Lie(G) \hookrightarrow \Der_\bk(B)$ for each $G \in \NC(X) \cup \big\{ \Aut(X) \big\}$. 
We always identify $\Lie(G)$ with its image in $\Der_\bk(B)$.
Some properties of $\Lie(G)$ are summarized in the following Lemma. 
\end{noname}

\begin{lemma} \label {29B1EE01-E26F-43CA-ADB2-5A0E15668FCE}
For each $G \in \NC(X) \cup \big\{ \Aut(X) \big\}$, a Lie subalgebra $\Lie(G)$ of $\Der_\bk(B)$ is defined.
The assignment $G \mapsto \Lie(G)$ has the following properties.
\begin{enumerate}[\rm(a)]
\setlength{\itemsep}{1mm}

\item If $G,H \in \NC(X) \cup \big\{ \Aut(X) \big\}$ and $G \subseteq H$ then $\Lie(G) \subseteq \Lie(H)$.

\item \label {1E3B0B16-D6E9-4E74-96A9-A42E3459F98F}
If $G \in \NC(X)$ and $G_0 \subseteq G_1 \subseteq G_2 \subseteq \cdots$ are connected algebraic subgroups
of $\Aut(X)$ such that $G = \bigcup_{i \in \Nat} G_i$, then $\Lie(G) = \bigcup_{i \in \Nat} \Lie(G_i)$.

\item \label {4805DA99-0AEF-43B5-B27F-53CC2E949B97}
Let $U \in \NC(X)$, let $T$ be a connected algebraic subgroup of $\Aut(X)$ that normalizes $U$ 
(i.e., $t U t^{-1} = U$ for all $t \in T$), and let $G$ be the subgroup of $\Aut(X)$ generated by $T \cup U$.
Then $G \in \NC(X)$ and $\Lie(G) = \Lie(T) + \Lie(U)$.

\item \label {15EE7F9A-4838-434C-A2D0-704C44E7CF31}
Each locally finite $\bk$-derivation of $B$ is an element of $\Lie(\Aut X)$.

\end{enumerate}
\end{lemma}

% \begin{proof}
% Part \ref{15EE7F9A-4838-434C-A2D0-704C44E7CF31} follows from either one of \cite[Corollary 7.6.4]{FurterKraft} or \cite[Theorem 1]{CD03}.
% The other assertions are folklore.
% \end{proof}

\begin{proof}
Part \ref{15EE7F9A-4838-434C-A2D0-704C44E7CF31} follows from either one of \cite[Corollary 7.6.4]{FurterKraft} or \cite[Theorem 1]{CD03}.
The other assertions are folklore, but we include the proof of (c) for the reader's convenience.

(c) Let $U_0 \subseteq U_1 \subseteq U_2 \subseteq \cdots$ be a sequence of connected algebraic subgroups of $\Aut(X)$ such that $U = \bigcup_{i \in \Nat} U_i$
and let $(V_i)_{i \in \Nat}$ be an admissible filtration of $\Aut(X)$.
Let $i \in \Nat$ and consider the subset $S_i = \setspec{ tut^{-1} }{ t \in T \text{ and } u \in U_i }$ of $U$.
Since $\phi : \Aut(X) \times \Aut(X) \to \Aut(X)$, $\phi(g,h) = ghg^{-1}$ is an ind-morphism,
there exists $j$ such that $S_i \subseteq U \cap V_j$;
also note that $S_i = \phi(T \times U_i)$ is connected, because $T \times U_i$ is connected.
Since $\nes U = \clo U$, $(U_n)_{n \in \Nat}$ and  $(U \cap V_n)_{n \in \Nat}$ are equivalent filtrations of $U$,
so $U \cap V_j \subseteq U_\ell$  for some $\ell \in \Nat$.
So we have $U_i \subseteq S_i \subseteq U_\ell$.
If $U_i'$ denotes the closure of the subgroup of $U_\ell$ generated by $S_i$, then $U_i'$ is connected because $S_i$ is connected and contains the identity element of $U_\ell$.
So $U_i'$ is a connected algebraic subgroup of $\Aut(X)$, $U_i \subseteq U_i' \subseteq U_\ell$, and $T$ normalizes $U_i'$.
This shows that there exists a sequence $U_0' \subseteq U_1' \subseteq U_2' \subseteq \cdots$ of connected algebraic subgroups of $\Aut(X)$
such that $U = \bigcup_{i \in \Nat} U_i'$ and such that $T$ normalizes each $U_i'$.

For each $n \in \Nat$, let $G_n$ be the subgroup of $\Aut(X)$ generated by $T \cup U_n'$.
Let $i \in \Nat$.
Since $T$ normalizes $U_i'$, we have $G_i = \psi(T \times U_i')$ where $\psi : \Aut(X) \times \Aut(X) \to \Aut(X)$, $\psi(g,h) = gh$.
As $\psi$ is an ind-morphism, it follows that $G_i \subseteq V_j$ for some $j$.
So the closure $\overline G_i$ of $G_i$ in $\Aut(X)$ is an algebraic subgroup of $\Aut(X)$.
Since $T$ and $U_i'$ are closed subgroups of the algebraic group $\overline G_i$ such that $T$ normalizes $U_i'$,
the subgroup that they generate (namely, $G_i$) is closed in $\overline G_i$,
because it is the image of the morphism of algebraic groups $T \ltimes U_i' \to \overline G_i$, $(t,u) \mapsto tu$.
Hence, $G_i$ is closed in $\Aut(X)$.
Since $T \cup U_i'$ is connected and contains the identity element, $G_i$ is connected.
So $G_0 \subseteq G_1 \subseteq G_2 \subseteq \cdots$ are connected algebraic subgroups of $\Aut(X)$.
Since $G = \bigcup_{i \in \Nat} G_i$, we have $G \in \NC(X)$ by \ref{7BACADD4-796D-4491-8D0A-93D9F0132FFC}(a).
Also, $\Lie(G) = \bigcup_{i \in \Nat} \Lie(G_i)$ and $\Lie(U) = \bigcup_{i \in \Nat} \Lie(U_i')$ by part \ref{1E3B0B16-D6E9-4E74-96A9-A42E3459F98F}.
Since the result that we are proving is known to be true in the context of algebraic groups, we have $\Lie(G_i) = \Lie(T) + \Lie(U_i')$ for each $i \in \Nat$,
and consequently $\Lie(G) = \Lie(T) + \Lie(U)$.
This proves (c).
\end{proof}

\begin{noname} \label {groupLieLNcorrespondence}
We recall the terminology and fix our notation for unipotent automorphisms and the exponential and logarithm maps.
For each $D \in \lnd(B)$, the map $\exp(D) : B \to B$ given by $b \mapsto \sum_{n = 0}^\infty \frac{1}{n!}D^n(b)$
is a $\bk$-automorphism of $B$.
An automorphism $\theta \in \Aut_\bk(B)$ is {\it unipotent\/} if the following equivalent conditions are satisfied:
\begin{enumerate}[\rm(i)]
\item the linear map $\theta - \Id_B : B \to B$ is locally nilpotent;
\item there exists $D \in \lnd(B)$ such that $\exp(D) = \theta$.
\end{enumerate}
Let $\Aut_{\text{\rm u}}(B)$ be the set of unipotent elements of $\Aut_\bk(B)$.
If $\theta \in \Aut_{\text{\rm u}}(B)$ then the map $\log(\theta) : B \to B$ given by
$b \mapsto \sum_{n=1}^\infty \frac{(-1)^{n+1}}{n} (\theta - \Id_B)^n(b)$ is an element of $\lnd(B)$
such that $\exp( \log(\theta) ) = \theta$.
(For this fact and the equivalence of (i) and (ii), see \cite[Proposition 2.1.3]{van2000polynomial}.)
One can also show that $D \mapsto \exp(D)$ is an injective map,
so the maps 
$\xymatrix@1{ \lnd(B) \ar @<.4ex>[r]^-{ \exp }  & \Aut_{\text{\rm u}}(B) \ar @<.4ex>[l]^-{ \log } }$
are bijective and inverse of each other.

Recall that $\Aut_\bk(B) \to \Aut(X)$, $\theta \mapsto \Spec(\theta)$, is an anti-isomorphism of groups.
An automorphism of $X$ is called {\it unipotent\/} if it is equal to $\Spec(\theta)$ for some $\theta \in \Aut_{\text{\rm u}}(B)$,
and we write $\Aut_{\text{\rm u}}(X)$ for the set of unipotent elements of $\Aut(X)$.\footnote{This is equivalent to the definition given at the beginning of the Introduction, but the present definition is the only one we use.}
A subgroup of $\Aut(X)$ is said to be {\it unipotent\/} if all its elements are unipotent.
We define $\EXP : \lnd(B) \to \Aut_{\text{\rm u}}(X)$ and $\LOG : \Aut_{\text{\rm u}}(X) \to \lnd(B)$ to be the  compositions
$\lnd(B) \xrightarrow{ \exp } \Aut_{\text{\rm u}}(B) \xrightarrow{\Spec} \Aut_{\text{\rm u}}(X)$
and $\Aut_{\text{\rm u}}(X) \xrightarrow{\Spec^{-1}} \Aut_{\text{\rm u}}(B) \xrightarrow{ \log } \lnd(B)$, respectively.
To summarize, we have:
\begin{equation}  \label {62E66CBC-68A6-4EFF-8124-66672AAB79CF}
\xymatrix{ \lnd(B) \ar @<.5ex>[r]^-{ \exp } \ar@/_2em/@<.5ex>[dr]^-{\EXP} & \Aut_{\text{\rm u}}(B) \ar @<.5ex>[l]^-{ \log } \ar[d]^-{\Spec} \\
&  \Aut_{\text{\rm u}}(X) \ar@/^2em/@<.5ex>[ul]^-{\LOG} 
}
\end{equation}
where all maps are bijective,
$\exp$ and $\log$ are inverse of each other,
$\EXP$ and $\LOG$ are inverse of each other,
$\EXP = \Spec \circ \exp$ and $\log = \LOG \circ \Spec$.
Given subsets $\Lscr \subseteq \lnd(B)$ and $U \subseteq\Aut_{\text{\rm u}}(X)$, we define
$\EXP(\Lscr) = \setspec{ \EXP(D) }{ D \in \Lscr }$ 
and $\LOG(U) = \setspec{ \LOG(u) }{ u \in U }$.
\end{noname}

\begin{lemma} \label {unipotentLieLNDiscussion}
Let $\NU(X)$ denote the set of all nested unipotent subgroups of $\Aut(X)$.
\begin{enumerate}[\rm(a)]

\item $\NU(X) \subseteq \NC(X)$

\item If $U \in \NU(X)$ then $\Lie(U) = \LOG(U)$.

\item  \label {4BFE0092-4242-48D7-A3D3-0C484D4BC2B5}
The maps \ \ $\xymatrix@1{ \LieLN_\bk(B) \ar @<.4ex>[r]^-{ \EXP }  & \NU(X) \ar @<.4ex>[l]^-{ \Lie } }$ \ \ are bijective and inverse of each other.

\end{enumerate}
\end{lemma}

\begin{proof}
We prove (a) and (b) together.
Let $U \in \NU(X)$ and consider a sequence
$U_0 \subseteq U_1 \subseteq U_2 \subseteq \cdots$ of algebraic subgroups of $\Aut(X)$ such that $U = \bigcup_{i \in \Nat} U_i$.
Since unipotent algebraic groups are connected, each $U_i$ is connected, so $U \in \NC(X)$ by \ref{7BACADD4-796D-4491-8D0A-93D9F0132FFC}(a).
Since the fact that $\Lie(U) = \LOG(U)$ is well known in the context of algebraic groups, we have $\Lie(U_i) = \LOG(U_i)$ for all $i \in \Nat$.
We have $\Lie(U) = \bigcup_{i \in \Nat} \Lie(U_i)$ by Lemma \ref{29B1EE01-E26F-43CA-ADB2-5A0E15668FCE}\ref{1E3B0B16-D6E9-4E74-96A9-A42E3459F98F},
and it is obvious that $\LOG(U) = \bigcup_{i \in \Nat} \LOG(U_i)$,
so $\Lie(U) = \LOG(U)$. This proves (a) and (b).

(c) Proposition 3.5 of \cite{perepechko2025} shows that  $\EXP : \LieLN_\bk(B) \to \NU(X)$ is a well-defined bijection.
The inverse of this map is $\LOG: \NU(X) \to \LieLN_\bk(B)$, because $\EXP$ and $\LOG$ in \eqref{62E66CBC-68A6-4EFF-8124-66672AAB79CF} are bijective and inverse of each other.
In view of (b), this is the desired conclusion.
\end{proof}

\medskip

Paragraphs \ref{FD621230-87A8-49EC-B5C9-5E39E7C97DCE}--\ref{92E834E4-3895-42FF-A577-61C5AF52410C} recall how $\Integ^n$-gradings are related to torus actions
and how to compute the Lie algebra of a torus.  This being very well known, we omit the proofs.

\begin{noname} \label {FD621230-87A8-49EC-B5C9-5E39E7C97DCE}
Let $\ggoth = (B_i)_{i \in G}$ be a grading of $B$ by an abelian group $G$.
Consider the submonoid $\bbM = \bbM(\ggoth) = \setspec{ i \in G }{ B_i \neq 0 }$ of $G$ and its multiplicative character group $\Hom( \bbM , \bk^* )$.
For each character $\chi : \bbM \to \bk^*$, we define the $\bk$-automorphism $\Theta_\ggoth^\chi : B \to B$ by declaring that, for each $i \in \bbM$,
$\Theta_\ggoth^\chi |_{B_i} : B_i \to B_i$ is the multiplication by $\chi(i)$.
We define:
$$
\bbT_{\ggoth} = \setspec{ \Theta_\ggoth^\chi }{ \chi \in \Hom( \bbM , \bk^* ) } \subseteq \Aut_\bk(B)
\quad \text{and} \quad
T_\ggoth = \setspec{ \Spec(\Theta) }{ \Theta \in \bbT_\ggoth }  \subseteq \Aut(X) .
$$
Note that $\Theta_\ggoth^{\chi_1 \cdot \chi_2} = \Theta_\ggoth^{\chi_1} \circ \Theta_\ggoth^{\chi_2}$ for all $\chi_1,\chi_2 \in \Hom(\bbM,\bk^*)$,
so $\chi \mapsto \Theta_\ggoth^\chi$ is a group homomorphism $\Hom( \bbM , \bk^* ) \to \Aut_\bk(B)$ whose image is $\bbT_\ggoth$.
So $\bbT_\ggoth$ is a subgroup of $\Aut_\bk(B)$ and hence $T_\ggoth$ is a subgroup of $\Aut(X)$.
Since $\Hom( \bbM , \bk^* )$ is an abelian group, so are $\bbT_\ggoth$ and $T_\ggoth$.
\end{noname}

\begin{definition}
Recall that each algebraic subgroup of $\Aut(X)$ is naturally endowed with the structure of an algebraic group;
by an {\it algebraic torus of $\Aut(X)$} we mean an algebraic subgroup $T$ of $\Aut(X)$ which, endowed with this structure,
is isomorphic as an algebraic group to a torus $\bbG_{\text{\rm m}}^{\, n}$ for some $n \in \Nat$.
\end{definition}

\begin{noname}  \label {92E834E4-3895-42FF-A577-61C5AF52410C}
% Note to myself: I proved this in Proposition 14.5 of my personal notes ``ReadingFurKraft.pdf''.
\begin{enumerate}[\rm(a)]

\item Let $n \in \Nat$ and let $\ggoth$ be a $\Integ^n$-grading of $B$.
Then $T_\ggoth$ is an algebraic torus of $\Aut(X)$ isomorphic to $\bbG_{\text{\rm m}}^{\, r}$ where $r = \rank(\ggoth)$; that is, $\dim T_\ggoth = \rank(\ggoth)$.
Moreover, we have the equality $\Lie\big( T_\ggoth \big) = L(\ggoth)$ of Lie subalgebras of $\Der_\bk(B)$,
where $L(\ggoth)$ is defined in \ref{FA176F6C-52D4-43FA-A7AF-B41C94E23E8B}.

\item Let $T$ be an algebraic torus of $\Aut(X)$ and let $d = \dim(T)$.
Then the natural action $T \times X \to X$,  $(t,x) \mapsto t(x)$, is an algebraic action of $\bbG_{\text{\rm m}}^{\, d}$ on $X$.
This action determines a $\Integ^d$-grading $\ggoth$ of $B$, and we have the equality $T = T_\ggoth$ of subgroups of $\Aut(X)$.
By (a), it follows that
$$
\setspec{ T_\ggoth }{ \text{$\ggoth$ is a $\Integ^n$-grading of $B$ for some $n \in \Nat$} }
$$
is equal to the set of all algebraic tori of $\Aut(X)$.

\end{enumerate}
\end{noname}

\begin{lemma} \label {1E03B95D-4E12-4AA5-8D53-357AA4219E2B}
Let $d \in \Nat$, let $\ggoth$ be a $\Integ^d$-grading of $B$, and consider the algebraic torus $T_\ggoth$ of $\Aut(X)$.
Let $\Lscr \in \LieLN_\bk(B)$ and $U = \EXP(\Lscr) \in \NU(X)$.
Then $\Lscr$ is $\ggoth$-homogeneous if and only if $U$ is normalized by $T_\ggoth$ (i.e., $t U t^{-1} = U$ for all $t \in T_\ggoth$).
\end{lemma}

\begin{proof} 
This is folklore, but we include a proof for completeness.
Given $\bt = (t_1,\dots,t_d) \in (\bk^*)^d$ and $\bi = (i_1,\dots,i_d) \in \Integ^d$, we define $\bt^\bi = t_1^{i_1} \cdots t_d^{i_d} \in \bk^*$.
For each $\bt \in (\bk^*)^d$, let $\phi_\bt : B \to B$ be the $\bk$-automorphism defined by declaring that,
for each $\bi \in \Integ^d$, $\phi_\bt |_{B_\bi} : B_\bi \to B_\bi$ is the multiplication by $\bt^\bi$.
Then $\bbT_\ggoth = \setspec{ \phi_\bt }{ \bt \in (\bk^*)^d }$ and  $T_\ggoth = \setspec{ \Spec(\phi_\bt) }{ \bt \in (\bk^*)^d }$.
It is clear that $U$ is normalized by $T_\ggoth$ if and only if $\exp(\Lscr)$ is normalized by $\bbT_\ggoth$
(meaning that $\phi_\bt \circ \exp(D) \circ \phi_\bt^{-1} \in \exp(\Lscr)$ for every $D \in \Lscr$ and $\bt \in (\bk^*)^d$),
if and only if
\begin{equation} \label {8BCF4C2A-4A5E-4292-BCDB-788C60BC25D8}
\text{for every $D \in \Lscr$ and $\bt \in (\bk^*)^d$, \quad $\phi_\bt \circ D \circ \phi_\bt^{-1} \in \Lscr$.}
\end{equation}
So we have to show that \eqref{8BCF4C2A-4A5E-4292-BCDB-788C60BC25D8} is equivalent to $\Lscr$ being $\ggoth$-homogeneous.
Let $D \in \Lscr$ and let $D = \sum_{j=1}^n D_j$ be the homogeneous decompositon of $D$,
where each $D_j$ is nonzero and $\ggoth$-homogeneous of degree $\bi_j$, and where $\bi_1, \dots, \bi_n$ are distinct elements of $\Integ^d$.
To finish the proof, it suffices to show that 
\begin{equation}  \label {EC8396F0-EF4A-416A-845A-0C0F9FA03D27}
\text{$\phi_\bt \circ D \circ \phi_\bt^{-1} \in \Lscr$ for all $\bt \in (\bk^*)^d$ \ $\iff$ \ $D_1, \dots, D_n \in \Lscr$.}
\end{equation}
A straightforward calculation shows that $\phi_\bt \circ D \circ \phi_\bt^{-1} = \sum_{j=1}^n \bt^{\bi_j} D_j$ for all $\bt \in (\bk^*)^d$.
So $(\Leftarrow)$ is clearly true, and $(\Rightarrow)$ is a consequence of the following statement, whose proof we leave to the reader:
{\it Let $V$ be a vector space over $\bk$ and $L$ a subspace of $V$.
Let  $\bi_1, \dots, \bi_n$ be distinct elements of $\Integ^d$.
If $v_1,\dots,v_n \in V$ are such that $\sum_{j=1}^n \bt^{\bi_j} v_j \in L$ for all $\bt \in (\bk^*)^d$, then $v_1,\dots,v_n \in L$.}
\end{proof}

The following is a special case of \cite[Corollary B and Theorem D]{solvableAutGroups} and plays an essential role in our proof of the main results of this section:

\begin{theorem}  \label {A38C833F-DB8D-45CD-9664-1D5AE5BE00CF}
Let $H$ be a solvable subgroup of $\Aut(X)$ that is closed and connected in the ind-topology of $\Aut(X)$.
Then the set $\urad H = H \cap \Aut_{\text{\rm u}}(X)$ is a nested unipotent subgroup of $\Aut(X)$ and
there exists an algebraic torus $T$ of $\Aut(X)$ that normalizes $\urad H$ and is such that $H = T \ltimes \urad H$.

{\rm The normal subgroup $\urad H = H \cap \Aut_{\text{\rm u}}(X)$ of $H$ is called the {\it unipotent radical\/} of $H$.}
\end{theorem}

Corollary B of \cite{solvableAutGroups} is stronger than Theorem \ref{A38C833F-DB8D-45CD-9664-1D5AE5BE00CF} in several respects.
Their result holds over a field of arbitrary characteristic, and their variety $X$ is quasi-affine and not necessarily irreducible
(compare with our global assumptions on $\bk$, $B$, and $X$ at the beginning of this section).
Our assumption that $H$ is closed and connected in the ind-topology of $\Aut(X)$ is stronger than the corresponding assumption in \cite{solvableAutGroups}.
Also, our conclusion that $\urad H$ is a nested unipotent subgroup of $\Aut(X)$ follows from their conclusion together with \cite[Theorem D]{solvableAutGroups}.

\begin{noname} \label {C14524DE-554D-4B0D-A527-D763AFE4A977}
By Lemma \ref{29B1EE01-E26F-43CA-ADB2-5A0E15668FCE}\ref{4805DA99-0AEF-43B5-B27F-53CC2E949B97},
the groups $H$, $\urad H$ and $T$ that appear in the statement of Theorem \ref{A38C833F-DB8D-45CD-9664-1D5AE5BE00CF}
have the following properties:
\begin{equation}  \label {614A02E0-C9C8-49D5-9285-451382A1986A}
H \in \NC(X) \quad \text{and} \quad \Lie(H) = \Lie(T) + \Lie(\urad H) .
\end{equation}
Note that Theorem \ref{A38C833F-DB8D-45CD-9664-1D5AE5BE00CF} and \eqref{614A02E0-C9C8-49D5-9285-451382A1986A} are valid in 
particular when $H$ is a Borel subgroup of $\Aut(X)$.
(Recall that a {\it Borel subgroup\/} of $\Aut(X)$ is a maximal element of the set of all solvable connected subgroups of $\Aut(X)$. 
Since the closure of a solvable connected subgroup of $\Aut(X)$ is a solvable connected subgroup of $\Aut(X)$, it follows that 
Borel subgroups are closed in the ind-topology of $\Aut(X)$ and hence satisfy the hypothesis of Theorem \ref{A38C833F-DB8D-45CD-9664-1D5AE5BE00CF}.)
\end{noname}

\begin{noname}  \label {4B7B6F7E-43D3-4719-91CA-E6301252556A}
In \cite{Arz_Zaid_2026}, a Lie subalgebra of $\Lie( \Aut X )$ is called {\it locally integrable\/} if and only if it is weakly locally finite,
and a {\it locally integrable Borel subalgebra of $\Lie( \Aut X )$} is defined to be a maximal element of the set
of Lie subalgebras of $\Lie( \Aut X )$ that are locally integrable and solvable.
On the other hand, Lemma \ref{29B1EE01-E26F-43CA-ADB2-5A0E15668FCE}\ref{15EE7F9A-4838-434C-A2D0-704C44E7CF31} implies that
each element of $\LieWLF_\bk(B)$  is a Lie subalgebra of $\Lie(\Aut X)$.
So the following equalities hold, and constitute a dictionary between our notation and the terminology of \cite{Arz_Zaid_2026}: 
\begin{align*}
\LieWLF_\bk(B) &= \text{ set of Lie subalgebras of $\Lie( \Aut X )$ that are locally integrable,}  \\
\LieWLF_\bk(B)_{\text{\rm sol}} &= \text{ set of Lie subalgebras of $\Lie( \Aut X )$ that are locally integrable and solvable,} \\
\max\LieWLF_\bk(B)_{\text{\rm sol}} &= \text{ set of \textbf{locally integrable Borel subalgebras} of $\Lie( \Aut X )$.}
\end{align*}
Translating the results of \cite{Arz_Zaid_2026} into the language of the present article, we find that
Lemma 3.3, Theorem 3.4 and Corollary 3.6 of \cite{Arz_Zaid_2026} have the following immediate consequence:
\end{noname}

\begin{theorem}[\cite{Arz_Zaid_2026}] \label {29BFE460-6348-4F3D-9CB7-B7FA0F1316A4}
Let $\Borel(X)$ denote the set of all Borel subgroups of $\Aut(X)$.
Then $\Beul \mapsto \Lie(\Beul)$ defines a bijection $\Borel(X) \to \max\LieWLF_\bk(B)_{\text{\rm sol}}$,
and two elements of $\Borel(X)$ are conjugate if and only if their images by this map are conjugate.
\end{theorem}

\begin{proposition}  \label {7B4377C8-6C88-47ED-98C3-C46D5A07A731}
Let $\Beul \in \Borel(X)$ and let $\urad{\Beul}$ be the unipotent radical of $\Beul$.
Consider the set $\Lgoth(B)$ defined in \ref{B6DFEABF-3379-4ECF-9C30-26F2BCDB4FD6}.
\begin{enumerate}[\rm(a)]

\item $\Lie(\Beul)$ is a maximal element of $\Lgoth(B)$ and $\Lie(\Beul) \cap \lnd(B) = \Lie( \urad{\Beul} )$.

\item $\Lie( \urad{\Beul} )$ is a maximal element of $\LieLN_\bk(B)$.

\item $ \urad{\Beul} $ is a maximal element of $\NU(X)$.

\end{enumerate}
\end{proposition}

\begin{proof}
(a)  Theorem \ref{A38C833F-DB8D-45CD-9664-1D5AE5BE00CF} implies that $\urad{\Beul} \in \NU(X)$
and that $\Beul = T \ltimes  \urad{\Beul} $ for some algebraic torus $T \subseteq \Aut(X)$ that normalizes $\urad{\Beul}$.
By \ref{92E834E4-3895-42FF-A577-61C5AF52410C}(b), there exist $d \in \Nat$ and a $\Integ^d$-grading $\ggoth$ of $B$ such that $T = T_\ggoth$,
and \ref{92E834E4-3895-42FF-A577-61C5AF52410C}(a) gives $\Lie(T) = L(\ggoth)$.
As  $\urad{\Beul} \in \NU(X)$,  Lemma \ref{unipotentLieLNDiscussion}\ref{4BFE0092-4242-48D7-A3D3-0C484D4BC2B5} implies that there exists $\Lscr \in \LieLN_\bk(B)$
such that $\EXP(\Lscr) = \urad{\Beul}$ and $\Lie(\urad{\Beul}) = \Lscr$.
Since $T$ normalizes $\urad{\Beul}$, Lemma \ref{1E03B95D-4E12-4AA5-8D53-357AA4219E2B} implies that $\Lscr$ is $\ggoth$-homogeneous;
so  $L(\ggoth) + \Lscr \in \Lgoth(B)$.
Assertion \eqref{614A02E0-C9C8-49D5-9285-451382A1986A} gives $\Lie(\Beul) = \Lie(T) + \Lie( \urad{\Beul} ) = L(\ggoth) + \Lscr \in \Lgoth(B)$,
so Lemma \ref{36F36D71-D42B-4E92-8612-63FD8842DBA9}(b) gives $\Lie(\Beul) \cap \lnd(B) = \Lscr = \Lie( \urad{\Beul} )$.
Theorem \ref{29BFE460-6348-4F3D-9CB7-B7FA0F1316A4} implies that $\Lie(\Beul)$ is a maximal element of $\LieWLF_\bk(B)_{\text{\rm sol}}$;
as $\Lie(\Beul) \in \Lgoth(B) \subseteq\LieWLF_\bk(B)_{\text{\rm sol}}$, it follows that $\Lie(\Beul)$ is a maximal element of $\Lgoth(B)$.
This proves (a).

Since $\Lie(\Beul) = L(\ggoth) + \Lscr$ is maximal in $\Lgoth(B)$,
Lemma \ref{36F36D71-D42B-4E92-8612-63FD8842DBA9}(c) implies that $\Lscr$ is maximal in $\LieLN_\bk(B)$, which proves (b).
By Lemma \ref{unipotentLieLNDiscussion}\ref{4BFE0092-4242-48D7-A3D3-0C484D4BC2B5}, it follows that $ \urad{\Beul} $ is a maximal element of $\NU(X)$, proving (c).
\end{proof}

\begin{corollary}  \label {CC2270E1-029B-4F26-89F8-81B03963C6CD}
Each maximal element of $\LieWLF_\bk(B)_{\text{\rm sol}}$ belongs to the set $\Lgoth(B)$.
In other words, the condition $(*)$ in Proposition \ref{CD178309-4EC6-4386-9C1E-34F625BEBC29} is satisfied.
\end{corollary}

\begin{proof}
By Theorem \ref{29BFE460-6348-4F3D-9CB7-B7FA0F1316A4}, each maximal element of $\LieWLF_\bk(B)_{\text{\rm sol}}$
is equal to $\Lie(\Beul)$ for some $\Beul \in \Borel(X)$, and  Proposition \ref{7B4377C8-6C88-47ED-98C3-C46D5A07A731}(a) gives $\Lie(\Beul) \in \Lgoth(B)$.
\end{proof}

The following is the main result of this section.
The assumptions on $\bk$, $B$ and $X$ are stated at the beginning of the section.

\begin{theorem}  \label {0CA62AB7-E680-4831-972F-E730B069503F}
Consider the diagram of set maps
\begin{equation}  \label {A3583570-EE9B-4EE9-883D-05372F0F5495}
\hspace{-12mm} \raisebox{5mm}{\scalebox{.95}{\xymatrix@R=7mm{
\Borel(X) \ar[d]_-{\Lie} \ar[r]^{ \rho_X }  &  \max \NU(X)  \ar @<-.4ex> [d]_-{\Lie}   \\
\max\LieWLF_\bk(B)_{\text{\rm sol}} \ar[r]_-{\rho_B}  &  \max\LieLN_\bk(B) \makebox[0mm][l]{$\, = \max\DJL(B)$} \ar @<-.4ex>[u]_-{ \EXP }
}}}
\end{equation}
defined by the following stipulations:
\begin{itemize}

\item The map $\rho_X$ sends each $\Beul \in \Borel(X)$ to its unipotent radical $ \urad{\Beul}  = \Beul \cap \Aut_{\text{\rm u}}(X)$.

\item The map $\rho_B$ is the bijection $\Mscr \mapsto \Mscr \cap \lnd(B)$ defined in Proposition \ref{CD178309-4EC6-4386-9C1E-34F625BEBC29}.

\item The left vertical arrow is the bijection $\Beul \mapsto \Lie(\Beul)$ of Theorem \ref{29BFE460-6348-4F3D-9CB7-B7FA0F1316A4}.

\item The two vertical arrows on the right are the restrictions of the two maps of Lemma \ref{unipotentLieLNDiscussion}\ref{4BFE0092-4242-48D7-A3D3-0C484D4BC2B5}.

\end{itemize}
Then the following hold:
\begin{enumerate}[\rm(a)]

\item If $f$ denotes any one of the five maps in \eqref{A3583570-EE9B-4EE9-883D-05372F0F5495} then:
\begin{enumerate}[\rm(i)]

\item $f$ is a well-defined bijection;
\item two elements of $\dom(f)$ are conjugate if and only if their images by $f$ are conjugate.

\end{enumerate}

\item The diagram is commutative, i.e., we have $\Lie \circ \rho_X = \rho_B \circ \Lie$ and the two vertical arrows on the right are inverse of each other.

\item Given  $\Beul \in \Borel(X)$ and $U \in\max\NU(X)$, we have $U \subseteq \Beul$ $\Leftrightarrow$ $\rho_X(\Beul) = U$.

\item Given $U \in \max\NU(X)$, exactly one Borel subgroup of $\Aut(X)$ contains $U$.

\end{enumerate}
\end{theorem}

\begin{proof}
The equality in the lower-right corner of \eqref{A3583570-EE9B-4EE9-883D-05372F0F5495} is Proposition \ref{6D6E22A4-236A-4B28-96D9-86157812D34A}(c).
By Proposition {\rm\ref{7B4377C8-6C88-47ED-98C3-C46D5A07A731}(c)}, $\rho_X$ is a well-defined map. 
By Lemma \ref{unipotentLieLNDiscussion}\ref{4BFE0092-4242-48D7-A3D3-0C484D4BC2B5},
it is clear that the two vertical arrows on the right are well-defined bijective maps and are inverse of each other.
The left vertical arrow is a bijection by Theorem \ref{29BFE460-6348-4F3D-9CB7-B7FA0F1316A4}.
By Corollary \ref{CC2270E1-029B-4F26-89F8-81B03963C6CD}, the hypothesis of Proposition \ref{CD178309-4EC6-4386-9C1E-34F625BEBC29} is satisfied,
so the conclusions of that result are valid; in particular, $\rho_B$ is a well-defined bijection.

If $\Beul \in \Borel(X)$ then $\rho_B\big( \Lie(\Beul) \big) = \Lie(\Beul) \cap \lnd(B) = \Lie( \urad{\Beul} )$,
the last equality by Proposition \ref{7B4377C8-6C88-47ED-98C3-C46D5A07A731}(a).
This shows that $\rho_B \circ \Lie = \Lie \circ \rho_X$, so diagram \eqref{A3583570-EE9B-4EE9-883D-05372F0F5495} is commutative.
It follows that $\rho_X$ is bijective, because all the other maps in the diagram are bijective.
So far, we proved (a-i) and (b). Assertions (c) and (d) trivially follow.
% Next, we prove (c).
% Let  $\Beul \in \Borel(X)$ and $U \in\max\NU(X)$.
% If $U \subseteq \Beul$ then $U \subseteq \Beul \cap \Aut_{\text{\rm u}}(X) = \urad{\Beul}$
% and $ \urad{\Beul}  \in \NU(X)$ by Proposition \ref{7B4377C8-6C88-47ED-98C3-C46D5A07A731}(c), 
% so $ \urad{\Beul}  = U$ (by maximality of $U$) and hence $\rho_X(\Beul) = U$.
% The converse being clear, we have $U \subseteq \Beul$ $\Leftrightarrow$ $\rho_X(\Beul) = U$, proving (c).
% Since $\rho_X$ is bijective, (d) follows from (c).

It remains to prove assertion (a-ii).
We first show that $\Beul, \Beul' \in \Borel(X)$ are conjugate if and only if $\rho_X(\Beul) =  \urad{\Beul} $ is conjugate to $\rho_X(\Beul') =  \urad{\Beul} '$.
Suppose that $\theta \in\Aut(X)$ is such that $\theta \Beul \theta^{-1} = \Beul'$.
Since $ \urad{\Beul} $ is the set of unipotent elements of $\Beul$,
it follows that $\theta  \urad{\Beul}  \theta^{-1}$ is the set of unipotent elements of $\theta \Beul \theta^{-1} = \Beul'$, i.e., 
$\theta  \urad{\Beul}  \theta^{-1} =  \urad{\Beul} '$.
Conversely, if $\theta \in\Aut(X)$ is such that $\theta  \urad{\Beul}  \theta^{-1} =  \urad{\Beul} '$
then both $\Beul'$ and $\theta \Beul \theta^{-1}$ are Borel subgroups of $\Aut(X)$ that contain $ \urad{\Beul} '$,
so (d) implies that $\theta \Beul \theta^{-1} = \Beul'$.

This shows that assertion (a-ii) is true for $\rho_X$.
Since (a-ii) is true for $\rho_B$ (by Proposition \ref{CD178309-4EC6-4386-9C1E-34F625BEBC29}(d))
and for the left vertical arrow (by Theorem \ref{29BFE460-6348-4F3D-9CB7-B7FA0F1316A4}),
the commutativity of the diagram implies that (a-ii) is also true for the two vertical arrows on the right.
This proves the Theorem.
\end{proof}

The next result complements the above Theorem by making the inverse maps $\rho_X^{-1}$ and $\rho_B^{-1}$ of
diagram \eqref{A3583570-EE9B-4EE9-883D-05372F0F5495} completely explicit. Theorem
\ref{0CA62AB7-E680-4831-972F-E730B069503F} shows that a given $U \in \max\NU(X)$ lies in a unique Borel
subgroup; we now construct that Borel subgroup explicitly from $U$ (or, in Lie-algebraic terms, from
$\Lscr = \Lie(U) \in \max\LieLN_\bk(B)$) using the tori that normalize $U$ and the gradings in
$\Grad(\Lscr)$. The proposition also establishes some additional facts concerning conjugation.
See \ref{B77C75D9-E381-4069-A82A-870BECEA6849} for the definitions of $\grank(\Lscr)$ and $\Grad(\Lscr)$.

\begin{proposition}  \label{38182F16-1D38-42ED-9BA4-C7CB5326711E}
Let $U \in \NU(X)$ and $\Lscr \in \LieLN_\bk(B)$ be such that $\EXP(\Lscr)=U$ and $\Lie(U)=\Lscr$.\footnote{Recall from
Lemma \ref{unipotentLieLNDiscussion}\ref{4BFE0092-4242-48D7-A3D3-0C484D4BC2B5} that  
$\xymatrix@1{ \LieLN_\bk(B) \ar @<.4ex>[r]^-{ \EXP }  & \NU(X) \ar @<.4ex>[l]^-{ \Lie } }$
are mutually inverse bijections.}
Let $\Tscr(U)$ be the set of all algebraic tori in $\Aut(X)$ that normalize $U$.
\begin{enumerate}[\rm(a)]

\item The rule $\ggoth \mapsto T_\ggoth$ is a well-defined surjective map $\Grad(\Lscr) \to \Tscr(U)$,
and $\dim T_\ggoth = \rank(\ggoth)$ for every $\ggoth \in \Grad(\Lscr)$.
Consequently, $\max\setspec{ \dim T }{ T \in \Tscr(U) } = \grank(\Lscr)$.

\item The following conditions are equivalent:
\begin{enumerate}[\rm(i)]

\item There exists $T \in \Tscr(U)$ such that $\langle T , U \rangle_{\text{\rm gp}}$ is a Borel subgroup of $\Aut(X)$;

\item $U$ is maximal in $\NU(X)$;

\item $\Lscr$ is maximal in $\LieLN_\bk(B)$.

\end{enumerate}
\end{enumerate}

\noindent In {\rm(c--f)},
assume that the equivalent conditions {\rm(i--iii)} of {\rm(b)} are satisfied.

\begin{enumerate}[\rm(a)]
\addtocounter{enumi}{2}

\item For each $T \in \Tscr(U)$, the following conditions are equivalent:

\begin{enumerate}[\rm(i)]

\item $\langle T , U \rangle_{\text{\rm gp}}$ is a Borel subgroup of $\Aut(X)$;

\item $T$ is maximal in $\Tscr(U)$;

\item $\dim T = \grank(\Lscr)$.

\end{enumerate}

\item  \label {95CB7549-845E-436A-BD45-B14B7DC28F63}
For each $\ggoth \in \Grad(\Lscr)$, the following are equivalent:
\begin{enumerate}[\rm(i)]

\item $L(\ggoth) + \Lscr \in \max\LieWLF_\bk(B)_{\text{\rm sol}}$

\item $\rank(\ggoth) = \grank(\Lscr)$.

\end{enumerate}

\item \label {CFA2596B-55F6-4864-9F15-0E5004E73AA2}
  If $T,T' \in \max\Tscr(U)$ then there exists $u \in U$ such that $u T u^{-1} = T'$.

\item \label{8E3DCA2F-E04D-4436-8149-693BB7FDEAE1}
If $\ggoth,\hgoth \in \Grad(\Lscr)$ and $\rank(\ggoth) = \rank(\hgoth) = \grank(\Lscr)$ then there exists $D \in \Lscr$
such that the automorphism $\epsilon = \exp(D)$ of $B$ satisfies $\epsilon \, \bbT_\ggoth \, \epsilon^{-1} = \bbT_\hgoth$.

\end{enumerate}
\end{proposition}

\begin{proof}
By \ref{92E834E4-3895-42FF-A577-61C5AF52410C} and Lemma \ref{1E03B95D-4E12-4AA5-8D53-357AA4219E2B},
$\ggoth \mapsto T_\ggoth$ is a well-defined surjective map $\Grad(\Lscr) \to \Tscr(U)$.
Since \ref{92E834E4-3895-42FF-A577-61C5AF52410C}(a) gives  $\dim T_\ggoth = \rank(\ggoth)$ for all $\ggoth \in \Grad(\Lscr)$, we have
$\max\setspec{ \dim T }{ T \in \Tscr(U) } = \max\setspec{ \rank(\ggoth) }{ \ggoth \in \Grad(\Lscr) } = \grank(\Lscr)$, proving (a).

(b) Observe that if $T \in \Tscr(U)$ then $T \cap U = \{ \id_X \}$, so the subgroup $G=\langle T , U \rangle_{\text{\rm gp}}$ of $\Aut(X)$
is the semi-direct product $T \ltimes U$.
It is clear that $U \subseteq \urad G$; if  $U \subset \urad G$ is a strict inclusion then the basic properties of the semi-direct product show that
$G = T \ltimes U \subset T \ltimes \urad G = G$ is strict, which is absurd.  So $\urad G = U$.
By diagram \eqref{A3583570-EE9B-4EE9-883D-05372F0F5495},
if $G$ is a Borel subgroup of $\Aut(X)$ then $\urad G = U$ is maximal in $\NU(X)$,
which proves that (i)$\Rightarrow$(ii).
Conversely, if $U \in \max\NU(X)$ then surjectivity of $\rho_X$ in \eqref{A3583570-EE9B-4EE9-883D-05372F0F5495} implies that there exists $\Beul \in \Borel(X)$
such that $\urad{\Beul} = U$, and Theorem \ref{A38C833F-DB8D-45CD-9664-1D5AE5BE00CF} shows that there exists $T \in \Tscr(U)$ such
that $\Beul = \langle T , U \rangle_{\text{\rm gp}}$.  So (ii)$\Rightarrow$(i).
The equivalence  (ii)$\Leftrightarrow$(iii) follows from  \eqref{A3583570-EE9B-4EE9-883D-05372F0F5495}, so (b) is proved.

For the reminder of the proof, assume that the equivalent conditions {\rm(i--iii)} of {\rm(b)} are satisfied.
By Theorem \ref{0CA62AB7-E680-4831-972F-E730B069503F}, exactly one $\Beul \in \Borel(X)$ contains $U$, and this $\Beul$ satisfies $\urad{\Beul} = U$.
We claim:
\begin{equation}  \label {7BA0B7B2-A444-4876-A480-8B3553BE1DD3}
\begin{minipage}{.8\textwidth}
For each $T \in \Tscr(U)$ we have $\langle T , U \rangle_{\text{\rm gp}} \subseteq \Beul$,
and moreover $\langle T , U \rangle_{\text{\rm gp}} = \Beul$ implies $T \in \max\Tscr(U)$.
\end{minipage}
\end{equation}
Indeed, let $T \in \Tscr(U)$.
The subgroup $\langle T , U \rangle_{\text{\rm gp}} = T \ltimes U$ of $\Aut(X)$ is connected and solvable,
and hence is included  in some Borel subgroup of $\Aut(X)$ by \cite[Corollary 6.22]{solvableAutGroups};
since $\Beul$ is the only Borel that contains $U$, we have $\langle T , U \rangle_{\text{\rm gp}} \subseteq \Beul$.
We also note that if  $T \subset T'$ is a strict inclusion of elements of $\Tscr(U)$ then
$\langle T , U \rangle_{\text{\rm gp}} = T \ltimes U \subset T' \ltimes U = \langle T' , U \rangle_{\text{\rm gp}}$ is a strict inclusion,
by properties of the semi-direct product. This proves \eqref{7BA0B7B2-A444-4876-A480-8B3553BE1DD3}.
We shall now prove \ref{CFA2596B-55F6-4864-9F15-0E5004E73AA2}, (c), \ref{8E3DCA2F-E04D-4436-8149-693BB7FDEAE1} and \ref{95CB7549-845E-436A-BD45-B14B7DC28F63}, in this order.

\ref{CFA2596B-55F6-4864-9F15-0E5004E73AA2}
Theorem \ref{A38C833F-DB8D-45CD-9664-1D5AE5BE00CF} implies that $\Beul = \langle T_0 , U \rangle_{\text{\rm gp}}$ for some $T_0 \in \Tscr(U)$,
and we have $T_0 \in \max\Tscr(U)$ by \eqref{7BA0B7B2-A444-4876-A480-8B3553BE1DD3}.
Also consider an arbitrary $T \in \max\Tscr(U)$.
Since $\Beul \in \NC(X)$ by \eqref{614A02E0-C9C8-49D5-9285-451382A1986A}, we may consider a sequence
$B_0 \subseteq B_1 \subseteq B_2 \subseteq \cdots$ of connected algebraic subgroups of $\Beul$ such that $\Beul = \bigcup_{n \in \Nat}B_n$.
There is a finite subset $S$ of $T$ such that the only closed subgroup of $T$ that contains $S$ is $T$ itself;
this implies that $T \subseteq B_m$ for some  $m \in \Nat$.
Applying the same argument to $T_0$ shows that there exists  $n \in \Nat$ such that $T_0 \cup T \subseteq B_n$.
So we found a connected algebraic subgroup $G$ of $\Aut(X)$ such that $T_0 \cup T \subseteq G \subseteq \Beul$.
Since $\Beul = T_0 \ltimes U$, we have $U \lhd \Beul$; 
as $G \subseteq \Beul$, it follows that each torus of $G$ normalizes $U$ and hence belongs to $\Tscr(U)$.
This implies that $T_0$ and $T$ are maximal tori of $G$.
Since all maximal tori of a connected algebraic group are conjugate, there exists $g \in G$ such that $g \, T_0 \, g^{-1} = T$.
Since $g \in \Beul = T_0 \ltimes U$, we have $g = ut$ for some $u \in U$ and $t \in T_0$.
Then $T = g \, T_0 \, g^{-1} = ut \, T_0 \, t^{-1} u^{-1}  = u \, T_0 \, u^{-1}$.
This shows that for each $T \in \max\Tscr(U)$, there exists $u \in U$ such that $u \, T_0 \, u^{-1} = T$.
Assertion \ref{CFA2596B-55F6-4864-9F15-0E5004E73AA2} easily follows from this.

(c) The fact that $u \, T_0 \, u^{-1} = T$ also implies that
$$
u \, \Beul \, u^{-1} = u \, \langle T_0 , U \rangle_{\text{\rm gp}} \, u^{-1} = \langle (u \, T_0 \, u^{-1}) , (u U u^{-1}) \rangle_{\text{\rm gp}}
= \langle T , U \rangle_{\text{\rm gp}} ,
$$
so $\langle T , U \rangle_{\text{\rm gp}}$ is a Borel subgroup of $\Aut(X)$ and hence $\Beul = \langle T , U \rangle_{\text{\rm gp}}$.
This being true for an arbitrary $T \in \max\Tscr(U)$, we obtain that (c-ii) implies (c-i).
By \eqref{7BA0B7B2-A444-4876-A480-8B3553BE1DD3}, we get (c-i)$\Leftrightarrow$(c-ii).

Note that if  $T \subset T'$ is a strict inclusion of elements of $\Tscr(U)$ then $\dim T < \dim T'$.
By (a), it follows that each $T \in \Tscr(U)$ satisfying $\dim T = \grank(\Lscr)$ is maximal in $\Tscr(U)$.
Since \ref{CFA2596B-55F6-4864-9F15-0E5004E73AA2} implies that all maximal elements of $\Tscr(U)$ have the same dimension,
we obtain (c-ii)$\Leftrightarrow$(c-iii), which completes the proof of (c).

To prove \ref{8E3DCA2F-E04D-4436-8149-693BB7FDEAE1},
consider $\ggoth,\hgoth \in \Grad(\Lscr)$ such that $\rank(\ggoth) = \rank(\hgoth) = \grank(\Lscr)$.
Note that $\dim T_\ggoth = \rank(\ggoth) = \grank(\Lscr)$ and $\dim T_\hgoth = \rank(\hgoth) = \grank(\Lscr)$;
since (c-iii) implies (c-ii), we see that $T_\ggoth$ and $T_\hgoth$ are maximal in $\Tscr(U)$.
By \ref{CFA2596B-55F6-4864-9F15-0E5004E73AA2}, we have $u T_\hgoth u^{-1} = T_\ggoth$ for some $u \in U$.
There exists $D \in \Lscr$ such that $u = \EXP(D)$.
Let $\epsilon = \exp(D) \in \Aut_{\text{\rm u}}(B)$; then $u = \Spec(\epsilon)$, so $\Spec(\epsilon) T_\hgoth \Spec(\epsilon^{-1}) = T_\ggoth$.
Since $f \mapsto \Spec(f)$ is an anti-isomorphism of groups from $\Aut_\bk(B)$ to $\Aut(X)$, we obtain
$\epsilon^{-1} \bbT_\hgoth \epsilon = \bbT_\ggoth$, which proves \ref{8E3DCA2F-E04D-4436-8149-693BB7FDEAE1}.

Before proving \ref{95CB7549-845E-436A-BD45-B14B7DC28F63}, define the set $ \Gscr = \setspec{ \ggoth \in \Grad(\Lscr) }{ L(\ggoth) + \Lscr = \Lie(\Beul) } $.
We claim:
\begin{equation}  \label {CDD1A23B-4D16-4616-BC24-117A25B847D7}
\text{$\rank(\ggoth) = \grank(\Lscr)$ \ \ for all $\ggoth \in \Gscr$.}
\end{equation}
Indeed, if $\ggoth \in \Gscr$ then $L(\ggoth) \cap \Lscr = \{0\}$,
so we have $\Lie(B) = L(\ggoth) + \Lscr = L(\ggoth) + \Lie(\urad{\Beul})$ where the sum is direct;
it follows that $\Lie(\Beul) / \Lie(\urad{\Beul}) \isom L(\ggoth)$ and hence $\rank(\ggoth) = \dim_\bk L(\ggoth) = \dim_\bk \big( \Lie(\Beul) / \Lie(\urad{\Beul}) \big)$
(the first equality was noted in \ref{FA176F6C-52D4-43FA-A7AF-B41C94E23E8B}).
So all elements of $\Gscr$ have the same rank and, to prove \eqref{CDD1A23B-4D16-4616-BC24-117A25B847D7},
it suffices to find $\ggoth_0 \in \Gscr$ such that $\rank(\ggoth_0) = \grank(\Lscr)$.
As we argued in the proof of \ref{CFA2596B-55F6-4864-9F15-0E5004E73AA2},
there exists $T_0 \in \max\Tscr(U)$ such that $\Beul = \langle T_0 , U \rangle_{\text{\rm gp}}$,
and we have $\dim T_0 = \grank(\Lscr)$ by (c).
Recall that $\ggoth \mapsto T_\ggoth$ is a surjection $\Grad(\Lscr) \to \Tscr(U)$ and
choose $\ggoth_0 \in \Grad(\Lscr)$ such that $T_{\ggoth_0} = T_0$.
By \eqref{614A02E0-C9C8-49D5-9285-451382A1986A} and \ref{92E834E4-3895-42FF-A577-61C5AF52410C}(a), 
$\Lie(\Beul) = \Lie(T_0) + \Lie(\urad{\Beul}) = \Lie( T_{\ggoth_0} ) + \Lscr = L(\ggoth_0) + \Lscr$, so $\ggoth_0 \in \Gscr$.
Assertion \ref{92E834E4-3895-42FF-A577-61C5AF52410C}(a) also gives the first equality in
 $\rank(\ggoth_0) = \dim T_{\ggoth_0} = \dim T_0 = \grank(\Lscr)$,
so \eqref{CDD1A23B-4D16-4616-BC24-117A25B847D7} is proved.

We prove \ref{95CB7549-845E-436A-BD45-B14B7DC28F63}.
By Corollary \ref{CC2270E1-029B-4F26-89F8-81B03963C6CD}, the hypothesis $(*)$ of Proposition \ref{CD178309-4EC6-4386-9C1E-34F625BEBC29} is satisfied,
so part (b) of that Proposition shows that (f-ii) implies (f-i).
Conversely, suppose that (f-i) holds, i.e., $L(\ggoth) + \Lscr \in \max\LieWLF_\bk(B)_{\text{\rm sol}}$.
Commutativity of diagram \eqref{A3583570-EE9B-4EE9-883D-05372F0F5495} gives
$\rho_B \Lie (\Beul) = \Lie \rho_X(\Beul) = \Lie(U) = \Lscr = \rho_B\big( L(\ggoth) + \Lscr \big)$,
where the last equality is Lemma \ref{36F36D71-D42B-4E92-8612-63FD8842DBA9}(b);
so $\Lie(\Beul) = L(\ggoth) + \Lscr$ by injectivity of $\rho_B$.
It follows that $\ggoth \in \Gscr$, so \eqref{CDD1A23B-4D16-4616-BC24-117A25B847D7} gives $\rank(\ggoth) = \grank(\Lscr)$.
\end{proof}

\begin{noname}  \label {347219A2-47B7-44CA-B8EF-D539653C1734}
{\bf Classifying Borel subgroups of $\Aut(X)$. }
It is clear from Theorem \ref{0CA62AB7-E680-4831-972F-E730B069503F} that:
\begin{quote}
\it The problem of classifying Borel subgroups of $\Aut(X)$ up to conjugacy is reduced to that of classifying the maximal elements of $\DJL(B)$ up to conjugacy.
\end{quote}
We now make this reduction completely explicit.
Referring to diagram \eqref{A3583570-EE9B-4EE9-883D-05372F0F5495}, define the map
$\rho = \Lie \circ \rho_X = \rho_B \circ \Lie : \Borel(X) \to \max\LieLN_\bk(B)$.
Since $\max\LieLN_\bk(B) = \max\DJL(B)$ by Proposition \ref{6D6E22A4-236A-4B28-96D9-86157812D34A},
the map $\rho$ is described as follows:
$$
\rho : \Borel(X) \to \max\DJL(B), \qquad \Beul \mapsto \Lie(\urad{\Beul}) = \Lie(\Beul) \cap \lnd(B) .
$$
Theorem \ref{0CA62AB7-E680-4831-972F-E730B069503F} implies that $\rho$ is bijective,
and that two elements of $\Borel(X)$ are conjugate if and only if their images by $\rho$ are conjugate.
Proposition \ref{38182F16-1D38-42ED-9BA4-C7CB5326711E} allows the following description of $\rho^{-1}$: 
\begin{equation} \label {D26DB24A3725450D} 
\begin{minipage}{.88\textwidth}
Given $\Lscr \in \max\DJL(B)$, choose any $\ggoth \in \Grad(\Lscr)$ of maximal rank (i.e., $\rank(\ggoth) = \grank(\Lscr)$)
and define the subgroup $\Beul =  \langle T_\ggoth , \EXP(\Lscr) \rangle_{\text{\rm gp}}$ of $\Aut(X)$.
Then $\Beul \in \Borel(X)$ and $\rho(\Beul) = \Lscr$. Moreover, $\Beul =  T_\ggoth \ltimes \EXP(\Lscr)$ and $\Lie(\Beul) = L(\ggoth) + \Lscr$.
\end{minipage}
\end{equation}
\end{noname}

The procedure in \eqref{D26DB24A3725450D} is also useful outside the classification problem:
it lets us construct Borel subgroups explicitly, as we now illustrate with Nagata's automorphism.

\begin{noname} \label {6A24CF4A9C6743B6} 
{\bf Nagata's automorphism. }
If $f = xz + y^2 \in \bk[x,y,z] = \kk3$ and $\Delta = z \partial_y - 2y \partial_x \in \lnd(\bk[x,y,z])$
then $\nu := \EXP(f\Delta) \in \Aut(\aff^3)$ is Nagata's automorphism.  We claim:
\begin{equation} \label {511387D7208441BA} 
\text{\it There exists a unique Borel subgroup of $\Aut(\aff^3)$ that contains Nagata's automorphism.}
\end{equation}
Indeed, Lemma \ref{DJLNagata1} shows that $\Lscr := \Lscr(\bk[z], \ker(\Delta))$ is the unique maximal element of $\DJL(B)$ that contains $f\Delta$.
It follows from Lemma \ref{unipotentLieLNDiscussion}\ref{4BFE0092-4242-48D7-A3D3-0C484D4BC2B5}
 that $U := \EXP(\Lscr)$ is the unique maximal nested unipotent subgroup of $\Aut(\aff^3)$ that contains $\nu$.
By Theorem \ref{0CA62AB7-E680-4831-972F-E730B069503F}, if $\Beul \in \Borel(\aff^3)$ contains $\nu$
then $\urad{\Beul}$ is a  maximal nested unipotent subgroup of $\Aut(\aff^3)$ that contains $\nu$, so $\urad{\Beul} = U$.
By bijectivity of $\rho_X$ in diagram \eqref{A3583570-EE9B-4EE9-883D-05372F0F5495},
it follows that there is a unique  $\Beul \in \Borel(\aff^3)$ that contains $\nu$, which proves \eqref{511387D7208441BA}.

Let us describe $\Beul$ explicitly.
It is clear that $\rho(\Beul) = \Lscr$, where $\rho : \Borel(X) \to \max\DJL(B)$ and its inverse are explicitly described in 
paragraph \ref{347219A2-47B7-44CA-B8EF-D539653C1734}.
Let $\ggoth$ be the $\Integ^2$-grading of $\bk[x,y,z]$ such that $x$, $y$ and $z$ are $\ggoth$-homogeneous of degrees $(1,1)$, $(1,0)$ and $(1,-1)$ respectively.
Then $\ggoth \in \Grad(\Lscr)$ and $\rank(\ggoth) = 2 = \grank(\Lscr)$, so 
\eqref{D26DB24A3725450D} implies that $\Beul = \rho^{-1}(\Lscr) = T_\ggoth \ltimes \EXP(\Lscr)$.
We have $T_\ggoth = \setspec{ \Spec(\Theta_{s,t}) }{ (s,t) \in \bk^* \times \bk^* }$ where for each $(s,t) \in \bk^* \times \bk^*$ we define
$\Theta_{s,t}$ to be the $\bk$-automorphism of $\bk[x,y,z]$ that sends $x$, $y$ and $z$ to $st x$, $sy$ and $(s/t) z$ respectively.
Let us argue that
\begin{equation} \label {CACA338559EB4A23} 
\Lscr = L(\partial_x, \Delta) = \bk[z] \partial_x + \bk[z,f]\Delta .
\end{equation}
First note that the kernel sequence of the admissible sequence $(\partial_x, \Delta)$ is equal to $(\bk[z], \ker(\Delta)) = \kerseq(\Lscr)$,
so  $\Lscr = \Lscr(\partial_x, \Delta)$.
Arguing by contradiction, assume that $\Lscr(\partial_x, \Delta) \neq L(\partial_x, \Delta)$;
then there exist $\delta \in \Lscr(\partial_x, \Delta) \setminus L(\partial_x, \Delta)$
and a prime element $p$ of $\bk[z]$ such that $p\delta \in L(\partial_x, \Delta)$.
Then $p\delta = a_1 \partial_x + a_2 \Delta$ with $a_1 \in \bk[z]$ and $a_2 \in \ker\Delta$.
Evaluating both sides of $p\delta = a_1 \partial_x + a_2 \Delta$ at $y$ shows that $p \mid a_2 z$.
If $p$ divides $a_2$ then it also divides $a_1$, so $\delta \in L(\partial_x, \Delta)$, a contradiction; so $p \nmid a_2$ and hence $p \mid z$.
We might as well assume that $p=z$; then $z\delta = a_1 \partial_x + a_2 \Delta$.
Evaluation at $x$ gives $z \delta(x) = a_1 -2y a_2$, so if $\pi: \bk[x,y,z] \to \bk[x,y]$ is the map $P(x,y,z) \mapsto P(x,y,0)$,
then $\bk \ni \pi(a_1) = 2y\pi(a_2) \in y\pi( \ker\Delta ) = y\pi( \bk[z,f] ) = y\bk[y^2]$.
As $y\bk[y^2] \cap \bk = \{0\}$, we obtain $\pi(a_1) = 0 = \pi(a_2)$, so $z$ divides both $a_1$ and $a_2$, which contradicts $\delta \notin L(\partial_x, \Delta)$.
This proves \eqref{CACA338559EB4A23}.
It follows that\footnote{As in the introduction, we use the fact that 
if $L = L_1 \ltimes L_2$ are Lie algebras consisting of locally nilpotent derivations then $\EXP(L) = \EXP(L_1) \ltimes \EXP(L_2)$.
The proof uses the Baker-Campbell-Hausdorff Theorem and related results.}
$$
\EXP(\Lscr) = \EXP\big( \bk[z] \partial_x + \bk[z,f]\Delta \big) = \EXP\big( \bk[z] \partial_x \big) \ltimes \EXP\big( \bk[z,f]\Delta \big) .
$$
So the unique Borel subgroup $\Beul$ of $\Aut(\aff^3)$ that contains $\nu$ is:
$$
\Beul = \setspec{ \Spec(\Theta_{s,t}) }{ (s,t) \in \bk^* \times \bk^* } \ \ltimes \ \bigg( \EXP\big( \bk[z] \partial_x \big) \ltimes \EXP\big( \bk[z,f]\Delta \big) \bigg) .
$$
\end{noname}

\medskip

We conclude this section with a discussion of homogeneous automorphisms.
Let $\bk$, $B$ and $X$ be subject to the global assumptions stated at the beginning of the section.

\begin{definition} \label {4107DDFB-7A69-45EB-9AA3-D267BC7A0DE7}
Let $\ggoth$ be a grading of $B$ by some abelian group.
We say that $\sigma \in \Aut_\bk(B)$ is {\it $\ggoth$-homogeneous\/} if, for each $b \in B$,
$\sigma(b)$ is $\ggoth$-homogeneous if and only if $b$ is $\ggoth$-homogeneous.
Let $\Aut_\bk(B,\ggoth)$ denote the group of $\ggoth$-homogeneous elements of $\Aut_\bk(B)$.

Write $\ggoth = (B_i)_{i \in G}$ and consider the submonoid $\bbM = \setspec{ i \in G }{ B_i \neq 0 }$ of $G$.
If $\sigma \in \Aut_\bk(B,\ggoth)$ then there exists a unique automorphism $f_\sigma$ of the monoid $\bbM$
such that $\sigma(B_i) \subseteq B_{f_\sigma(i)}$ for all $i \in \bbM$.
If $f_\sigma$ is the identity map of $\bbM$, we say that $\sigma$ is {\it degree-preserving.}
The map $\Aut_\bk(B,\ggoth) \to \Aut(\bbM)$, $\sigma \mapsto f_\sigma$, is a group homomorphism:
$f_{\sigma_1 \circ \sigma_2} = f_{\sigma_1} \circ f_{\sigma_2}$ for all $\sigma_1,\sigma_2 \in \Aut_\bk(B,\ggoth)$.
Its kernel is the group of degree-preserving $\ggoth$-homogeneous automorphisms of $B$. 

Observe that if $\ggoth$ is an $\Nat$-grading then every $\ggoth$-homogeneous automorphism of $B$ is degree-preserving,
because every submonoid of $\Nat$ has a trivial automorphism group.
\end{definition}

\begin{lemma}  \label {B3A5B9B6-6F11-4CF9-80CB-3282C3B12873}
Let $\ggoth$ be a grading of $B$ by some abelian group.
An automorphism $\sigma \in \Aut_\bk(B)$ is $\ggoth$-homogeneous if and only if $\sigma \, \bbT_\ggoth \, \sigma^{-1} = \bbT_\ggoth$.
\end{lemma}

\begin{proof}
This is folklore, but we include a proof for the reader's convenience.
Let $\ggoth = (B_i)_{i \in G}$ and $\bbM = \setspec{ i \in G }{ B_i \neq 0 }$.
Suppose that $\sigma$ is $\ggoth$-homogeneous and let $f = f_\sigma \in \Aut(\bbM)$ (notation as in \ref{4107DDFB-7A69-45EB-9AA3-D267BC7A0DE7}).
Consider an element $\Theta_\ggoth^\chi$ of $\bbT_\ggoth$ (where $\chi \in \Hom( \bbM , \bk^* )$).
Since $f \in \Aut(\bbM)$ and $\chi \in \Hom( \bbM , \bk^* )$, we have $\chi \circ f^{-1} \in \Hom( \bbM , \bk^* )$
and consequently $\Theta_\ggoth^{\chi \circ f^{-1}} \in \bbT_\ggoth$.
If $i \in \bbM$ and $h \in B_i$ then $\sigma^{-1}(h) \in B_{f^{-1}(i)}$, so
$$
\big( \sigma \Theta_\ggoth^\chi \sigma^{-1} \big)(h)
= \sigma\big[ \chi( f^{-1}(i) ) \sigma^{-1}(h)  \big]
= \chi( f^{-1}(i) ) \sigma(\sigma^{-1}(h))
= \chi( f^{-1}(i) ) h
= \Theta_\ggoth^{\chi \circ f^{-1}} (h) .
$$
This shows that $\sigma \Theta_\ggoth^\chi \sigma^{-1} = \Theta_\ggoth^{\chi \circ f^{-1}} \in \bbT_\ggoth$,
and it follows that $\sigma \bbT_\ggoth \sigma^{-1} \subseteq \bbT_\ggoth$.
Since $\sigma^{-1}$ is $\ggoth$-homogeneous, we obtain $\sigma^{-1} \bbT_\ggoth \sigma \subseteq \bbT_\ggoth$
and hence $\sigma \bbT_\ggoth \sigma^{-1} = \bbT_\ggoth$.

To prove the converse, it suffices to show that
\begin{equation}  \label {614668F2-5EE0-4B7E-982A-B82329FD395D}
\text{if $\bbT_\ggoth \subseteq \sigma \bbT_\ggoth \sigma^{-1}$ then $\sigma$ sends $\ggoth$-homogeneous elements to $\ggoth$-homogeneous elements.}
\end{equation}
We prove \eqref{614668F2-5EE0-4B7E-982A-B82329FD395D} by contradiction:
suppose that $\bbT_\ggoth \subseteq \sigma \bbT_\ggoth \sigma^{-1}$ and that there exist $d \in \bbM$ and $h \in B_d \setminus \{0\}$
such that $\sigma(h)$ is not $\ggoth$-homogeneous.
Write $\sigma(h) = \sum_{i \in S} w_i$ where $S$ is a finite subset of $\bbM$ and $w_i \in B_i \setminus \{0\}$ for each $i \in S$,
and choose $i_1,i_2 \in S$ with $i_1 \neq i_2$.

Let $\chi \in \Hom(\bbM,\bk^*)$.
Then $\Theta_\ggoth^\chi \in \bbT_\ggoth \subseteq \sigma \bbT_\ggoth \sigma^{-1}$, so there exists $\chi' \in \Hom( \bbM , \bk^* )$
such that $\Theta_\ggoth^\chi =  \sigma \Theta_\ggoth^{\chi'} \sigma^{-1}$, i.e., $\Theta_\ggoth^\chi \sigma =  \sigma \Theta_\ggoth^{\chi'}$.
Since 
$( \Theta_\ggoth^\chi \sigma )(h) = \Theta_\ggoth^\chi ( \sum_{i \in S} w_i ) = \sum_{i \in S} \chi(i) w_i$
is equal to 
$( \sigma \Theta_\ggoth^{\chi'} )(h) = \sigma( \chi'(d) h ) = \chi'(d) \sigma( h ) = \chi'(d) \sum_{i \in S} w_i$,
it follows that $\chi(i) = \chi'(d)$ for all $i \in S$, so $\chi(i_1) = \chi(i_2)$.
We proved that  $\chi(i_1) = \chi(i_2)$ for all $\chi \in \Hom(\bbM,\bk^*)$.
If $\langle\bbM\rangle$ denotes the subgroup of $G$ generated by $\bbM$ then each $\phi \in \Hom(\langle\bbM\rangle,\bk^*)$
satisfies $\phi|_\bbM \in \Hom(\bbM,\bk^*)$ and hence $\phi(i_1)=\phi(i_2)$.
So we have $0 \neq i_1-i_2 \in \bigcap_{\phi \in \Hom(\langle\bbM\rangle,\bk^*)} \ker(\phi)$.
However, the fact that $\bk$ is algebraically closed and of characteristic zero implies that 
$\textstyle\bigcap_{ \phi \in \Hom(W,\bk^*) } \ker(\phi) \, = \, \{0\}$ for any abelian group $(W,+)$.
This contradiction proves \eqref{614668F2-5EE0-4B7E-982A-B82329FD395D}, and it follows that the
condition $\sigma \bbT_\ggoth \sigma^{-1} = \bbT_\ggoth$ implies that $\sigma$ is $\ggoth$-homogeneous.
\end{proof}

\begin{proposition}  \label{28D787E0-CB1E-413E-A9E5-CA9A302DDC55}
Suppose that $\Lscr, \Lscr'$ are maximal elements of $\LieLN_\bk(B)$ that are conjugate, and let  $r = \grank(\Lscr) = \grank(\Lscr')$.
If $\ggoth$ is a $\Integ^n$-grading of $B$ of rank $r$ (with $n \in \Nat$ arbitrary) such that both $\Lscr$ and $\Lscr'$ are $\ggoth$-homogeneous,
then there exists a $\ggoth$-homogeneous $\theta \in \Aut_\bk(B)$ such that $\theta \Lscr \theta^{-1} = \Lscr'$.
\end{proposition}

\begin{proof}
Let $\beta \in \Aut_\bk(B)$ be such that $\beta \Lscr \beta^{-1} = \Lscr'$.
Write $\ggoth = ( B_i )_{ i \in \Integ^n }$ and define $\hgoth = ( \beta^{-1}(B_i) )_{i \in \Integ^n}$;
then $\hgoth$ is a $\Integ^n$-grading of $B$ of rank $r$ such that $\Lscr$ is $\hgoth$-homogeneous, and:
\begin{equation}  \label {F41A70CA-CBDC-4502-A4B3-BD1337FC1424}
\beta \, \bbT_\hgoth \, \beta^{-1} = \bbT_\ggoth .
\end{equation}
By Proposition \ref{38182F16-1D38-42ED-9BA4-C7CB5326711E}\ref{8E3DCA2F-E04D-4436-8149-693BB7FDEAE1} there exists $D \in \Lscr$
such that the automorphism $\epsilon = \exp(D) \in \Aut_{\text{\rm u}}(B)$ satisfies 
\begin{equation}  \label {3703F28D-919C-4207-BDE1-B7193EEE8833}
\epsilon \, \bbT_\ggoth \, \epsilon^{-1} = \bbT_\hgoth .
\end{equation}
Let $\theta = \beta \circ \epsilon \in \Aut_\bk(B)$; equalities \eqref{F41A70CA-CBDC-4502-A4B3-BD1337FC1424} and \eqref{3703F28D-919C-4207-BDE1-B7193EEE8833} give
$\theta \, \bbT_\ggoth \, \theta^{-1} = \bbT_\ggoth$, so $\theta$ is $\ggoth$-homogeneous by Lemma \ref{B3A5B9B6-6F11-4CF9-80CB-3282C3B12873}.
Since $\epsilon \, \Lscr \, \epsilon^{-1} = \Lscr$, we obtain
$\theta \, \Lscr \, \theta^{-1}
= (\beta \circ \epsilon) \, \Lscr \, (\beta \circ \epsilon)^{-1}
= \beta ( \epsilon \, \Lscr \, \epsilon^{-1} ) \beta^{-1}
= \beta \, \Lscr \, \beta^{-1}
= \Lscr'$, as claimed.
\end{proof}

\noindent{\bf Acknowledgements.} The authors used the AI assistant Claude Sonnet 5 (Anthropic) for occasional language editing in the Introduction of the paper. All mathematical content, including all definitions, results, and proofs, is the authors' own.

\bibliographystyle{amsalpha}

\providecommand{\bysame}{\leavevmode\hbox to3em{\hrulefill}\thinspace}
\providecommand{\MR}{\relax\ifhmode\unskip\space\fi MR }
\providecommand{\MRhref}[2]{%
  \href{http://www.ams.org/mathscinet-getitem?mr=#1}{#2}
}
\providecommand{\href}[2]{#2}

\end{document}